\documentclass[11pt]{amsart}

\usepackage{amssymb,amsfonts,amsthm,amsmath,enumerate,amscd}
\usepackage{graphicx}
\usepackage{color, colordvi}
\usepackage{times}
\usepackage[english]{babel}
\usepackage{tikz}
\usepackage{tikz-cd}
\usepackage{ulem}

\newtheorem{Lem}{Lemma}[section]

\newtheorem{Def}[Lem]{Definition}
\newtheorem{Teo}[Lem]{Theorem}
\newtheorem{Prop}[Lem]{Proposition}

\newtheorem{remark}[Lem]{Remark}

\theoremstyle{definition} 
\newtheorem{Exam}[Lem]{Example}

\newcommand{\tord}{{\rm tord}}
\newcommand{\itord}{{\rm itord}}
\newcommand{\Lsing}{{\rm Lsing}}

\newcommand{\field}[1]{\mathbb{#1}}

\newcommand{\N}{\field{N}}

\newcommand{\R}{\field{R}}

\newcommand{\F}{\field{F}}
\title[On the minimality of pancake decomposition of surface germs]{On the minimality of pancake decomposition of surface germs}

\author[D. L. Medeiros]{Davi Lopes Medeiros$\dagger$}
\thanks{$\dagger$Researcher supported by the Serrapilheira Institute (grant number Serra -- R-2110-39576) and by FAPESP (grant number FAPESP: 2024/13488-6).}
\address{$\dagger$Departamento de Matem\'atica, Universidade Federal do Cear\'a (UFC), Campus do Pici, Bloco 914, Cep.~60455-760, Fortaleza-Ce, Brasil}
\address{$\dagger$Departamento de Matem\'atica, Instituto de Ci\^encias Matem\'aticas e de Computa\c{c}\~ao (ICMC-USP). Avenida Trabalhador São-carlense, 400, Centro, 13566-590. São Carlos, SP, Brasil}
\email{davi\_lopes90@hotmail.com \& profdavilopes@gmail.com}

\author[E. Silva]{Euripedes da Silva$\star$}
\thanks{$\star$Researcher partially supported by the Serrapilheira Institute (grant number Serra -- R-2110-39576).}
\address{$\star$Departamento de Matem\'atica, Instituto Federal de Educação, Ciência e Tecnologia Cear\'a (IFCE), Av. Parque Central, 1315 - Distrito Industrial I, Maracanaú - CE, 61939-140, Brasil}
\email{euripedes.carvalho@ifce.edu.br}

\author[E. Souza]{Emanoel Souza$\sharp$}
\address{$\sharp$Departamento de Matem\'atica, Universidade Estadual do Cear\'a (UECE), Avenida Dr. Silas Munguba, nº. 1700, bairro Itaperi, Fortaleza/CE, Brasil}
\email{emanoel.souza@uece.br}

\subjclass[2000]{}


\begin{document}

\begin{abstract}
 The abnormal surfaces called snakes and circular snakes, defined in \cite{GabrielovSouza}, are special types of surface germs capturing the outer Lipschitz phenomena relevant to the outer classification problem. We provide algorithms to obtain a minimal pancake decomposition, i.e., where the number of pancakes is minimal, for snakes and circular snakes. We call a pancake decomposition obtained from our algorithm a greedy pancake decomposition. We also prove that greedy pancake decompositions of weakly outer Lipschitz equivalent snakes (or circular snakes) are weakly equivalent, in the sense that there is a weakly outer bi-Lipschitz homeomorphism between the surfaces mapping each greedy pancake to a greedy pancake. This implies that such minimal decompositions are also canonical up to weakly outer bi-Lipschitz equivalence.
\end{abstract}

\maketitle

\section{Introduction}
For the past two decades, Lipschitz geometry of singularities has garnered significant interest as a natural method for classifying singularities, striking a balance between their bi-regular (too fine) and topological (too coarse) equivalences. Notably, the finiteness theorems presented in \cite{Mostowski85Dissertation} and \cite{Parusinski94} indicate the potential for an effective bi-Lipschitz classification of definable real surface germs. 

As demonstrated in \cite{LBirbMosto2000NormalEmbedding}, every singular germ (of a semialgebraic set) $X$ admits two metrics from its surrounding space: the inner metric where the distance between two points of $X$ is the length of the shortest path connecting them inside $X$, and the outer metric with the distance between two points of $X$ being just their distance in the ambient space. This defines two classification problems: equivalence up to bi-Lipschitz homeomorphisms with respect to the inner and outer metrics, or simply, the ``inner classification problem'' and the ``outer classification problem''. 

When $X$ is a surface germ (definable in a polynomially bounded o-minimal structure over the reals), the problems above have different outcomes. The inner classification problem was solved by Birbrair in \cite{birbrair1999local} and \cite{birbrairOminimal} but the outer classification problem remains open. Birbrair showed that any semialgebraic surface germ with a link homeomorphic to a line segment is bi-Lipschitz equivalent with respect to the inner metric to the standard $\beta$-H\"older triangle $T_{\beta}=\{ (x,y)\in \R^2 \mid 0\le x\le 1,\;0\le y\le x^\beta\}$. Moreover, any semialgebraic surface with an isolated singularity and connected link is bi-Lipschitz equivalent to a $\beta$-horn $H_\beta = \{(x,y,z)\in \R^3 \mid z\ge 0, \; x^2+y^2=z^{2\beta}\}$. Later developments in the direction of a solution for the outer metric classification of surface germs were given in \cite{PizzaPaper2017}, \cite{GabrielovSouza} and ~\cite{normalpairs}. In \cite{GabrielovSouza}, Gabrielov and Souza identified basic ``abnormal'' parts of a surface germ, called snakes, and investigated their geometric and combinatorial properties. Indeed, they showed that any given surface germ is either a circular snake or contains finitely many snakes. 

The Lipschitz normally embedded (LNE for short) singularities are the ones where the inner and outer metrics are equivalent, thus the two classifications are the same in this case. It was proved in \cite{KurdykaOrro97} that any semialgebraic set can be decomposed into the union of finitely many normally embedded semialgebraic sets. Later, in \cite{LBirbMosto2000NormalEmbedding}, Birbrair and Mostowski called this decomposition a ``pancake decomposition'' and used Kurdyka's construction to prove that any given semialgebraic set is inner Lipschitz equivalent to a normally embedded semialgebraic set.

We say that a pancake decomposition of a surface germ is reduced if the union of any two pairs of adjacent (with nonempty intersection outside the origin) pancakes is not LNE. We say that it is minimal if the number of pancakes is minimal. Despite Kurdyka's Theorem showing that a pancake decomposition always exists for any given semialgebraic set, which also proves that a minimal pancake decomposition always exists, it gives no hint on how to obtain such a decomposition, not even for surface germs. Indeed, as shown in Example \ref{Exam: two non-equiv decomp}, it is not always possible to obtain a minimal pancake decomposition from a given pancake decomposition. In this work we provide algorithms to obtain minimal pancake decompositions for especial types of surface germs, the so called snakes and circular snakes, which are fundamental for the outer classification problem, as evidenced by Gabrielov and Souza in \cite{GabrielovSouza}. Any pancake decomposition obtained by such algorithms is called a greedy pancake decomposition.

Anyone aiming the goal of solving the outer classification problem will necessarily need to classify snakes and circular snakes up to outer bi-Lipschitz homeomorphisms. In fact, a weak version of this classification problem, the so called weak classification problem, considering weakly outer bi-Lipschitz homeomorphisms (See Definition \ref{Def:weak equivalence}) instead of outer bi-Lipschitz homeomorphisms, was solved for snakes by Gabrielov and Souza in \cite{GabrielovSouza}, and for circular snakes by Costa, Medeiros and Souza in \cite{circular-snakes}. We also prove in this work that two greedy pancake decompositions of two weakly outer Lipschitz equivalent snakes (or circular snakes) are also weakly equivalent, in the sense that there is a weakly outer bi-Lipschitz homeomorphism between the surfaces mapping each greedy pancake to a greedy pancake (See Definition \ref{Def: equiv of pancake decomp}). This also shows that the greedy decomposition is canonical up to weakly outer bi-Lipschitz equivalence. Otherwise specified we will be using the notion of ``canonical'' along this text having this weak equivalence in mind. It is worth noticing that it follows from the weak classification theorems mentioned in the previous paragraph that the pancake decomposition presented in Proposition 4.56 of \cite{GabrielovSouza} for snakes (respectively, Corollary 3.37 of \cite{circular-snakes} for circular snakes with nodal zones) is canonical, however, it is not necessarily minimal. 

One could establish the notion of outer equivalence of pancake decompositions considering outer bi-Lipschitz homeomorphisms instead of weakly outer bi-Lipschitz ones in Definition \ref{Def: equiv of pancake decomp}. However, as demonstrated in Example \ref{Exam: not outer equiv pancake decomp}, neither the pancake decomposition obtained for snakes in Proposition 4.56 of \cite{GabrielovSouza} (respectively, for circular snakes with nodal zones in Corollary 3.37 of \cite{circular-snakes}) nor the greedy pancake decomposition are  canonical with respect to this outer equivalence. Therefore, the canonicity obtained in this paper for weak equivalence is sharp.

This article is organized as follows. In Section 2 we recall the necessary notions of Lipschitz Geometry related to the paper. It may look quite extensive for an experienced reader, but in order to facilitate the reading, we decided to show most of the main results used instead of just citing them. In section 3 we also recall the definition of pancake decomposition and present some interesting facts about the notions of reduced and minimal decompositions (See Example \ref{Exam: two non-equiv decomp}). Moreover, in this section we introduce the weakly equivalence of pancakes (See Definition \ref{Def: equiv of pancake decomp}) and the minimality problem. Section 4 is devoted to presenting the algorithm that produces the so called greedy pancake decomposition for snakes through their minimal sequences (See Definition \ref{Def: Greedy decomposition in Snakes}) and also prove that they minimal (See Theorem \ref{Teo: minimal pancake decomposition}). In Section 5 we address the necessary adaptations to the previous algorithm so we can determine a greedy pancake decomposition for circular snakes and establish their minimality (See Definition \ref{Def: Greedy decomposition in circ Snakes no nodes} and \ref{Def: Greedy decomposition in circ Snakes with nodes} and Theorems \ref{Teo: minimal-circular-no-nodes} and \ref{Teo:Minimal-pancake-circular-snakes-with-nodes}). In Section 6 we prove that greedy decompositions are also canonical up to weakly outer bi-Lipschitz equivalence (See Propositions \ref{Prop: canonical-snakes}, \ref{Prop: canonical-circular-snakes-with-nodes} and \ref{Prop: canonicity for circular snakes with nodal zones}). Finally, in Section 7, we provide final remarks on the greedy pancake decomposition. We demonstrate that the greedy algorithm fails to provide a minimal pancake decomposition for H\"older triangles in general, although it works for snakes and circular snakes (See Example \ref{Exam: greedy algorithm failing for HT}). We also emphasize the sharpness of the canonicity of the greedy pancake decomposition for snakes and circular snakes by showing that it would not be canonical if we consider outer equivalence instead of weak equivalence (See Example \ref{Exam: not outer equiv pancake decomp}).


We would like to thank Lev Birbrair and Andrei Gabrielov for productive discussions on Lipschitz geometry of surfaces and their interest in the results of this article. We also would like to thank Edson Sampaio for his interest and encouragement regarding this manuscript.

\section{Preliminaries}\label{Section: Preliminaries}

All sets, functions and maps in this paper are assumed to be definable in a polynomially bounded o-minimal structure over $\mathbb{R}$ with the field of exponents $\mathbb{F}$, for example, real semialgebraic or subanalytic (see \cite{LvandenDRIES98} and \cite{DRIES-SPREISSEGGER}). Unless the contrary is explicitly stated, we consider germs at the origin of all sets and maps.

\subsection{Basic concepts in Lipschitz geometry}\label{Subsec: Holder triangles}

We present the necessary nomenclature and requisites in Lipschitz geometry of germs for the proper understanding of this paper. Most of the preliminaries used in the subsection are included with much more details in the survey paper \cite{livro}.

\begin{Def}\label{DEF: inner, outer and normally embedded}
	 Given a germ at the origin of a set  $X\subset \mathbb{R}^{n}$ we can define two metrics on $X$, the \textbf{outer metric}  $d(x,y)=||x-y||$ and the \textbf{inner metric} $d_{i}(x,y)=\inf_{\alpha}\{l(\alpha)\}$, where $l(\alpha)$ is the length of a rectifiable path $\alpha$ from $x$ to $y$ in $X$. Note that such a path $\alpha$ always exists since $X$ is definable. A set $X \subset \R^{n}$ is \textbf{Lipschitz normally embedded} (\textbf{LNE} for short) if the outer and inner metrics are equivalent.
\end{Def}

\begin{remark}
	 The inner metric is not definable, but it is equivalent to a definable metric (see \cite{KurdykaOrro97}), for example, the \textbf{pancake metric} (see \cite{LBirbMosto2000NormalEmbedding}).
\end{remark}

\begin{Def}\label{Def: arc}
		An \textbf{arc} in $\mathbb{R}^{n}$ is a germ at the origin of a mapping $\gamma \colon [0,\epsilon) \longrightarrow \mathbb{R}^{n}$ such that $\gamma(0) = 0$. Unless otherwise specified, we suppose that arcs are parameterized by the distance to the origin, i.e., $||\gamma(t)||=t$. We usually identify an arc $\gamma$ with its image in $\mathbb{R}^{n}$. For a germ at the origin of a set $X$, the set of all arcs $\gamma \subset X$ is denoted by $V(X)$ (known as the \textbf{Valette link of $X$}, see \cite{valette2007link}).
\end{Def}

\begin{Def}\label{DEF: order of tangency}
	 The \textbf{tangency order} of two arcs $\gamma_{1}$ and $\gamma_{2}$ in $V(X)$ (notation $\tord(\gamma_{1},\gamma_{2})$) is the exponent $q$ where $||\gamma_{1}(t) - \gamma_{2}(t)|| = ct^{q} + o(t^{q})$ with $c\neq 0$. By convention, $\operatorname{tord}(\gamma,\gamma)=\infty$. For an arc $\gamma$ and a set of arcs $Z \subset V(X)$, the tangency order of $\gamma$ and $Z$ (notation $\tord(\gamma, Z)$), is the supremum of $\tord(\gamma, \lambda)$ over all arcs $\lambda \in Z$. The tangency order of two sets of arcs $Z$ and $Z'$ (notation $\tord(Z,Z')$) is the supremum of $\tord(\gamma, Z')$ over all arcs $\gamma \in Z$. Similarly, we define the tangency orders in the inner metric of $X$, denoted by $\itord(\gamma_{1},\gamma_{2}),\; \itord(\gamma, Z)$ and $\itord(Z,Z')$.
\end{Def}

\begin{remark}\label{Rem: non-archimedean property}
	 An interesting fact about the tangency order of arcs in $\R^n$ is the so called ``non-archimedean property'' (it first appeared in \cite{B.F.METRIC-THEORY} as ``Isosceles property''): given arcs $\gamma_1,\gamma_2,\gamma_3$ in $\R^n$, we have $$\tord(\gamma_2,\gamma_3) \ge \min(\tord(\gamma_1,\gamma_2),\tord(\gamma_1,\gamma_3)).$$ If $\tord(\gamma_1,\gamma_2)\ne \tord(\gamma_1,\gamma_3)$ then $\tord(\gamma_2,\gamma_3) = \min(\tord(\gamma_1,\gamma_2),\tord(\gamma_1,\gamma_3))$.
\end{remark}

\begin{Def}\label{DEF: standard Holder triangle}
	 For $\beta \in \mathbb{F}$, $\beta \ge 1$, the \textbf{standard $\beta$-H\"older triangle} $T_\beta\subset\R^2$ is the germ at the origin of the set
	\begin{equation*}
		T_\beta = \{(x,y)\in \R^2 \mid 0\le x\le 1, \; 0\le y \le x^\beta\}.
	\end{equation*}
	The curves $\{x\ge 0,\; y=0\}$ and $\{x\ge 0,\; y=x^\beta\}$ are the  \textbf{boundary arcs} of $T_\beta$.

    For $\beta \in \mathbb{F}$, $\beta \ge 1$, the \textbf{standard  $\beta$-horn}  $H_\beta\subset\R^2$ is the germ at the origin of the set
	\begin{equation*}\label{Formula:Standard Holder triangle}
		H_\beta = \{(x,y,z)\in \R^3 \mid z\ge 0, \; x^2+y^2=z^{2\beta}\}.
	\end{equation*}
\end{Def}

\begin{Def}\label{DEF: Holder triangle}
	 A germ at the origin of a set $T \subset \mathbb{R}^{n}$ that is bi-Lipschitz equivalent with respect to the inner metric to the standard $\beta$-H\"older triangle $T_\beta$ is called a \textbf{$\beta$-H\"older triangle} (see \cite{birbrair1999local}).
	The number $\beta \in \mathbb{F}$ is called the \textbf{exponent of $T$} (notation $\beta=\mu(T)$). The arcs $\gamma_{1}$ and $\gamma_{2}$ of $T$ mapped to the boundary arcs of $T_\beta$ by the homeomorphism are the \textbf{boundary arcs of $T$} (notation $T=T(\gamma_{1},\gamma_{2})$). All other arcs of $T$ are \textbf{interior arcs}. The set of interior arcs of $T$ is denoted by $I(T)$.

    A germ at the origin of a set $H \subset \mathbb{R}^{n}$ that is bi-Lipschitz equivalent with respect to the inner metric to the standard $\beta$-horn $H_\beta$ is called a \textbf{$\beta$-horn} (see \cite{birbrair1999local}).
	The number $\beta \in \mathbb{F}$ is called the \textbf{exponent of $H$} (notation $\beta=\mu(H)$).
\end{Def}

\begin{remark}\label{Rem: NE HT condition}
	It was proved in \cite{birbrair1999local} that $\mu(T)$ and $\mu(H)$ are inner bi-Lipschitz invariants. Moreover, it was proved in \cite{birbrair2018arc}, using the Arc Selection Lemma (see Theorem 2.2), that a H\"older triangle $T$ is Lipschitz normally embedded if, and only if, $\tord(\gamma,\gamma')=\itord(\gamma,\gamma')$ for any two arcs $\gamma$ and $\gamma'$ of $T$.
\end{remark}

\begin{Def}\label{DEF: Lipschitz non-singular arc}
	 Let $X$ be a surface (a two-dimensional set). An arc $\gamma \subset X$ is \textbf{Lipschitz non-singular} if there exists a Lipschitz normally embedded H\"older triangle $T \subset X$ such that $\gamma$ is an interior arc of  $T$ and $\gamma \not\subset \overline{X\setminus T}$. Otherwise, $\gamma$ is \textbf{Lipschitz singular}. In particular, any interior arc of a Lipschitz normally embedded H\"older triangle is Lipschitz non-singular. The union of all Lipschitz singular arcs in $X$ is denoted by $\Lsing(X)$.
\end{Def}

\begin{remark}\label{Rem: boundary arcs are Lips sing arcs}
	 It follows from pancake decomposition (see Definition \ref{Def: pancake decomposition}) that a surface $X$ contains finitely many Lipschitz singular arcs. For an interesting example of a Lipschitz singular arc see Example 2.11 of \cite{GabrielovSouza}. Arcs which are boundary arcs of H\"older triangles or self-intersections of the surface are trivial examples of Lipschitz singular arcs.
\end{remark}

\begin{Def}\label{DEF: non-singular Holder triangle}
	 A H\"older triangle $T$ is \textbf{non-singular} if all interior arcs of $T$ are Lipschitz non-singular.
\end{Def}

\begin{Def}\label{Def: generic arc of a surface}
	 Let $X$ be a surface germ with connected link. The \textbf{exponent $\mu(X)$ of $X$} is defined as $\mu(X)=\min\, \itord(\gamma,\gamma')$, where the minimum is taken over all arcs $\gamma,\,\gamma'$ of $X$.
	A surface $X$ with exponent $\beta$ is called a $\beta$-surface. An arc $\gamma \subset X\setminus \Lsing(X)$ is \textbf{generic} if $\itord(\gamma,\gamma') = \mu(X)$ for all arcs $\gamma'\subset \Lsing(X)$. The set of generic arcs of $X$ is denoted by $G(X)$. 
\end{Def}

\begin{remark}\label{Rem: generic arcs of a non-singular HT}
	 If $X=T(\gamma_{1},\gamma_{2})$ is a non-singular $\beta$-H\"older triangle then an arc $\gamma\subset X$ is generic if, and only if, $\itord(\gamma_{1},\gamma) = \itord(\gamma,\gamma_{2}) = \beta$. 
     
     Let $X$ be the standard $\beta$-H\"older triangle for some $\beta\ge 1$ in $\F$. For any $n\in \N$, $n\ge 2$, the arcs $\delta_n(t) = \dfrac{x^\beta}{n}$ are generic arcs of $X$, although the sequence $\{\delta_n\}_{n=2}^{\infty}$ converges to the boundary arc $\{x\ge 0, \; y=0\}$. Similarly, $\{\tilde \delta_n\}_{n=2}^{\infty}$, where $\tilde \delta_n(t) = \left(1 - \dfrac{1}{n}\right)x^\beta$, is a sequence of generic arcs of $X$ converging to the boundary arc $\{x\ge 0, \; y=x^\beta\}$. Since inner tangency orders are preserved by inner bi-Lipschitz homeomorphisms, this argument proves that for any given H\"older triangle $X$ there exists a sequence of its generic arcs converging to one of its boundary arcs.

\end{remark}

\subsection{Zones, abnormal surfaces and snakes}\label{subsec 1.4}

Some of the definitions below were first introduced in \cite{LevMendes2018}, while the definition of snake is given in \cite{GabrielovSouza}. 

\begin{Def}\label{Def: zone}
	A nonempty set of arcs $Z \subset V(X)$ is a \textbf{zone} if, for any two distinct arcs $\gamma_{1}$ and $\gamma_{2}$ in $Z$, there exists a non-singular H\"older triangle $T=T(\gamma_{1},\gamma_{2}) \subset X$ such that $V(T) \subset Z$. If $Z = \{\gamma\}$ then $Z$ is a \textbf{singular zone}.
\end{Def}

\begin{Def}\label{Def: maximal zone in}
	Let $B \subset V(X)$ be a nonempty set. A zone $Z\subset B$ is \textbf{maximal in} $B$ if, for any H\"older triangle $T$ such that $V(T) \subset B$, one has either $Z\cap V(T)=\emptyset$ or $V(T) \subset Z$.
\end{Def}

\begin{remark}
	A zone could be understood as an analog of a connected subset of $V(X)$, and a maximal zone in a set $B$ is an analog of a connected component of $B$.
\end{remark}

\begin{Def}\label{Def:order of zone}
	The \textbf{order} $\mu(Z)$ of a zone $Z$ is defined as the infimum of $\operatorname{tord}(\gamma,\gamma')$ over all arcs $\gamma$ and $\gamma'$ in $Z$. If $Z$ is a singular zone, we define $\mu(Z) = \infty$. A zone $Z$ of order $\beta$ is called a $\beta$-zone.
\end{Def}

\begin{remark}
	The tangency order can be replaced by the inner tangency order in Definition \ref{Def:order of zone}. Note that, for any arc $\gamma \in Z$, $\inf_{\gamma'\in Z}\operatorname{tord}(\gamma,\gamma')=\inf_{\gamma'\in Z}\itord(\gamma,\gamma')=\mu(Z)$. Moreover, differently from Definition \ref{Def: generic arc of a surface}, the order of a zone could not be a minimum (for such an example, see Example 2.46 of \cite{GabrielovSouza}).
\end{remark}

\begin{Def}\label{Def: LNE zone}
	A zone $Z$ is LNE if, for any two arcs $\gamma$ and $\gamma'$ in $Z$, there exists a LNE H\"older triangle $T=T(\gamma,\gamma')$ such that $V(T)\subset Z$.
\end{Def}

\begin{Def}\label{DEF: normal and abnormal arcs and zones}
	A Lipschitz non-singular arc $\gamma$ of a surface germ $X$ is \textbf{abnormal} if there are two LNE H\"older triangles $T\subset X$ and $T'\subset X$ such that $T\cap T' = \gamma$ and $T\cup T'$ is not LNE.
	Otherwise $\gamma$ is \textbf{normal}. A zone is \textbf{abnormal} (resp., \textbf{normal}) if all of its arcs are abnormal (resp., normal). The sets of abnormal and normal arcs of $X$ are denoted by $Abn(X)$ and $Nor(X)$, respectively. A surface germ $X$ is called \textbf{abnormal} if $Abn(X)=G(X)$.
\end{Def}

\begin{Def}\label{Def: maximal abnormal and normal zones}
	Given an abnormal (resp., normal) arc $\gamma \subset X$, the \textbf{maximal abnormal zone} (resp., \textbf{maximal normal zone}) in $V(X)$ containing $\gamma$ is the union of all abnormal (resp., normal) zones in $V(X)$ containing $\gamma$. Alternatively, the maximal abnormal (resp., normal) zone containing $\gamma$ is a maximal zone in $Abn(X)$ (resp., $Nor(X)$) containing $\gamma$.
\end{Def}

\begin{remark}\label{max zones are unique}
	It follows from Definition \ref{DEF: normal and abnormal arcs and zones} that the property of an arc to be abnormal (resp., normal)
	is an outer bi-Lipschitz invariant: if $h:X\to X'$
	is an outer bi-Lipschitz map then $h(\gamma)\subset X'$ is an abnormal (resp., normal) arc for any abnormal (resp., normal) arc $\gamma\subset X$.
	Since the property of an arc to be abnormal (resp., normal) is outer Lipschitz invariant,
	maximal abnormal (resp., normal) zones in $V(X)$ are also outer Lipschitz invariant: if $h:X\to X'$
	is an outer bi-Lipschitz map then $h(Z)\subset V(X')$ is a maximal abnormal (resp., normal) zone for any maximal abnormal (resp., normal) zone $Z\subset V(X)$. Here, $h: V(X) \to V(X') $ denotes the natural action of $h$ on the space of arcs in $X$.
 
\end{remark}

\begin{Def}\label{Def:snake}
	A non-singular $\beta$-H\"older triangle $T$ is called a $\beta$\textbf{-snake} if $T$ is an abnormal surface (see Definition \ref{DEF: normal and abnormal arcs and zones}). Similarly, a non-singular $\beta$-horn $X$ is called a $\beta$\textbf{-circular snake} if $X$ is an abnormal surface.
\end{Def}

\subsection{Segments, nodal zones and weakly outer bi-Lipschitz maps}\label{Subsection: weak equivalence}

In this subsection we summarize the concepts of segments and nodal zones of snakes and circular snakes. We also recall how those invariant parts of the Vallete link of an abnormal surface are used to obtain the classification theorems summarized in Theorem \ref{Teo:weak equivalence}. Roughly speaking, the segments are zones where we have space to move an arc to both sides without changing the Lipschitz contact (this notion is translated through the multiplicity of this arc) of the surface with itself, while the nodal zones are the zones where this cannot be done. Since segments and nodal zones are canonical up outer bi-Lipschitz homeomorphisms, given an orientation we could use them to associate a word with a snake (resp., a circular snake). This word is the combinatorial object used in Theorem \ref{Teo:weak equivalence}.


\begin{Def}\label{Def:horn-neighbourhood}
	 Let $X$ be a surface and $\gamma \subset X$ an arc. For $a>0$ and $1\leq\alpha\in \mathbb{F}$, the \textbf{$(a,\alpha)$-horn neighborhood of $\gamma$ in $X$} is defined as follows: 
 $$
 \mathcal{H}X_{a,\alpha}(\gamma) = \bigcup_{0\le t \ll1} X\cap S(0,t)\cap \overline{B}(\gamma(t),at^{\alpha}),
 $$
	where $S(0,t)=\{x \in \mathbb{R}^{n}\mid ||x||=t\}$ and $\overline{B}(y,R) = \{x \in \mathbb{R}^{n}\mid ||x - y||\le R\}$.
\end{Def}

\begin{Def}\label{Def of multiplicity}
	 If $X$ is a $\beta$-snake (or circular $\beta$-snake) and $\gamma$ an arc in $X$, the \textbf{multiplicity} of $\gamma$, denoted by $m_{X}(\gamma)$ (or just $m(\gamma)$, when $X$ is understood), is defined as the number of connected components of the set $\mathcal{H}X_{a,\beta}(\gamma)\setminus \{0\}$, for every $a>0$ small enough.
\end{Def}

\begin{Def}\label{constant zone}
    Let $X$ be a $\beta$-snake (or circular $\beta$-snake) and $Z\subset V(X)$ a zone. We say that $Z$ is a \textbf{constant zone of multiplicity $q$} if all arcs in $Z$ have the same multiplicity $q$. We say that $\gamma \in V(X)$ is a \textbf{segment arc} if there exists a $\beta$-H\"older triangle $T \subset X$ such that $\gamma$ is a generic arc of $T$ and $V(T)$ is a constant zone. Otherwise $\gamma$ is a \textbf{nodal arc}. We denote the set of segment arcs and the set of nodal arcs in $X$ by $S(X)$ and $N(X)$, respectively. A \textbf{segment of $X$} is a maximal zone in $S(X)$. A \textbf{nodal zone} of $X$ is a maximal zone in $N(X)$. We write $Seg_{\gamma}$ or $Nod_{\gamma}$ for a segment or a nodal zone containing an arc $\gamma$.
\end{Def}

\begin{remark}\label{Rem:words}
	It was proved in \cite{GabrielovSouza} (resp., in \cite{circular-snakes}) that if $X$ is a snake (resp., circular snake) then its Valette link can be decomposed into finitely many disjoint segments and nodal zones. For snakes the segments and nodal zones are always LNE, while for circular snakes we may have segments not LNE, althoug this only happens for the ones without nodal zones. 
    
    Using this decomposition the authors in \cite{GabrielovSouza} (resp., in \cite{circular-snakes}) created a combinatorial object, $W(X)$ (resp., $[[W_N(X)]]$, where $N$ is a given nodal zone of $X$), associated with such snake (resp., circular snake) $X$. More specifically, $W(X)$ (resp., $W_N(X)$ is a word (resp., circular word), in some alphabet, say $\{x_1, x_2,\ldots \}$, satisfying two conditions (see Definitions 6.6 in \cite{GabrielovSouza} and 5.9 in \cite{circular-snakes}). Any word satisfying those conditions is called a \textbf{snake name} (resp., \textbf{circular snake name}) and it was proved that, for any snake name $W$ (resp., circular snake name $W_N$) with length $m>3$ (resp., lenght $m\ge 5$), there is a snake (resp., circular snake) $X$ such that $W = W(X)$ (resp., $W=W_N(X)$) (see Theorems 6.23 in \cite{GabrielovSouza} and 7.10 in \cite{circular-snakes}).

    Theorems 6.23 in \cite{GabrielovSouza} and 7.10 in \cite{circular-snakes} are realization theorems. Along this text we will be giving examples of snakes and circular snakes presenting only their links, however, since the word associated with a snake (resp., circular snake) can be easily obtained from its link, the existence of a snake (resp., circular snake) with the given link is guaranteed by those theorems. 

\end{remark}

\begin{Exam}\label{Exam:snake-name-construction}
	In this example we illustrate how to associate a word $W(X)$ with a snake $X = T(\gamma_{1},\gamma_{2})$. Let $X$ be a snake with link as in Figure \ref{Fig. snake word}. Recall that a \textbf{node} of $X$ is the union of nodal zones with tangency orders higher than $\mu(X)$ (see Definition 4.31 of \cite{GabrielovSouza}). First, we choose an orientation for $X$, say from $\gamma_{1}$ to $\gamma_{2}$. Second, we assign letters to the nodes by moving through the link of $X$, respecting this orientation, in a way that the first node encountered is assigned to the first letter of the alphabet and so on, skipping the nodes already assigned. Finally, we obtain the word associated with $X$ by traversing the link again, accordingly to the orientation, adding a letter every time we pass through a node. In this case, in the alphabet $\{x_1, x_2,\ldots \}$, we have $W(X) = [x_1x_2x_1x_3x_2x_3]$.

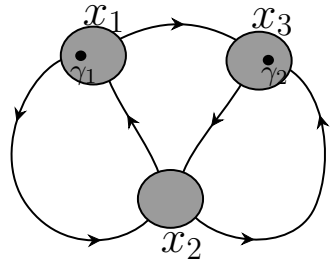
\begin{figure}[h!]
	\centering

\tikzset{every picture/.style={line width=0.75pt}} 

\begin{tikzpicture}[x=0.75pt,y=0.75pt,yscale=-0.8,xscale=0.8]

\draw    (304.45,47.02) .. controls (276.53,44.27) and (244.42,105.07) .. (270.32,141.12) .. controls (296.21,177.17) and (345.28,164.6) .. (355.14,138.83) .. controls (365,113.06) and (297.73,52.53) .. (326.18,37.84) .. controls (354.62,23.15) and (379.45,23.61) .. (406.34,39.22) .. controls (433.24,54.82) and (352.04,112.66) .. (369.1,140.2) .. controls (386.17,167.75) and (438.41,172.34) .. (451.86,142.96) .. controls (465.3,113.58) and (459.62,54.82) .. (419.79,49.77) ;
\draw [shift={(264.14,87.91)}, rotate = 288.06] [fill={rgb, 255:red, 0; green, 0; blue, 0 }  ][line width=0.08]  [draw opacity=0] (7.14,-3.43) -- (0,0) -- (7.14,3.43) -- (4.74,0) -- cycle    ;
\draw [shift={(316.72,163.34)}, rotate = 177.46] [fill={rgb, 255:red, 0; green, 0; blue, 0 }  ][line width=0.08]  [draw opacity=0] (7.14,-3.43) -- (0,0) -- (7.14,3.43) -- (4.74,0) -- cycle    ;
\draw [shift={(334.05,84.79)}, rotate = 59.04] [fill={rgb, 255:red, 0; green, 0; blue, 0 }  ][line width=0.08]  [draw opacity=0] (7.14,-3.43) -- (0,0) -- (7.14,3.43) -- (4.74,0) -- cycle    ;
\draw [shift={(369.2,27.25)}, rotate = 181.83] [fill={rgb, 255:red, 0; green, 0; blue, 0 }  ][line width=0.08]  [draw opacity=0] (7.14,-3.43) -- (0,0) -- (7.14,3.43) -- (4.74,0) -- cycle    ;
\draw [shift={(386.96,91.23)}, rotate = 306.52] [fill={rgb, 255:red, 0; green, 0; blue, 0 }  ][line width=0.08]  [draw opacity=0] (7.14,-3.43) -- (0,0) -- (7.14,3.43) -- (4.74,0) -- cycle    ;
\draw [shift={(413.36,162.99)}, rotate = 183.11] [fill={rgb, 255:red, 0; green, 0; blue, 0 }  ][line width=0.08]  [draw opacity=0] (7.14,-3.43) -- (0,0) -- (7.14,3.43) -- (4.74,0) -- cycle    ;
\draw [shift={(455.2,84.98)}, rotate = 75.41] [fill={rgb, 255:red, 0; green, 0; blue, 0 }  ][line width=0.08]  [draw opacity=0] (7.14,-3.43) -- (0,0) -- (7.14,3.43) -- (4.74,0) -- cycle    ;
\draw  [draw opacity=0][fill={rgb, 255:red, 155; green, 155; blue, 155 }  ,fill opacity=0.47 ] (291.83,46.91) .. controls (291.83,36.83) and (301.03,28.66) .. (312.37,28.66) .. controls (323.71,28.66) and (332.9,36.83) .. (332.9,46.91) .. controls (332.9,56.99) and (323.71,65.16) .. (312.37,65.16) .. controls (301.03,65.16) and (291.83,56.99) .. (291.83,46.91) -- cycle ;
\draw  [fill={rgb, 255:red, 0; green, 0; blue, 0 }  ,fill opacity=1 ] (301.56,47.02) .. controls (301.56,45.54) and (302.85,44.34) .. (304.45,44.34) .. controls (306.05,44.34) and (307.35,45.54) .. (307.35,47.02) .. controls (307.35,48.5) and (306.05,49.7) .. (304.45,49.7) .. controls (302.85,49.7) and (301.56,48.5) .. (301.56,47.02) -- cycle ;
\draw  [draw opacity=0][fill={rgb, 255:red, 155; green, 155; blue, 155 }  ,fill opacity=0.47 ] (340.45,137.34) .. controls (340.45,127.26) and (349.64,119.09) .. (360.98,119.09) .. controls (372.32,119.09) and (381.52,127.26) .. (381.52,137.34) .. controls (381.52,147.42) and (372.32,155.59) .. (360.98,155.59) .. controls (349.64,155.59) and (340.45,147.42) .. (340.45,137.34) -- cycle ;
\draw  [draw opacity=0][fill={rgb, 255:red, 155; green, 155; blue, 155 }  ,fill opacity=0.47 ] (396.31,50.12) .. controls (396.31,40.04) and (405.5,31.87) .. (416.84,31.87) .. controls (428.18,31.87) and (437.38,40.04) .. (437.38,50.12) .. controls (437.38,60.2) and (428.18,68.38) .. (416.84,68.38) .. controls (405.5,68.38) and (396.31,60.2) .. (396.31,50.12) -- cycle ;
\draw  [fill={rgb, 255:red, 0; green, 0; blue, 0 }  ,fill opacity=1 ] (419.79,49.77) .. controls (419.79,48.29) and (421.09,47.09) .. (422.69,47.09) .. controls (424.29,47.09) and (425.59,48.29) .. (425.59,49.77) .. controls (425.59,51.26) and (424.29,52.46) .. (422.69,52.46) .. controls (421.09,52.46) and (419.79,51.26) .. (419.79,49.77) -- cycle ;

\draw (295.81,51.66) node [anchor=north west][inner sep=0.75pt]    {$\gamma _{1}$};
\draw (416.32,52.49) node [anchor=north west][inner sep=0.75pt]    {$\gamma _{2}$};
\draw (303.94,13.56) node [anchor=north west][inner sep=0.75pt]  [font=\LARGE]  {$x_{1}$};
\draw (410.03,14.97) node [anchor=north west][inner sep=0.75pt]  [font=\LARGE]  {$x_{3}$};
\draw (353.07,156.78) node [anchor=north west][inner sep=0.75pt]  [font=\LARGE]  {$x_{2}$};

\end{tikzpicture}
	\caption{Example of a snake with three nodes (each containing exactly two nodal zones) oriented from $\gamma_{1}$ to $\gamma_{2}$. The letters $x_1$, $x_2$ and $x_3$ were assigned to its nodes regarding this orientation. Points inside the shaded disks represent arcs with tangency order higher than the respective surface exponent.}\label{Fig. snake word}
\end{figure}
\end{Exam}

\begin{Exam}
	In this example, we illustrate how to associate a word $W_N(X)$ with a circular snake $X$ with nodal zones. Let $X$ be a circular snake with link as in Figure \ref{Fig. snake word circular}. We choose an orientation for $X$, fix a nodal zone $N$ and assign letters to the nodes by moving through the link of $X$, respecting this orientation, starting at $N$. In this case, using the alphabet $\{x_1, x_2,\ldots \}$, for the orientation $\varepsilon$ in Figure \ref{Fig. snake word circular}a, we obtain $$W_N(X) = [x_1x_2x_3x_4x_5x_6x_4x_2x_3x_5x_1x_6x_1].$$ 
    For the same orientation $\varepsilon$ but now starting from the node $N'$ in Figure \ref{Fig. snake word circular}a, we have $$W_{N'}(X) = [x_4x_2x_3x_5x_1x_6x_1x_2x_3x_4x_5x_6x_4].$$ 
    Finally, for the orientation $-\varepsilon$ and the node $N$ in Figure \ref{Fig. snake word circular}b, $$W_{N}(X) = [x_1x_2x_1x_3x_4x_5x_6x_2x_3x_6x_4x_5x_1].$$ 

\begin{figure}[h!]
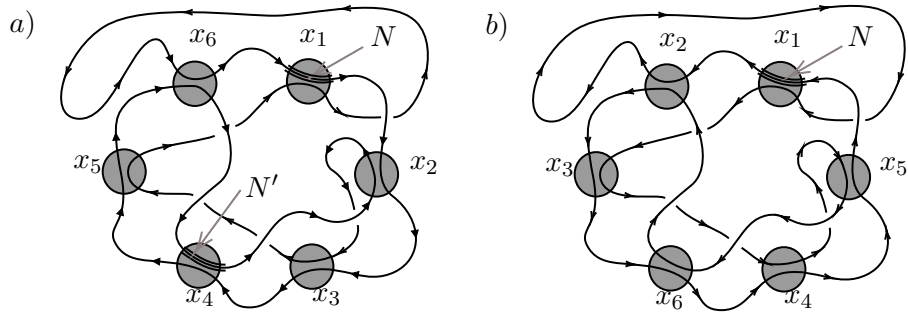

	\centering

\tikzset{every picture/.style={line width=0.75pt}} 


	\caption{Example of a circular snake and its possible orientations. The letters $x_1, \dots, x_6$ are assigned to its six nodes (each one with two nodal zones), starting from the nodal $N$ and accordingly to the given orientation. Points inside the shaded disks represent arcs with tangency order higher than the respective surface exponent.}\label{Fig. snake word circular}
\end{figure}
\end{Exam}


The (circular) words associated with (circular) snakes ignores many geometric properties of them such as the contact orders of arcs in a same node. Still, it is an interesting Lipschitz invariant combinatorial object, since it is possible to create a weaker notion of outer bi-Lipschitz homeomorphism for which those words are preserved.

\begin{Def}\label{Def:weak equivalence}
Let $h:X\to X'$ be an inner bi-Lipschitz map between two $\beta$-surfaces $X$ and $X'$. We say that $h$ is \textbf{weakly outer bi-Lipschitz} when, for any two arcs $\gamma$ and $\gamma'$ of $V(X)$, we have
$$
\tord(h(\gamma),h(\gamma'))>\beta \iff \tord(\gamma,\gamma')>\beta.
$$
If such a homeomorphism exists, we say that $X$ and $X'$ are \textbf{weakly outer Lipschitz equivalent}.
\end{Def}

\begin{Def}\label{Def: clusters}
Let $\mathcal{N}$ and $\mathcal{N}'$ be nodes of a $\beta$-snake (or a circular $\beta$-snake) $X$, and let $\mathcal{S}(\mathcal{N},\mathcal{N}')$ be the (possibly empty) set of all segments of $X$ having adjacent nodal zones in the nodes $\mathcal{N}$ and $\mathcal{N}'$. Two segments $S$ and $S'$ in $\mathcal{S}(\mathcal{N},\mathcal{N}')$ belong to the same \textbf{cluster}
if $\tord(S,S')>\beta$. This defines a \textbf{cluster partition} of $\mathcal{S}(\mathcal{N},\mathcal{N}')$.
The size of each cluster $C$ of this partition is equal to the multiplicity of each segment $S\in C$ (see Definition \ref{Def of multiplicity}).
\end{Def}

We are now ready to describe when two $\beta$-snakes (or circular $\beta$-snakes with nodal zones) are weakly outer Lipschitz equivalent. For $\beta$-snakes, this is Theorem 6.28 of \cite{GabrielovSouza}; for circular $\beta$-snakes with nodal zones, this is Theorem 8.3 of \cite{circular-snakes}. For the special case of circular $\beta$-snakes without nodal zones, Theorem \ref{Teo:Multiplicity without Nodes} is Theorem 8.4 of \cite{circular-snakes}.

\begin{Teo}[Theorem 8.3 in \cite{circular-snakes}]\label{Teo:weak equivalence}

Two $\beta$-snakes $X$ and $X'$ with nodal zones (or circular $\beta$-snakes with nodal zones) are weakly outer Lipschitz equivalent if, and only if, they can be oriented so that
\begin{enumerate}
\item[(i)] They have the same (circular) snake names, the nodes $\mathcal{N}_1,\ldots,\mathcal{N}_n$ of $X$ are in one-to-one correspondence with the nodes
$\mathcal{N}'_1,\ldots,\mathcal{N}'_n$ of $X'$ (with $\mathcal{N}_i$ corresponding to $\mathcal{N}'_i$, for each $i$), and the nodal zones $N_1,\ldots,N_m$ of $X$ are in one-to-one correspondence with the nodal zones $N'_1,\ldots,N'_m$ of $X'$ (with $N_i$ corresponding to $N'_i$, for each $i$);
\item[(ii)] For any two nodes $\mathcal{N}_j$ and $\mathcal{N}_k$ of $X$, and the corresponding nodes $\mathcal{N}'_j$ and $\mathcal{N}'_k$ of $X'$,
each cluster of the cluster partition of the set $\mathcal{S}(\mathcal{N}'_j,\mathcal{N}'_k)$ (see Definition \ref{Def: clusters}) consists
of the segments of $X'$ corresponding to the segments of $X$ contained in a cluster of the cluster partition of the set $\mathcal{S}(\mathcal{N}_j,\mathcal{N}_k)$. 
\end{enumerate}
\end{Teo}

\begin{Teo}[Theorem 8.4 in \cite{circular-snakes} ]\label{Teo:Multiplicity without Nodes}
Let $X$ and $X'$ be two circular $\beta$-snakes without nodal zones. Then, $X$ and $X'$ are weakly outer Lipschitz equivalent if, and only if, $X$ and $X'$ have the same multiplicity.
\end{Teo}

\section{Pancake decomposition of surfaces}\label{section:pancake decomposition}

In this section, we give the definitions of a pancake decomposition and the corresponding concepts of minimal and reduced decompositions. We present the definition of weakly bi-Lipschitz equivalence of pancake decompositions aiming to obtain later, in Section \ref{Sec: Weakly canonicity}, a minimal decomposition which is canonical with respect to this equivalence. We also give an example showing that not all reduced decompositions are minimal, which explicitly shows the need for a more sophisticated algorithm to find minimal pancake decompositions in a canonical way.

\begin{Def}\label{Def: pancake decomposition}
	 Let $X\subset \mathbb{R}^{n}$ be the germ at the origin of a closed set. A \textbf{pancake decomposition of $X$} is a finite collection of closed LNE subsets $X_{k}$ of $X$ with connected links, called \textbf{pancakes}, such that $X=\bigcup X_{k}$ and $$\dim(X_{j}\cap X_{k}) < \min(\dim(X_{j}),\dim(X_{k}))\quad\text{for all}\; j,k.$$
\end{Def}

\begin{remark}\label{Rem: existence of pancake decomp}
	 The term ``pancake'' was introduced in \cite{LBirbMosto2000NormalEmbedding}, but this notion first appeared (with a different name) in \cite{kurdyka1992subanalytic} and \cite{KurdykaOrro97}, where the existence of such a decomposition was established.
\end{remark}

\begin{remark}\label{Rem:pancake of holder triangle is holder triangle}
	 If $X$ is a H\"older triangle and $\{X_k\}_{i=1}^p$ is its pancake decomposition, then each pancake $X_k$ is also a H\"older triangle. Moreover, if $X$ is a non LNE surface germ and has circular link, then $p>1$ and each pancake is also a H\"older triangle.  
\end{remark}

\begin{Def}\label{Def: reduced pancak decomp}
	 A pancake decomposition $\{X_{k}\}_{i=1}^{p}$ of a set $X$ is \textbf{reduced} if the union of any two adjacent pancakes $X_{j}$ and $X_{k}$ (such that $X_{j}\cap X_{k}\ne \{0\}$) is not LNE. We also say that $\{X_{k}\}_{i=1}^{p}$ is \textbf{minimal} if $p$ is minimal among all pancake decompositions of $X$.
\end{Def}

\begin{remark}\label{Rem: reduced PD always exists}
	 When the union of two adjacent pancakes is LNE, they can be replaced by their union, reducing the number of pancakes. Thus, a reduced pancake decomposition of a set $X$ always exists. Moreover, every minimal pancake decomposition of $X$ is reduced, but the converse is false, as seen in Example \ref{Exam: two non-equiv decomp}. This example also shows that it is not possible, in general, to obtain a minimal pancake decomposition from a reduced one.
\end{remark}

Since every set $X$ can be decomposed into a finite number of pancakes (see \cite{kurdyka1992subanalytic}, \cite{KurdykaOrro97} and \cite{LBirbMosto2000NormalEmbedding}), a minimal pancake decomposition of $X$ always exists. However, one may wonder if there is a constructive way to find such a decomposition, and if such a construction is canonical. For snakes and circular snakes, we give an affirmative answer, but first, we must define weakly bi-Lipschitz equivalence between pancake decompositions.

\begin{Def}\label{Def: equiv of pancake decomp}
    Let $X, \tilde X$ be two surfaces and let $\{X_i\}_{i\in I}$, $\{\tilde X_j\}_{j\in J}$ be pancake decompositions of $X$ and $\tilde X$, respectively. We say that $\{X_i\}_{i\in I}$ and $\{\tilde X_j\}_{j\in J}$ are \textbf{weakly outer bi-Lipschitz equivalent} pancake decompositions for $X$ and $\tilde X$ if there is a bijection $\sigma \colon I \rightarrow J$ and a weakly outer bi-Lipschitz map $h\colon X \rightarrow \tilde X$ such that $h(X_i) = \tilde X_{\sigma(i)}$ for all $i\in I$.
\end{Def}

\begin{Exam}\label{Exam: two non-equiv decomp}
    In Figure \ref{Fig. two non-equiv decompositions} we have the link of an abnormal $\beta$-surface $X$ and the representation of two of its reduced pancake decompositions, denoted by $\{X_j = T(\lambda_{j-1},\lambda_j)\}_{j=1}^2$ and $\{\tilde X_j = T(\tilde \lambda_{j-1},\tilde \lambda_j)\}_{j=1}^3$. Clearly such decompositions are not weakly bi-Lipschitz equivalent (See Definition \ref{Def: equiv of pancake decomp}), since the number of pancakes is different (See Proposition \ref{Prop: canonical-snakes}). This example shows that a reduced decomposition is not necessarily minimal. Moreover, it shows that it is not always possible to obtain a minimal pancake decomposition from a given pancake decomposition by joining LNE adjacent pancakes into a single new pancake, as described in Remark \ref{Rem: reduced PD always exists}.
\end{Exam}

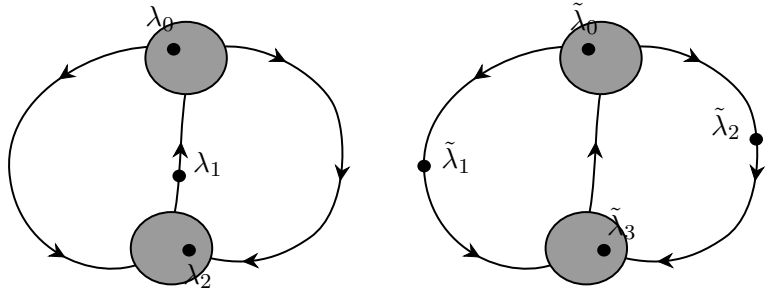
\begin{figure}[h!]
	\centering

\tikzset{every picture/.style={line width=0.75pt}} 

\begin{tikzpicture}[x=0.75pt,y=0.75pt,yscale=-1,xscale=1]

\draw    (234.48,45.36) .. controls (208.48,38.16) and (148.39,62.71) .. (151.59,106.98) .. controls (154.79,151.25) and (203.18,166.21) .. (225.07,149.42) .. controls (246.96,132.63) and (228.56,49.37) .. (253.84,44.67) .. controls (279.12,39.98) and (307.72,58.68) .. (315.07,73.27) .. controls (322.41,87.87) and (322.35,126.82) .. (303.45,140.19) .. controls (284.54,153.56) and (265.96,156.44) .. (239.38,146.79) ;
\draw [shift={(176.85,60.29)}, rotate = 325.79] [fill={rgb, 255:red, 0; green, 0; blue, 0 }  ][line width=0.08]  [draw opacity=0] (7.14,-3.43) -- (0,0) -- (7.14,3.43) -- (4.74,0) -- cycle    ;
\draw [shift={(179.5,150.2)}, rotate = 209.02] [fill={rgb, 255:red, 0; green, 0; blue, 0 }  ][line width=0.08]  [draw opacity=0] (7.14,-3.43) -- (0,0) -- (7.14,3.43) -- (4.74,0) -- cycle    ;
\draw [shift={(238.18,92.51)}, rotate = 93.12] [fill={rgb, 255:red, 0; green, 0; blue, 0 }  ][line width=0.08]  [draw opacity=0] (7.14,-3.43) -- (0,0) -- (7.14,3.43) -- (4.74,0) -- cycle    ;
\draw [shift={(290.84,51.56)}, rotate = 206.47] [fill={rgb, 255:red, 0; green, 0; blue, 0 }  ][line width=0.08]  [draw opacity=0] (7.14,-3.43) -- (0,0) -- (7.14,3.43) -- (4.74,0) -- cycle    ;
\draw [shift={(318.5,111.55)}, rotate = 278.32] [fill={rgb, 255:red, 0; green, 0; blue, 0 }  ][line width=0.08]  [draw opacity=0] (7.14,-3.43) -- (0,0) -- (7.14,3.43) -- (4.74,0) -- cycle    ;
\draw [shift={(269.58,152.46)}, rotate = 356.71] [fill={rgb, 255:red, 0; green, 0; blue, 0 }  ][line width=0.08]  [draw opacity=0] (7.14,-3.43) -- (0,0) -- (7.14,3.43) -- (4.74,0) -- cycle    ;
\draw  [draw opacity=0][fill={rgb, 255:red, 155; green, 155; blue, 155 }  ,fill opacity=0.47 ] (220.23,49.75) .. controls (220.23,39.67) and (229.43,31.5) .. (240.77,31.5) .. controls (252.11,31.5) and (261.3,39.67) .. (261.3,49.75) .. controls (261.3,59.83) and (252.11,68.01) .. (240.77,68.01) .. controls (229.43,68.01) and (220.23,59.83) .. (220.23,49.75) -- cycle ;
\draw  [fill={rgb, 255:red, 0; green, 0; blue, 0 }  ,fill opacity=1 ] (231.58,45.36) .. controls (231.58,43.88) and (232.88,42.68) .. (234.48,42.68) .. controls (236.08,42.68) and (237.38,43.88) .. (237.38,45.36) .. controls (237.38,46.84) and (236.08,48.04) .. (234.48,48.04) .. controls (232.88,48.04) and (231.58,46.84) .. (231.58,45.36) -- cycle ;
\draw  [draw opacity=0][fill={rgb, 255:red, 155; green, 155; blue, 155 }  ,fill opacity=0.47 ] (212.85,145.78) .. controls (212.85,135.7) and (222.04,127.53) .. (233.38,127.53) .. controls (244.72,127.53) and (253.92,135.7) .. (253.92,145.78) .. controls (253.92,155.86) and (244.72,164.04) .. (233.38,164.04) .. controls (222.04,164.04) and (212.85,155.86) .. (212.85,145.78) -- cycle ;
\draw  [fill={rgb, 255:red, 0; green, 0; blue, 0 }  ,fill opacity=1 ] (239.38,146.79) .. controls (239.38,145.31) and (240.68,144.11) .. (242.28,144.11) .. controls (243.88,144.11) and (245.18,145.31) .. (245.18,146.79) .. controls (245.18,148.27) and (243.88,149.48) .. (242.28,149.48) .. controls (240.68,149.48) and (239.38,148.27) .. (239.38,146.79) -- cycle ;
\draw  [fill={rgb, 255:red, 0; green, 0; blue, 0 }  ,fill opacity=1 ] (357.9,104.3) .. controls (357.9,102.82) and (359.2,101.61) .. (360.8,101.61) .. controls (362.4,101.61) and (363.7,102.82) .. (363.7,104.3) .. controls (363.7,105.78) and (362.4,106.98) .. (360.8,106.98) .. controls (359.2,106.98) and (357.9,105.78) .. (357.9,104.3) -- cycle ;
\draw    (443.69,45.36) .. controls (417.69,38.16) and (357.6,62.71) .. (360.8,106.98) .. controls (364,151.25) and (412.39,166.21) .. (434.28,149.42) .. controls (456.17,132.63) and (437.77,49.37) .. (463.05,44.67) .. controls (488.33,39.98) and (516.93,58.68) .. (524.28,73.27) .. controls (531.62,87.87) and (531.56,126.82) .. (512.66,140.19) .. controls (493.75,153.56) and (475.17,156.44) .. (448.59,146.79) ;
\draw [shift={(386.06,60.29)}, rotate = 325.79] [fill={rgb, 255:red, 0; green, 0; blue, 0 }  ][line width=0.08]  [draw opacity=0] (7.14,-3.43) -- (0,0) -- (7.14,3.43) -- (4.74,0) -- cycle    ;
\draw [shift={(388.71,150.2)}, rotate = 209.02] [fill={rgb, 255:red, 0; green, 0; blue, 0 }  ][line width=0.08]  [draw opacity=0] (7.14,-3.43) -- (0,0) -- (7.14,3.43) -- (4.74,0) -- cycle    ;
\draw [shift={(447.39,92.51)}, rotate = 93.12] [fill={rgb, 255:red, 0; green, 0; blue, 0 }  ][line width=0.08]  [draw opacity=0] (7.14,-3.43) -- (0,0) -- (7.14,3.43) -- (4.74,0) -- cycle    ;
\draw [shift={(500.05,51.56)}, rotate = 206.47] [fill={rgb, 255:red, 0; green, 0; blue, 0 }  ][line width=0.08]  [draw opacity=0] (7.14,-3.43) -- (0,0) -- (7.14,3.43) -- (4.74,0) -- cycle    ;
\draw [shift={(527.71,111.55)}, rotate = 278.32] [fill={rgb, 255:red, 0; green, 0; blue, 0 }  ][line width=0.08]  [draw opacity=0] (7.14,-3.43) -- (0,0) -- (7.14,3.43) -- (4.74,0) -- cycle    ;
\draw [shift={(478.79,152.46)}, rotate = 356.71] [fill={rgb, 255:red, 0; green, 0; blue, 0 }  ][line width=0.08]  [draw opacity=0] (7.14,-3.43) -- (0,0) -- (7.14,3.43) -- (4.74,0) -- cycle    ;
\draw  [draw opacity=0][fill={rgb, 255:red, 155; green, 155; blue, 155 }  ,fill opacity=0.47 ] (429.44,49.75) .. controls (429.44,39.67) and (438.64,31.5) .. (449.98,31.5) .. controls (461.32,31.5) and (470.51,39.67) .. (470.51,49.75) .. controls (470.51,59.83) and (461.32,68.01) .. (449.98,68.01) .. controls (438.64,68.01) and (429.44,59.83) .. (429.44,49.75) -- cycle ;
\draw  [fill={rgb, 255:red, 0; green, 0; blue, 0 }  ,fill opacity=1 ] (440.79,45.36) .. controls (440.79,43.88) and (442.09,42.68) .. (443.69,42.68) .. controls (445.29,42.68) and (446.59,43.88) .. (446.59,45.36) .. controls (446.59,46.84) and (445.29,48.04) .. (443.69,48.04) .. controls (442.09,48.04) and (440.79,46.84) .. (440.79,45.36) -- cycle ;
\draw  [draw opacity=0][fill={rgb, 255:red, 155; green, 155; blue, 155 }  ,fill opacity=0.47 ] (422.06,145.78) .. controls (422.06,135.7) and (431.25,127.53) .. (442.59,127.53) .. controls (453.93,127.53) and (463.13,135.7) .. (463.13,145.78) .. controls (463.13,155.86) and (453.93,164.04) .. (442.59,164.04) .. controls (431.25,164.04) and (422.06,155.86) .. (422.06,145.78) -- cycle ;
\draw  [fill={rgb, 255:red, 0; green, 0; blue, 0 }  ,fill opacity=1 ] (448.59,146.79) .. controls (448.59,145.31) and (449.89,144.11) .. (451.49,144.11) .. controls (453.09,144.11) and (454.39,145.31) .. (454.39,146.79) .. controls (454.39,148.27) and (453.09,149.48) .. (451.49,149.48) .. controls (449.89,149.48) and (448.59,148.27) .. (448.59,146.79) -- cycle ;
\draw  [fill={rgb, 255:red, 0; green, 0; blue, 0 }  ,fill opacity=1 ] (525.4,90.8) .. controls (525.4,89.32) and (526.7,88.11) .. (528.3,88.11) .. controls (529.9,88.11) and (531.2,89.32) .. (531.2,90.8) .. controls (531.2,92.28) and (529.9,93.48) .. (528.3,93.48) .. controls (526.7,93.48) and (525.4,92.28) .. (525.4,90.8) -- cycle ;
\draw  [fill={rgb, 255:red, 0; green, 0; blue, 0 }  ,fill opacity=1 ] (234.38,109.29) .. controls (234.38,107.81) and (235.68,106.61) .. (237.28,106.61) .. controls (238.88,106.61) and (240.18,107.81) .. (240.18,109.29) .. controls (240.18,110.77) and (238.88,111.98) .. (237.28,111.98) .. controls (235.68,111.98) and (234.38,110.77) .. (234.38,109.29) -- cycle ;

\draw (218.81,22) node [anchor=north west][inner sep=0.75pt]    {$\lambda _{0}$};
\draw (242.81,95.5) node [anchor=north west][inner sep=0.75pt]    {$\lambda _{1}$};
\draw (237.81,152) node [anchor=north west][inner sep=0.75pt]    {$\lambda _{2}$};
\draw (432.6,20.5) node [anchor=north west][inner sep=0.75pt]    {$\tilde{\lambda }_{0}$};
\draw (368.1,90.4) node [anchor=north west][inner sep=0.75pt]    {$\tilde{\lambda }_{1}$};
\draw (504.1,73.4) node [anchor=north west][inner sep=0.75pt]    {$\tilde{\lambda }_{2}$};
\draw (451.6,123.9) node [anchor=north west][inner sep=0.75pt]    {$\tilde{\lambda }_{3}$};

\end{tikzpicture}
	\caption{Links of two weakly outer bi-Lipschitz non-equivalent reduced pancake decompositions, $\{X_j = T(\lambda_{j-1},\lambda_j)\}_{j=1}^2$ and $\{\tilde X_j = T(\tilde \lambda_{j-1},\tilde \lambda_j)\}_{j=1}^3$, of an abnormal $\beta$-surface $X$. Points inside the shaded disks represent arcs with tangency order higher $\beta$.}\label{Fig. two non-equiv decompositions}
\end{figure}

\section{Minimal pancake decomposition of snakes}

Now we present a constructive way to obtain a minimal pancake decomposition for a given snake $X$. First, we move along the link of $X$ accordingly to some orientation, enumerating the nodal zones in the order they appear and taking into account nodal zones on a same node. Informally, the main idea of such construction is to look at the first time a nodal zone is on the same node of a previous nodal zone, and then ``break'' the link of $X$ at this nodal zone. We apply this same procedure now starting from this ``break point'' and repeat this process until we reach the endpoint of the link. The positions of such ``break points'' determine the minimal sequence of $X$, and such a sequence will generate a minimal pancake decomposition. Since the minimal sequence is obtained by applying an analog of the greedy algorithm for graphs, we call any corresponding pancake decomposition a greedy decomposition of $X$.

\begin{Def}\label{Def: minimal sequence}
    Let $X=T(\gamma_1,\gamma_2)$ be a $\beta$-snake oriented from $\gamma_1$ to $\gamma_2$. Let $\{N_i\}_{i=0}^{m}$ and $\{S_i\}_{i=1}^m$ be the decomposition of the Valette link of $X$ into nodal zones and segments, respectively, satisfying that the nodal zones were enumerated according to the orientation of $X$ and $N_{i-1}$, $N_i$ are the nodal zones adjacent to $S_i$, for each $i=1,\ldots,m$ (see Proposition 4.30 in \cite{GabrielovSouza}). Consider, for each $i=0,\ldots,m$, arcs $\theta_i \in N_i$. We define the sequence $\{j_0, j_1, \cdots , j_p\}$ recursively as follows: $j_0=0$ and, for each $i>0$, if there is no integer $k>j_{i-1}$ such that $T(\theta_{j_{i-1}},\theta_{k})$ is not LNE, then set $i-1=p$; otherwise, $j_i$ is the minimum integer greater than $j_{i-1}$ such that $T(\theta_{j_{i-1}},\theta_{j_i})$ is not LNE.
    
    The sequence $\{j_0, j_1, \cdots , j_p\}$ obtained as above is called the \textbf{minimal sequence of} $X$.
\end{Def}

\begin{remark}
    The minimal sequence $\{j_0, j_1, \cdots , j_p\}$ always satisfies $p\ge 1$ and $0=j_0<j_1<\cdots<j_p$. Moreover, the minimal sequence does not depend on the choice of the arcs $\theta_i$, but only on the nodal zones $N_i$ and thus only on the orientation of $X$.
\end{remark}

\begin{remark}
    As a direct consequence of Definition \ref{Def: minimal sequence}, one can alternatively obtain the minimal sequence $\{j_0, j_1, \cdots , j_p\}$ of a snake $X$ from its snake name $W=[x_0x_1\cdots x_m]$ (if it has at least two letters) as follows: $j_0=0$ and, for each $i>0$, if there is no integer $k>j_{i-1}$ such that the subword $[x_{j_{i-1}}\cdots x_{k}]$ of $W$ is not primitive, then set $i-1=p$; otherwise, $j_i$ is the minimum integer $k$ greater than $j_{i-1}$ such that the subword $[x_{j_{i-1}}\cdots x_{k}]$ of $W$ is not primitive. 
\end{remark}

\begin{Exam}\label{exemplo1}
    Consider the three snakes $X=T(\gamma_1,\gamma_2)$ oriented from $\gamma_1$ to $\gamma_2$, whose links are represented in Figure \ref{4} and their corresponding arcs $\theta_i$. In Figure \ref{4}a, the minimal sequence of $X$ is $\{0,1\}$; in Figure \ref{4}b, the minimal sequence of $X$ is $\{0,2\}$ and in Figure \ref{4}c, the minimal sequence of $X$ is $\{0,2,5\}$. Notice that $p=m$ in Figure \ref{4}a and \ref{4}c, and $p<m$ in Figure \ref{4}b.
    
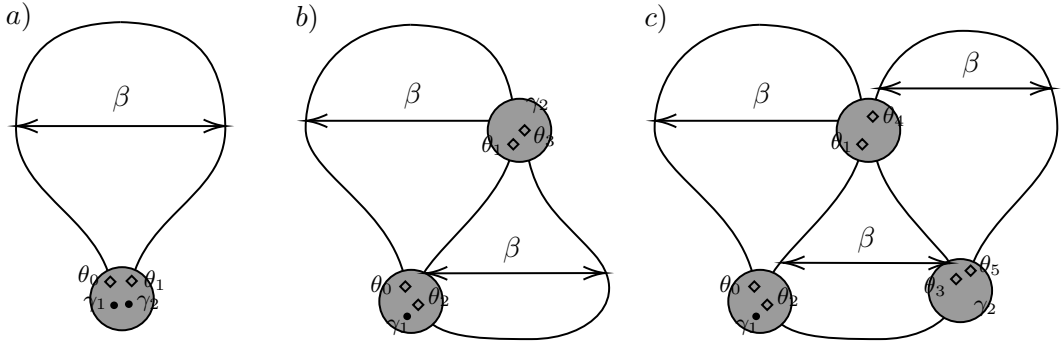
\begin{figure}[h!]
 \centering
{\tikzset{every picture/.style={line width=0.75pt}} 

\begin{tikzpicture}[x=0.75pt,y=0.75pt,yscale=-1,xscale=1]

\draw    (73.37,37.09) .. controls (39.17,38.06) and (21.58,51.67) .. (20.61,89.56) .. controls (19.63,127.45) and (70.06,135.76) .. (70.06,181.43) ;
\draw    (73.37,37.09) .. controls (107.56,38.06) and (125.15,51.67) .. (126.13,89.56) .. controls (127.1,127.45) and (77.49,135.19) .. (77.49,180.86) ;
\draw  [draw opacity=0][fill={rgb, 255:red, 155; green, 155; blue, 155 }  ,fill opacity=0.47 ] (58.04,176.71) .. controls (58.04,167.91) and (65.21,160.78) .. (74.06,160.78) .. controls (82.9,160.78) and (90.07,167.91) .. (90.07,176.71) .. controls (90.07,185.51) and (82.9,192.64) .. (74.06,192.64) .. controls (65.21,192.64) and (58.04,185.51) .. (58.04,176.71) -- cycle ;
\draw    (22.61,89.56) -- (124.13,89.56) ;
\draw [shift={(126.13,89.56)}, rotate = 180] [color={rgb, 255:red, 0; green, 0; blue, 0 }  ][line width=0.75]    (10.93,-3.29) .. controls (6.95,-1.4) and (3.31,-0.3) .. (0,0) .. controls (3.31,0.3) and (6.95,1.4) .. (10.93,3.29)   ;
\draw [shift={(20.61,89.56)}, rotate = 0] [color={rgb, 255:red, 0; green, 0; blue, 0 }  ][line width=0.75]    (10.93,-3.29) .. controls (6.95,-1.4) and (3.31,-0.3) .. (0,0) .. controls (3.31,0.3) and (6.95,1.4) .. (10.93,3.29)   ;
\draw  [fill={rgb, 255:red, 0; green, 0; blue, 0 }  ,fill opacity=1 ] (68.61,180.02) .. controls (68.61,179.25) and (69.26,178.62) .. (70.06,178.62) .. controls (70.86,178.62) and (71.5,179.25) .. (71.5,180.02) .. controls (71.5,180.8) and (70.86,181.43) .. (70.06,181.43) .. controls (69.26,181.43) and (68.61,180.8) .. (68.61,180.02) -- cycle ;
\draw  [fill={rgb, 255:red, 0; green, 0; blue, 0 }  ,fill opacity=1 ] (76.04,179.45) .. controls (76.04,178.68) and (76.69,178.05) .. (77.49,178.05) .. controls (78.28,178.05) and (78.93,178.68) .. (78.93,179.45) .. controls (78.93,180.23) and (78.28,180.86) .. (77.49,180.86) .. controls (76.69,180.86) and (76.04,180.23) .. (76.04,179.45) -- cycle ;
\draw    (395.43,38.43) .. controls (361.24,39.4) and (342.55,64.69) .. (342.67,90.89) .. controls (342.8,117.09) and (398.37,134.11) .. (393.87,185.61) ;
\draw    (395.43,38.43) .. controls (429.63,39.4) and (447.22,53) .. (448.19,90.89) .. controls (449.17,128.78) and (395.35,155.89) .. (398.55,176.29) .. controls (401.75,196.69) and (426.55,197.09) .. (456.55,197.09) .. controls (486.55,197.09) and (507.93,176.37) .. (487.93,153.97) .. controls (467.93,131.57) and (452.67,116.57) .. (452.27,90.97) .. controls (451.87,65.37) and (470.15,41.49) .. (498.15,41.49) .. controls (526.15,41.49) and (544.89,56.02) .. (545.75,89.49) .. controls (546.6,122.97) and (496.96,137.38) .. (499.62,181.38) ;
\draw  [draw opacity=0][fill={rgb, 255:red, 155; green, 155; blue, 155 }  ,fill opacity=0.47 ] (379.82,178.04) .. controls (379.82,169.25) and (386.99,162.12) .. (395.84,162.12) .. controls (404.69,162.12) and (411.86,169.25) .. (411.86,178.04) .. controls (411.86,186.84) and (404.69,193.97) .. (395.84,193.97) .. controls (386.99,193.97) and (379.82,186.84) .. (379.82,178.04) -- cycle ;
\draw    (344.67,86.89) -- (446.19,86.89) ;
\draw [shift={(448.19,86.89)}, rotate = 180] [color={rgb, 255:red, 0; green, 0; blue, 0 }  ][line width=0.75]    (10.93,-3.29) .. controls (6.95,-1.4) and (3.31,-0.3) .. (0,0) .. controls (3.31,0.3) and (6.95,1.4) .. (10.93,3.29)   ;
\draw [shift={(342.67,86.89)}, rotate = 0] [color={rgb, 255:red, 0; green, 0; blue, 0 }  ][line width=0.75]    (10.93,-3.29) .. controls (6.95,-1.4) and (3.31,-0.3) .. (0,0) .. controls (3.31,0.3) and (6.95,1.4) .. (10.93,3.29)   ;
\draw  [fill={rgb, 255:red, 0; green, 0; blue, 0 }  ,fill opacity=1 ] (392.42,185.61) .. controls (392.42,184.83) and (393.07,184.2) .. (393.87,184.2) .. controls (394.67,184.2) and (395.31,184.83) .. (395.31,185.61) .. controls (395.31,186.38) and (394.67,187.01) .. (393.87,187.01) .. controls (393.07,187.01) and (392.42,186.38) .. (392.42,185.61) -- cycle ;
\draw  [fill={rgb, 255:red, 0; green, 0; blue, 0 }  ,fill opacity=1 ] (498.18,181.38) .. controls (498.18,180.6) and (498.82,179.97) .. (499.62,179.97) .. controls (500.42,179.97) and (501.07,180.6) .. (501.07,181.38) .. controls (501.07,182.15) and (500.42,182.78) .. (499.62,182.78) .. controls (498.82,182.78) and (498.18,182.15) .. (498.18,181.38) -- cycle ;
\draw  [draw opacity=0][fill={rgb, 255:red, 155; green, 155; blue, 155 }  ,fill opacity=0.47 ] (434.49,91.71) .. controls (434.49,82.91) and (441.66,75.78) .. (450.51,75.78) .. controls (459.35,75.78) and (466.52,82.91) .. (466.52,91.71) .. controls (466.52,100.51) and (459.35,107.64) .. (450.51,107.64) .. controls (441.66,107.64) and (434.49,100.51) .. (434.49,91.71) -- cycle ;
\draw  [draw opacity=0][fill={rgb, 255:red, 155; green, 155; blue, 155 }  ,fill opacity=0.47 ] (480.82,172.61) .. controls (480.82,163.81) and (487.99,156.68) .. (496.84,156.68) .. controls (505.69,156.68) and (512.86,163.81) .. (512.86,172.61) .. controls (512.86,181.4) and (505.69,188.53) .. (496.84,188.53) .. controls (487.99,188.53) and (480.82,181.4) .. (480.82,172.61) -- cycle ;
\draw   (392.94,167.95) -- (395.53,170.55) -- (392.94,173.16) -- (390.35,170.55) -- cycle ;
\draw   (399.44,177.18) -- (402.03,179.78) -- (399.44,182.39) -- (396.84,179.78) -- cycle ;
\draw   (447.39,96.15) -- (449.98,98.75) -- (447.39,101.36) -- (444.8,98.75) -- cycle ;
\draw   (452.72,82.15) -- (455.31,84.75) -- (452.72,87.36) -- (450.13,84.75) -- cycle ;
\draw   (494.72,164.17) -- (497.31,166.78) -- (494.72,169.38) -- (492.13,166.78) -- cycle ;
\draw   (502.05,159.95) -- (504.64,162.55) -- (502.05,165.16) -- (499.46,162.55) -- cycle ;
\draw    (219.43,38.43) .. controls (185.24,39.4) and (166.55,64.69) .. (166.67,90.89) .. controls (166.8,117.09) and (222.37,134.11) .. (217.87,185.61) ;
\draw    (219.43,38.43) .. controls (253.63,39.4) and (271.22,53) .. (272.19,90.89) .. controls (273.17,128.78) and (219.35,155.89) .. (222.55,176.29) .. controls (225.75,196.69) and (250.55,197.09) .. (280.55,197.09) .. controls (310.55,197.09) and (331.93,176.37) .. (311.93,153.97) .. controls (291.93,131.57) and (277.47,126.52) .. (275.97,79.18) ;
\draw  [draw opacity=0][fill={rgb, 255:red, 155; green, 155; blue, 155 }  ,fill opacity=0.47 ] (203.82,178.04) .. controls (203.82,169.25) and (210.99,162.12) .. (219.84,162.12) .. controls (228.69,162.12) and (235.86,169.25) .. (235.86,178.04) .. controls (235.86,186.84) and (228.69,193.97) .. (219.84,193.97) .. controls (210.99,193.97) and (203.82,186.84) .. (203.82,178.04) -- cycle ;
\draw    (168.67,86.89) -- (270.19,86.89) ;
\draw [shift={(272.19,86.89)}, rotate = 180] [color={rgb, 255:red, 0; green, 0; blue, 0 }  ][line width=0.75]    (10.93,-3.29) .. controls (6.95,-1.4) and (3.31,-0.3) .. (0,0) .. controls (3.31,0.3) and (6.95,1.4) .. (10.93,3.29)   ;
\draw [shift={(166.67,86.89)}, rotate = 0] [color={rgb, 255:red, 0; green, 0; blue, 0 }  ][line width=0.75]    (10.93,-3.29) .. controls (6.95,-1.4) and (3.31,-0.3) .. (0,0) .. controls (3.31,0.3) and (6.95,1.4) .. (10.93,3.29)   ;
\draw  [fill={rgb, 255:red, 0; green, 0; blue, 0 }  ,fill opacity=1 ] (216.42,185.61) .. controls (216.42,184.83) and (217.07,184.2) .. (217.87,184.2) .. controls (218.67,184.2) and (219.31,184.83) .. (219.31,185.61) .. controls (219.31,186.38) and (218.67,187.01) .. (217.87,187.01) .. controls (217.07,187.01) and (216.42,186.38) .. (216.42,185.61) -- cycle ;
\draw  [fill={rgb, 255:red, 0; green, 0; blue, 0 }  ,fill opacity=1 ] (274.52,80.59) .. controls (274.52,79.81) and (275.17,79.18) .. (275.97,79.18) .. controls (276.77,79.18) and (277.41,79.81) .. (277.41,80.59) .. controls (277.41,81.36) and (276.77,81.99) .. (275.97,81.99) .. controls (275.17,81.99) and (274.52,81.36) .. (274.52,80.59) -- cycle ;
\draw  [draw opacity=0][fill={rgb, 255:red, 155; green, 155; blue, 155 }  ,fill opacity=0.47 ] (258.49,91.71) .. controls (258.49,82.91) and (265.66,75.78) .. (274.51,75.78) .. controls (283.35,75.78) and (290.52,82.91) .. (290.52,91.71) .. controls (290.52,100.51) and (283.35,107.64) .. (274.51,107.64) .. controls (265.66,107.64) and (258.49,100.51) .. (258.49,91.71) -- cycle ;
\draw   (216.94,167.95) -- (219.53,170.55) -- (216.94,173.16) -- (214.35,170.55) -- cycle ;
\draw   (223.44,177.18) -- (226.03,179.78) -- (223.44,182.39) -- (220.84,179.78) -- cycle ;
\draw   (271.39,96.15) -- (273.98,98.75) -- (271.39,101.36) -- (268.8,98.75) -- cycle ;
\draw   (277.1,89.11) -- (279.69,91.71) -- (277.1,94.31) -- (274.51,91.71) -- cycle ;
\draw    (229.17,163.89) -- (315.7,163.75) ;
\draw [shift={(317.7,163.75)}, rotate = 179.91] [color={rgb, 255:red, 0; green, 0; blue, 0 }  ][line width=0.75]    (10.93,-3.29) .. controls (6.95,-1.4) and (3.31,-0.3) .. (0,0) .. controls (3.31,0.3) and (6.95,1.4) .. (10.93,3.29)   ;
\draw [shift={(227.17,163.89)}, rotate = 359.91] [color={rgb, 255:red, 0; green, 0; blue, 0 }  ][line width=0.75]    (10.93,-3.29) .. controls (6.95,-1.4) and (3.31,-0.3) .. (0,0) .. controls (3.31,0.3) and (6.95,1.4) .. (10.93,3.29)   ;
\draw    (409.36,158.56) -- (489.7,158.75) ;
\draw [shift={(491.7,158.75)}, rotate = 180.13] [color={rgb, 255:red, 0; green, 0; blue, 0 }  ][line width=0.75]    (10.93,-3.29) .. controls (6.95,-1.4) and (3.31,-0.3) .. (0,0) .. controls (3.31,0.3) and (6.95,1.4) .. (10.93,3.29)   ;
\draw [shift={(407.36,158.56)}, rotate = 0.13] [color={rgb, 255:red, 0; green, 0; blue, 0 }  ][line width=0.75]    (10.93,-3.29) .. controls (6.95,-1.4) and (3.31,-0.3) .. (0,0) .. controls (3.31,0.3) and (6.95,1.4) .. (10.93,3.29)   ;
\draw    (457.76,70.56) -- (540.96,70.56) ;
\draw [shift={(542.96,70.56)}, rotate = 180] [color={rgb, 255:red, 0; green, 0; blue, 0 }  ][line width=0.75]    (10.93,-3.29) .. controls (6.95,-1.4) and (3.31,-0.3) .. (0,0) .. controls (3.31,0.3) and (6.95,1.4) .. (10.93,3.29)   ;
\draw [shift={(455.76,70.56)}, rotate = 0] [color={rgb, 255:red, 0; green, 0; blue, 0 }  ][line width=0.75]    (10.93,-3.29) .. controls (6.95,-1.4) and (3.31,-0.3) .. (0,0) .. controls (3.31,0.3) and (6.95,1.4) .. (10.93,3.29)   ;
\draw   (68.14,165.55) -- (70.73,168.15) -- (68.14,170.76) -- (65.55,168.15) -- cycle ;
\draw   (78.94,165.15) -- (81.53,167.75) -- (78.94,170.36) -- (76.35,167.75) -- cycle ;

\draw (67.67,67.98) node [anchor=north west][inner sep=0.75pt]    {$\beta $};
\draw (14,25.4) node [anchor=north west][inner sep=0.75pt]    {$a)$};
\draw (390.74,66.31) node [anchor=north west][inner sep=0.75pt]    {$\beta $};
\draw (336.07,26.73) node [anchor=north west][inner sep=0.75pt]    {$c)$};
\draw (381.87,185.87) node [anchor=north west][inner sep=0.75pt]  [font=\footnotesize]  {$\gamma _{1}$};
\draw (502.2,176.38) node [anchor=north west][inner sep=0.75pt]  [font=\footnotesize]  {$\gamma _{2}$};
\draw (79.99,173.79) node [anchor=north west][inner sep=0.75pt]  [font=\footnotesize]  {$\gamma _{2}$};
\draw (53.13,174.08) node [anchor=north west][inner sep=0.75pt]  [font=\footnotesize]  {$\gamma _{1}$};
\draw (373.89,163.38) node [anchor=north west][inner sep=0.75pt]  [font=\footnotesize]  {$\theta _{0}$};
\draw (430,91.42) node [anchor=north west][inner sep=0.75pt]  [font=\footnotesize]  {$\theta _{1}$};
\draw (403,169.42) node [anchor=north west][inner sep=0.75pt]  [font=\footnotesize]  {$\theta _{2}$};
\draw (477,163.98) node [anchor=north west][inner sep=0.75pt]  [font=\footnotesize]  {$\theta _{3}$};
\draw (456.33,78.31) node [anchor=north west][inner sep=0.75pt]  [font=\footnotesize]  {$\theta _{4}$};
\draw (504.78,153.42) node [anchor=north west][inner sep=0.75pt]  [font=\footnotesize]  {$\theta _{5}$};
\draw (214.74,66.31) node [anchor=north west][inner sep=0.75pt]    {$\beta $};
\draw (160.07,26.73) node [anchor=north west][inner sep=0.75pt]    {$b)$};
\draw (205.87,185.87) node [anchor=north west][inner sep=0.75pt]  [font=\footnotesize]  {$\gamma _{1}$};
\draw (276.34,73.66) node [anchor=north west][inner sep=0.75pt]  [font=\footnotesize]  {$\gamma _{2}$};
\draw (197.89,163.38) node [anchor=north west][inner sep=0.75pt]  [font=\footnotesize]  {$\theta _{0}$};
\draw (253.86,92.85) node [anchor=north west][inner sep=0.75pt]  [font=\footnotesize]  {$\theta _{1}$};
\draw (227,169.42) node [anchor=north west][inner sep=0.75pt]  [font=\footnotesize]  {$\theta _{2}$};
\draw (280.14,88.41) node [anchor=north west][inner sep=0.75pt]  [font=\footnotesize]  {$\theta _{3}$};
\draw (264.74,143.81) node [anchor=north west][inner sep=0.75pt]    {$\beta $};
\draw (443.54,138.81) node [anchor=north west][inner sep=0.75pt]    {$\beta $};
\draw (496.34,50.81) node [anchor=north west][inner sep=0.75pt]    {$\beta $};
\draw (50.29,158.56) node [anchor=north west][inner sep=0.75pt]  [font=\footnotesize]  {$\theta _{0}$};
\draw (83.49,160.96) node [anchor=north west][inner sep=0.75pt]  [font=\footnotesize]  {$\theta _{1}$};

\end{tikzpicture}}
\caption{Examples of snakes with choices of arcs $\theta_i$ in each nodal zone. Points inside shaded disks
represent arcs with the tangency order higher than the respective surface's exponent.}\label{4}
\end{figure}
\end{Exam}

\begin{Exam}
Consider $X$ the snake whose link is represented in both Figure \ref{5}a and Figure \ref{5}b. If $X=T(\gamma_1,\gamma_2)$ is oriented from $\gamma_1$ to $\gamma_2$, Figure \ref{5}a shows the corresponding arcs $\theta_i$ in nodal zones and hence the minimal sequence of $X$ is $\{0,3,6\}$. On the other hand, if $X=T(\gamma_2,\gamma_1)$ is oriented from $\gamma_2$ to $\gamma_1$, Figure \ref{5}b shows the corresponding arcs $\theta_i$ in nodal zones and hence the minimal sequence of $X$ is $\{0,2,5\}$.

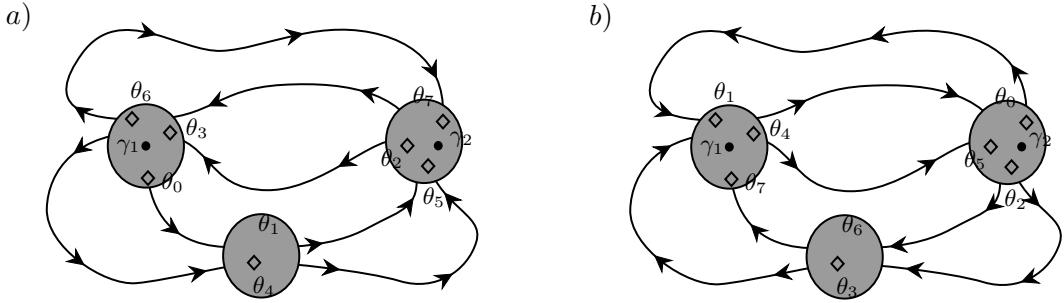
\begin{figure}[h!]
 \centering
{

\tikzset{every picture/.style={line width=0.75pt}} 

\begin{tikzpicture}[x=0.75pt,y=0.75pt,yscale=-1,xscale=1]

\draw    (115.83,122.53) .. controls (114.34,183.62) and (162.24,172.61) .. (174.24,173.27) .. controls (186.23,173.93) and (209.72,173.84) .. (227.61,164.03) .. controls (245.5,154.21) and (257.3,154.12) .. (250.4,128.03) .. controls (243.5,101.95) and (211.3,144.87) .. (177.41,144.87) .. controls (143.53,144.87) and (151.75,116.48) .. (116.36,115.82) .. controls (80.98,115.15) and (66.47,135) .. (65.87,154.15) .. controls (65.27,173.31) and (94.65,193.12) .. (110.85,191.8) .. controls (127.04,190.48) and (163.62,180.9) .. (175.02,181.23) .. controls (186.41,181.56) and (226.41,185.49) .. (244.1,190.78) .. controls (261.79,196.06) and (277.99,185.82) .. (282.48,170.96) .. controls (286.98,156.1) and (262.39,154.12) .. (256.1,125.06) .. controls (249.8,96) and (218.91,89.73) .. (184.73,91.71) .. controls (150.55,93.69) and (150.25,108.55) .. (116.06,109.21) .. controls (81.88,109.87) and (64.49,102.28) .. (83.68,75.86) .. controls (102.87,49.44) and (133.16,75.53) .. (154.15,74.54) .. controls (175.14,73.55) and (199.72,59.35) .. (229.41,65.29) .. controls (259.1,71.23) and (271.39,91.71) .. (263.29,122.42) ;
\draw [shift={(132.66,166.84)}, rotate = 216.19] [fill={rgb, 255:red, 0; green, 0; blue, 0 }  ][line width=0.08]  [draw opacity=0] (8.04,-3.86) -- (0,0) -- (8.04,3.86) -- (5.34,0) -- cycle    ;
\draw [shift={(204.84,171.66)}, rotate = 170.63] [fill={rgb, 255:red, 0; green, 0; blue, 0 }  ][line width=0.08]  [draw opacity=0] (8.04,-3.86) -- (0,0) -- (8.04,3.86) -- (5.34,0) -- cycle    ;
\draw [shift={(250.6,148.98)}, rotate = 124.24] [fill={rgb, 255:red, 0; green, 0; blue, 0 }  ][line width=0.08]  [draw opacity=0] (8.04,-3.86) -- (0,0) -- (8.04,3.86) -- (5.34,0) -- cycle    ;
\draw [shift={(213.14,132.91)}, rotate = 330.5] [fill={rgb, 255:red, 0; green, 0; blue, 0 }  ][line width=0.08]  [draw opacity=0] (8.04,-3.86) -- (0,0) -- (8.04,3.86) -- (5.34,0) -- cycle    ;
\draw [shift={(144.84,127.8)}, rotate = 46.79] [fill={rgb, 255:red, 0; green, 0; blue, 0 }  ][line width=0.08]  [draw opacity=0] (8.04,-3.86) -- (0,0) -- (8.04,3.86) -- (5.34,0) -- cycle    ;
\draw [shift={(79.58,126.27)}, rotate = 323.39] [fill={rgb, 255:red, 0; green, 0; blue, 0 }  ][line width=0.08]  [draw opacity=0] (8.04,-3.86) -- (0,0) -- (8.04,3.86) -- (5.34,0) -- cycle    ;
\draw [shift={(84.62,182.53)}, rotate = 216.25] [fill={rgb, 255:red, 0; green, 0; blue, 0 }  ][line width=0.08]  [draw opacity=0] (8.04,-3.86) -- (0,0) -- (8.04,3.86) -- (5.34,0) -- cycle    ;
\draw [shift={(146.24,185.6)}, rotate = 168.7] [fill={rgb, 255:red, 0; green, 0; blue, 0 }  ][line width=0.08]  [draw opacity=0] (8.04,-3.86) -- (0,0) -- (8.04,3.86) -- (5.34,0) -- cycle    ;
\draw [shift={(213.37,184.95)}, rotate = 187.54] [fill={rgb, 255:red, 0; green, 0; blue, 0 }  ][line width=0.08]  [draw opacity=0] (8.04,-3.86) -- (0,0) -- (8.04,3.86) -- (5.34,0) -- cycle    ;
\draw [shift={(270.24,187.22)}, rotate = 149.62] [fill={rgb, 255:red, 0; green, 0; blue, 0 }  ][line width=0.08]  [draw opacity=0] (8.04,-3.86) -- (0,0) -- (8.04,3.86) -- (5.34,0) -- cycle    ;
\draw [shift={(266.6,146.42)}, rotate = 48.2] [fill={rgb, 255:red, 0; green, 0; blue, 0 }  ][line width=0.08]  [draw opacity=0] (8.04,-3.86) -- (0,0) -- (8.04,3.86) -- (5.34,0) -- cycle    ;
\draw [shift={(224.06,94.55)}, rotate = 16.25] [fill={rgb, 255:red, 0; green, 0; blue, 0 }  ][line width=0.08]  [draw opacity=0] (8.04,-3.86) -- (0,0) -- (8.04,3.86) -- (5.34,0) -- cycle    ;
\draw [shift={(147.54,102.28)}, rotate = 335.15] [fill={rgb, 255:red, 0; green, 0; blue, 0 }  ][line width=0.08]  [draw opacity=0] (8.04,-3.86) -- (0,0) -- (8.04,3.86) -- (5.34,0) -- cycle    ;
\draw [shift={(80.04,102.84)}, rotate = 33.09] [fill={rgb, 255:red, 0; green, 0; blue, 0 }  ][line width=0.08]  [draw opacity=0] (8.04,-3.86) -- (0,0) -- (8.04,3.86) -- (5.34,0) -- cycle    ;
\draw [shift={(120.89,66.36)}, rotate = 194.55] [fill={rgb, 255:red, 0; green, 0; blue, 0 }  ][line width=0.08]  [draw opacity=0] (8.04,-3.86) -- (0,0) -- (8.04,3.86) -- (5.34,0) -- cycle    ;
\draw [shift={(194.98,65.91)}, rotate = 168.35] [fill={rgb, 255:red, 0; green, 0; blue, 0 }  ][line width=0.08]  [draw opacity=0] (8.04,-3.86) -- (0,0) -- (8.04,3.86) -- (5.34,0) -- cycle    ;
\draw [shift={(262.39,88.19)}, rotate = 243.52] [fill={rgb, 255:red, 0; green, 0; blue, 0 }  ][line width=0.08]  [draw opacity=0] (8.04,-3.86) -- (0,0) -- (8.04,3.86) -- (5.34,0) -- cycle    ;
\draw  [draw opacity=0][fill={rgb, 255:red, 155; green, 155; blue, 155 }  ,fill opacity=0.47 ] (96.63,122.53) .. controls (96.63,110.91) and (105.23,101.49) .. (115.83,101.49) .. controls (126.44,101.49) and (135.04,110.91) .. (135.04,122.53) .. controls (135.04,134.15) and (126.44,143.57) .. (115.83,143.57) .. controls (105.23,143.57) and (96.63,134.15) .. (96.63,122.53) -- cycle ;
\draw   (174.24,169.84) -- (177.34,173.27) -- (174.24,176.71) -- (171.13,173.27) -- cycle ;
\draw   (116.89,135.33) -- (120,138.77) -- (116.89,142.21) -- (113.79,138.77) -- cycle ;
\draw  [draw opacity=0][fill={rgb, 255:red, 155; green, 155; blue, 155 }  ,fill opacity=0.47 ] (154.8,178.01) .. controls (154.8,166.39) and (163.4,156.97) .. (174.01,156.97) .. controls (184.62,156.97) and (193.22,166.39) .. (193.22,178.01) .. controls (193.22,189.63) and (184.62,199.05) .. (174.01,199.05) .. controls (163.4,199.05) and (154.8,189.63) .. (154.8,178.01) -- cycle ;
\draw  [draw opacity=0][fill={rgb, 255:red, 155; green, 155; blue, 155 }  ,fill opacity=0.47 ] (236.96,120.22) .. controls (236.96,108.6) and (245.56,99.18) .. (256.17,99.18) .. controls (266.78,99.18) and (275.38,108.6) .. (275.38,120.22) .. controls (275.38,131.84) and (266.78,141.26) .. (256.17,141.26) .. controls (245.56,141.26) and (236.96,131.84) .. (236.96,120.22) -- cycle ;
\draw  [fill={rgb, 255:red, 0; green, 0; blue, 0 }  ,fill opacity=1 ] (114.1,122.53) .. controls (114.1,121.51) and (114.88,120.68) .. (115.83,120.68) .. controls (116.79,120.68) and (117.57,121.51) .. (117.57,122.53) .. controls (117.57,123.56) and (116.79,124.39) .. (115.83,124.39) .. controls (114.88,124.39) and (114.1,123.56) .. (114.1,122.53) -- cycle ;
\draw  [fill={rgb, 255:red, 0; green, 0; blue, 0 }  ,fill opacity=1 ] (261.56,122.42) .. controls (261.56,121.4) and (262.33,120.56) .. (263.29,120.56) .. controls (264.25,120.56) and (265.03,121.4) .. (265.03,122.42) .. controls (265.03,123.44) and (264.25,124.27) .. (263.29,124.27) .. controls (262.33,124.27) and (261.56,123.44) .. (261.56,122.42) -- cycle ;
\draw   (170.41,178.01) -- (173.52,181.45) -- (170.41,184.88) -- (167.3,181.45) -- cycle ;
\draw   (247.7,118.98) -- (250.81,122.42) -- (247.7,125.86) -- (244.59,122.42) -- cycle ;
\draw   (128.06,112.41) -- (131.17,115.85) -- (128.06,119.28) -- (124.95,115.85) -- cycle ;
\draw   (108.87,105.47) -- (111.98,108.91) -- (108.87,112.35) -- (105.76,108.91) -- cycle ;
\draw   (258.2,129.25) -- (261.3,132.69) -- (258.2,136.12) -- (255.09,132.69) -- cycle ;
\draw   (265.69,106.79) -- (268.8,110.23) -- (265.69,113.67) -- (262.58,110.23) -- cycle ;
\draw    (410.09,123.06) .. controls (408.59,184.15) and (456.5,173.14) .. (468.49,173.8) .. controls (480.49,174.46) and (503.98,174.37) .. (521.87,164.56) .. controls (539.76,154.74) and (551.55,154.65) .. (544.66,128.56) .. controls (537.76,102.47) and (505.56,145.4) .. (471.67,145.4) .. controls (437.79,145.4) and (446.01,117) .. (410.62,116.34) .. controls (375.24,115.68) and (360.73,135.53) .. (360.13,154.68) .. controls (359.53,173.84) and (388.91,193.65) .. (405.11,192.33) .. controls (421.3,191.01) and (457.88,181.43) .. (469.27,181.76) .. controls (480.67,182.09) and (520.67,186.02) .. (538.36,191.3) .. controls (556.05,196.59) and (572.24,186.35) .. (576.74,171.49) .. controls (581.24,156.63) and (556.65,154.65) .. (550.36,125.59) .. controls (544.06,96.53) and (513.17,90.26) .. (478.99,92.24) .. controls (444.81,94.22) and (444.51,109.08) .. (410.32,109.74) .. controls (376.14,110.4) and (358.75,102.8) .. (377.94,76.39) .. controls (397.13,49.97) and (427.41,76.06) .. (448.4,75.07) .. controls (469.39,74.07) and (493.98,59.87) .. (523.67,65.82) .. controls (553.35,71.76) and (565.65,92.24) .. (557.55,122.95) ;
\draw [shift={(421.04,162.3)}, rotate = 45.22] [fill={rgb, 255:red, 0; green, 0; blue, 0 }  ][line width=0.08]  [draw opacity=0] (8.04,-3.86) -- (0,0) -- (8.04,3.86) -- (5.34,0) -- cycle    ;
\draw [shift={(491.25,173.3)}, rotate = 353.14] [fill={rgb, 255:red, 0; green, 0; blue, 0 }  ][line width=0.08]  [draw opacity=0] (8.04,-3.86) -- (0,0) -- (8.04,3.86) -- (5.34,0) -- cycle    ;
\draw [shift={(539.6,155.16)}, rotate = 321.18] [fill={rgb, 255:red, 0; green, 0; blue, 0 }  ][line width=0.08]  [draw opacity=0] (8.04,-3.86) -- (0,0) -- (8.04,3.86) -- (5.34,0) -- cycle    ;
\draw [shift={(514.24,129.52)}, rotate = 149.97] [fill={rgb, 255:red, 0; green, 0; blue, 0 }  ][line width=0.08]  [draw opacity=0] (8.04,-3.86) -- (0,0) -- (8.04,3.86) -- (5.34,0) -- cycle    ;
\draw [shift={(444.5,134.11)}, rotate = 226.87] [fill={rgb, 255:red, 0; green, 0; blue, 0 }  ][line width=0.08]  [draw opacity=0] (8.04,-3.86) -- (0,0) -- (8.04,3.86) -- (5.34,0) -- cycle    ;
\draw [shift={(380.52,122.44)}, rotate = 150.46] [fill={rgb, 255:red, 0; green, 0; blue, 0 }  ][line width=0.08]  [draw opacity=0] (8.04,-3.86) -- (0,0) -- (8.04,3.86) -- (5.34,0) -- cycle    ;
\draw [shift={(372.55,178.02)}, rotate = 40.83] [fill={rgb, 255:red, 0; green, 0; blue, 0 }  ][line width=0.08]  [draw opacity=0] (8.04,-3.86) -- (0,0) -- (8.04,3.86) -- (5.34,0) -- cycle    ;
\draw [shift={(432.92,187.64)}, rotate = 348.68] [fill={rgb, 255:red, 0; green, 0; blue, 0 }  ][line width=0.08]  [draw opacity=0] (8.04,-3.86) -- (0,0) -- (8.04,3.86) -- (5.34,0) -- cycle    ;
\draw [shift={(499.42,184.43)}, rotate = 7] [fill={rgb, 255:red, 0; green, 0; blue, 0 }  ][line width=0.08]  [draw opacity=0] (8.04,-3.86) -- (0,0) -- (8.04,3.86) -- (5.34,0) -- cycle    ;
\draw [shift={(557.55,191.13)}, rotate = 338.73] [fill={rgb, 255:red, 0; green, 0; blue, 0 }  ][line width=0.08]  [draw opacity=0] (8.04,-3.86) -- (0,0) -- (8.04,3.86) -- (5.34,0) -- cycle    ;
\draw [shift={(566.44,152.77)}, rotate = 224.42] [fill={rgb, 255:red, 0; green, 0; blue, 0 }  ][line width=0.08]  [draw opacity=0] (8.04,-3.86) -- (0,0) -- (8.04,3.86) -- (5.34,0) -- cycle    ;
\draw [shift={(525.89,97.57)}, rotate = 200.58] [fill={rgb, 255:red, 0; green, 0; blue, 0 }  ][line width=0.08]  [draw opacity=0] (8.04,-3.86) -- (0,0) -- (8.04,3.86) -- (5.34,0) -- cycle    ;
\draw [shift={(448.83,99.54)}, rotate = 155.18] [fill={rgb, 255:red, 0; green, 0; blue, 0 }  ][line width=0.08]  [draw opacity=0] (8.04,-3.86) -- (0,0) -- (8.04,3.86) -- (5.34,0) -- cycle    ;
\draw [shift={(381.45,106.95)}, rotate = 200.47] [fill={rgb, 255:red, 0; green, 0; blue, 0 }  ][line width=0.08]  [draw opacity=0] (8.04,-3.86) -- (0,0) -- (8.04,3.86) -- (5.34,0) -- cycle    ;
\draw [shift={(407.23,65.1)}, rotate = 10.65] [fill={rgb, 255:red, 0; green, 0; blue, 0 }  ][line width=0.08]  [draw opacity=0] (8.04,-3.86) -- (0,0) -- (8.04,3.86) -- (5.34,0) -- cycle    ;
\draw [shift={(481.28,68.2)}, rotate = 346.8] [fill={rgb, 255:red, 0; green, 0; blue, 0 }  ][line width=0.08]  [draw opacity=0] (8.04,-3.86) -- (0,0) -- (8.04,3.86) -- (5.34,0) -- cycle    ;
\draw [shift={(552.5,81.74)}, rotate = 54.94] [fill={rgb, 255:red, 0; green, 0; blue, 0 }  ][line width=0.08]  [draw opacity=0] (8.04,-3.86) -- (0,0) -- (8.04,3.86) -- (5.34,0) -- cycle    ;
\draw  [draw opacity=0][fill={rgb, 255:red, 155; green, 155; blue, 155 }  ,fill opacity=0.47 ] (390.88,123.06) .. controls (390.88,111.44) and (399.48,102.02) .. (410.09,102.02) .. controls (420.7,102.02) and (429.3,111.44) .. (429.3,123.06) .. controls (429.3,134.68) and (420.7,144.1) .. (410.09,144.1) .. controls (399.48,144.1) and (390.88,134.68) .. (390.88,123.06) -- cycle ;
\draw   (468.49,170.37) -- (471.6,173.8) -- (468.49,177.24) -- (465.39,173.8) -- cycle ;
\draw   (411.15,135.86) -- (414.26,139.3) -- (411.15,142.74) -- (408.05,139.3) -- cycle ;
\draw  [draw opacity=0][fill={rgb, 255:red, 155; green, 155; blue, 155 }  ,fill opacity=0.47 ] (449.06,178.54) .. controls (449.06,166.92) and (457.66,157.5) .. (468.27,157.5) .. controls (478.87,157.5) and (487.48,166.92) .. (487.48,178.54) .. controls (487.48,190.16) and (478.87,199.58) .. (468.27,199.58) .. controls (457.66,199.58) and (449.06,190.16) .. (449.06,178.54) -- cycle ;
\draw  [draw opacity=0][fill={rgb, 255:red, 155; green, 155; blue, 155 }  ,fill opacity=0.47 ] (531.22,120.75) .. controls (531.22,109.13) and (539.82,99.71) .. (550.43,99.71) .. controls (561.04,99.71) and (569.64,109.13) .. (569.64,120.75) .. controls (569.64,132.37) and (561.04,141.79) .. (550.43,141.79) .. controls (539.82,141.79) and (531.22,132.37) .. (531.22,120.75) -- cycle ;
\draw  [fill={rgb, 255:red, 0; green, 0; blue, 0 }  ,fill opacity=1 ] (408.36,123.06) .. controls (408.36,122.04) and (409.14,121.21) .. (410.09,121.21) .. controls (411.05,121.21) and (411.83,122.04) .. (411.83,123.06) .. controls (411.83,124.08) and (411.05,124.92) .. (410.09,124.92) .. controls (409.14,124.92) and (408.36,124.08) .. (408.36,123.06) -- cycle ;
\draw  [fill={rgb, 255:red, 0; green, 0; blue, 0 }  ,fill opacity=1 ] (555.82,122.95) .. controls (555.82,121.92) and (556.59,121.09) .. (557.55,121.09) .. controls (558.51,121.09) and (559.29,121.92) .. (559.29,122.95) .. controls (559.29,123.97) and (558.51,124.8) .. (557.55,124.8) .. controls (556.59,124.8) and (555.82,123.97) .. (555.82,122.95) -- cycle ;
\draw   (464.67,178.54) -- (467.78,181.98) -- (464.67,185.41) -- (461.56,181.98) -- cycle ;
\draw   (541.96,119.51) -- (545.07,122.95) -- (541.96,126.39) -- (538.85,122.95) -- cycle ;
\draw   (422.32,112.94) -- (425.42,116.37) -- (422.32,119.81) -- (419.21,116.37) -- cycle ;
\draw   (403.13,106) -- (406.23,109.44) -- (403.13,112.88) -- (400.02,109.44) -- cycle ;
\draw   (552.45,129.78) -- (555.56,133.22) -- (552.45,136.65) -- (549.35,133.22) -- cycle ;
\draw   (559.95,107.32) -- (563.06,110.76) -- (559.95,114.2) -- (556.84,110.76) -- cycle ;

\draw (43.7,48.61) node [anchor=north west][inner sep=0.75pt]    {$a)$};
\draw (99.73,117.87) node [anchor=north west][inner sep=0.75pt]  [font=\footnotesize]  {$\gamma _{1}$};
\draw (122.14,134.56) node [anchor=north west][inner sep=0.75pt]  [font=\footnotesize]  {$\theta _{0}$};
\draw (171.02,155.61) node [anchor=north west][inner sep=0.75pt]  [font=\footnotesize]  {$\theta _{1}$};
\draw (232.79,122.59) node [anchor=north west][inner sep=0.75pt]  [font=\footnotesize]  {$\theta _{2}$};
\draw (134.54,107.4) node [anchor=north west][inner sep=0.75pt]  [font=\footnotesize]  {$\theta _{3}$};
\draw (168.32,186.98) node [anchor=north west][inner sep=0.75pt]  [font=\footnotesize]  {$\theta _{4}$};
\draw (253.96,142.02) node [anchor=north west][inner sep=0.75pt]  [font=\footnotesize]  {$\theta _{5}$};
\draw (106.55,88.58) node [anchor=north west][inner sep=0.75pt]  [font=\footnotesize]  {$\theta _{6}$};
\draw (248.98,91.22) node [anchor=north west][inner sep=0.75pt]  [font=\footnotesize]  {$\theta _{7}$};
\draw (337.95,49.14) node [anchor=north west][inner sep=0.75pt]    {$b)$};
\draw (393.98,118.4) node [anchor=north west][inner sep=0.75pt]  [font=\footnotesize]  {$\gamma _{1}$};
\draw (416.4,135.09) node [anchor=north west][inner sep=0.75pt]  [font=\footnotesize]  {$\theta _{7}$};
\draw (465.28,156.14) node [anchor=north west][inner sep=0.75pt]  [font=\footnotesize]  {$\theta _{6}$};
\draw (527.05,123.12) node [anchor=north west][inner sep=0.75pt]  [font=\footnotesize]  {$\theta _{5}$};
\draw (428.8,107.93) node [anchor=north west][inner sep=0.75pt]  [font=\footnotesize]  {$\theta _{4}$};
\draw (462.58,187.51) node [anchor=north west][inner sep=0.75pt]  [font=\footnotesize]  {$\theta _{3}$};
\draw (547.42,142.99) node [anchor=north west][inner sep=0.75pt]  [font=\footnotesize]  {$\theta _{2}$};
\draw (400.81,89.1) node [anchor=north west][inner sep=0.75pt]  [font=\footnotesize]  {$\theta _{1}$};
\draw (542.44,91.75) node [anchor=north west][inner sep=0.75pt]  [font=\footnotesize]  {$\theta _{0}$};
\draw (267.65,114) node [anchor=north west][inner sep=0.75pt]  [font=\footnotesize]  {$\gamma _{2}$};
\draw (559.9,114.44) node [anchor=north west][inner sep=0.75pt]  [font=\footnotesize]  {$\gamma _{2}$};

\end{tikzpicture}}
\caption{Examples of snakes with choices of arcs $\theta_i$ in each nodal zone. Points inside shaded disks
represent arcs with the tangency order higher than the respective snake's exponent.}\label{5}
\end{figure}
\end{Exam}

In the rest of this section, let $X=T(\gamma_1, \gamma_2)$ be a $\beta$-snake oriented from $\gamma_1$ to $\gamma_2$ and let $\{N_i\}_{i=0}^{m}$, $\{S_i\}_{i=1}^m$ be the decomposition of the $V(X)$ into nodal zones and segments, respectively, as in Definition \ref{Def: minimal sequence}, enumerated accordingly to the given orientation.

\begin{Lem}\label{Lem: tord segments and nodal zones}
    For every $i \in \{0,1,\dots,m\}$, $j\in \{1,\dots,m\}$ such that $j \ne i, i+1$,we have $\operatorname{tord}(N_i,S_j) = \beta.$
\end{Lem}
\begin{proof}
    By Proposition 4.27 of \cite{GabrielovSouza}, given $\gamma \in N_i$, if $\itord(\gamma,\gamma') >\beta$ then $\gamma' \in N_i$, for any $\gamma'\in V(X)$. Thus, since $S_j\cap N_i= \emptyset$ for all $j\ne i, i+1$, we must have $\operatorname{tord}(S_i,N_j) = \beta$. 
\end{proof}

\begin{Lem}\label{Lem: segments gr beta implies adj nodal z gr beta}
   Let $S, S'$ be distinct segments of $X$ and let $N,\tilde N$ and $N',\Tilde{N'}$ be nodal zones adjacent to $S$ and $S'$, respectively. Assume that $\operatorname{tord}(\tilde N, \tilde N') >\beta$ and $\operatorname{tord}(S,S')>\beta$. Then $\operatorname{tord}(N,N')>\beta$.
\end{Lem}
\begin{proof}
    Suppose that $N$ and $N'$ are on nodes $\mathcal{N}$ and $\mathcal{N}'$, respectively. We have $\operatorname{tord}(\tilde N, \tilde N') >\beta$, implying that $\tilde N$ and $\tilde N'$ are in the same node $\tilde{\mathcal{N}}$. Since $\operatorname{tord}(S,S')>\beta$, $S$ and $S'$ belongs to the same cluster of $\mathcal{S}(\tilde{\mathcal{N}},\mathcal{N})$ and the same cluster of $\mathcal{S}(\tilde{\mathcal{N}},\mathcal{N}')$. Therefore, $\mathcal{N}=\mathcal{N}'$ by Proposition 4.59 of \cite{GabrielovSouza} and the result follows.
\end{proof}

\begin{remark}
    Lemmas \ref{Lem: tord segments and nodal zones} and \ref{Lem: segments gr beta implies adj nodal z gr beta} are also true when $X$ is a circular snake with nodal zones, and their respective proofs are analogous.
\end{remark}

\begin{Teo}\label{Teo: minimal pancake decomposition}
    Let $\{j_0,j_1,\dots,j_p\}$ be the minimal sequence of $X$. Consider the following decomposition of $X$ into H\"older triangles $\{X_i\}_{i=1}^{p+1}$, with $X_i = T(\lambda_{i-1},\lambda_i)$ defined recursively as follows: 
    \begin{enumerate}
        \item set $\lambda_0 = \gamma_1$ and choose any $\lambda_1 \in S_{j_1}$;
        \item for all $1<i\le p$, choose $\lambda_i \in S_{j_i}$ such that $\mathcal{H}(X_i)_{a,\mu(X)}(\lambda_i)\setminus\{0\}$ consists in only one connected component, for $a>0$ small enough;
        \item set $\lambda_{p+1}=\gamma_2$.
    \end{enumerate}
    Then,  $\{X_i\}_{i=1}^{p+1}$ is a minimal pancake decomposition of $X$. 
\end{Teo}

\begin{proof}
    Choosing any $\lambda_1$ in $S_{j_1}$ we obtain that $X_1$ is LNE, otherwise, it follows from the condition on the construction of the minimal sequence and Lemma \ref{Lem: tord segments and nodal zones} that $\operatorname{tord}(S_k, S_{j_1})>\beta$, for some segment $S_k$ where $j_0 < k < j_1$. However, Lemma \ref{Lem: segments gr beta implies adj nodal z gr beta} implies that there would exist a nodal zone $N_l$, adjacent to the segment $S_k$, with $0\le l<j_1$, where $\operatorname{tord}(N_l,N_{j_1 - 1})>\beta$, a contradiction with the minimality of $j_1$ in the minimal sequence of $X$. 

    If $\operatorname{tord}(S_{j_2},S_{j_1})=\beta$, then any arc $\lambda_2 \in S_{j_2}$ clearly satisfies $(2)$ and a similar argument shows that $X_2$ is LNE. If $\operatorname{tord}(S_{j_2}, S_{j_1})>\beta$, we must be careful when choosing $\lambda_2 \in S_{j_2}$ (see Figure \ref{Fig: Condition (2) of Theorem}). In this case we necessarily have $\operatorname{tord}(N_{j_1},N_{j_2})>\beta$. Notice that the nodal zones adjacent to $S_{j_2}$ are $N_{j_2}$ and $N_{j_2 - 1}$. Let $\lambda'_1 \in S_{j_2}$ be an arc such that $\operatorname{tord}(\lambda_1,\lambda'_1)>\beta$. Given an arc $\theta_{j_2 - 1} \in N_{j_2 - 1}$ we have $\lambda_2 \in G(T(\theta_{j_2 - 1},\lambda'_1)) \subset S_{j_2}$ such that $\operatorname{tord}(\lambda_1, \mathcal{H}(X_2)_{a,\mu(X)}(\lambda_2)) = \beta$ for any $a>0$ small enough (the existence of such arcs $\lambda_2$ is guaranteed by taking a sequence of generic arcs of $T(\theta_{j_2 - 1},\lambda'_1)$ converging to $\theta_{j_2 - 1}$, see Remark \ref{Rem: generic arcs of a non-singular HT}). In particular, this implies that $\mathcal{H}(X_2)_{a,\mu(X)}(\lambda_2)\setminus\{0\}$, for $a>0$ small enough, has a single connected component. Hence, arcs satisfying condition $(2)$ do exist. If $\lambda_2$ is any arc in $S_{j_2}$ satisfying $(2)$, then, for any $\theta_{j_1} \in N_{j_1}$, $\operatorname{tord}(\lambda_2,T(\lambda_1, \theta_{j_1})) = \beta$. Thus $X_2=T(\lambda_1, \lambda_2)$ is LNE. Therefore, we proved in both cases that $X_2$ is LNE. The proof that $X_i$, $i>2$, is LNE follows analogously.  
    
    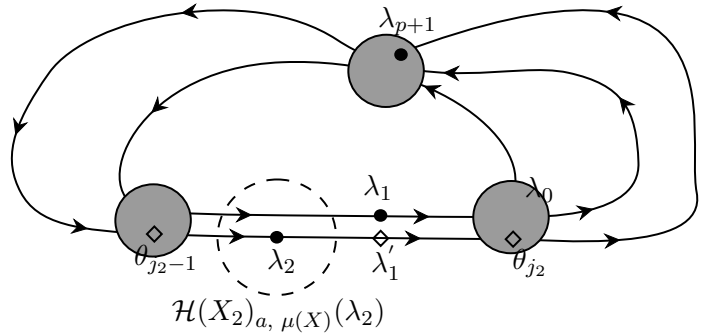
\begin{figure}[h!]
 \centering

\tikzset{every picture/.style={line width=0.75pt}} 

\begin{tikzpicture}[x=0.75pt,y=0.75pt,yscale=-1,xscale=1]

\draw    (431.81,140.09) .. controls (445,114.13) and (395.85,68.69) .. (316.46,71.63) .. controls (237.06,74.56) and (214.07,145.37) .. (252.46,147.21) .. controls (290.86,149.05) and (307.77,148.75) .. (343.42,149.15) .. controls (379.08,149.56) and (414.3,149.75) .. (431.71,149.8) .. controls (449.13,149.86) and (496.9,149.73) .. (495.99,122.4) .. controls (495.09,95.06) and (488.05,83.04) .. (460.62,76.95) .. controls (433.2,70.85) and (387.85,89.1) .. (343.93,60.28) .. controls (300.02,31.47) and (218.93,48.78) .. (198.93,76.78) .. controls (178.93,104.78) and (175.93,121.28) .. (184.43,137.28) .. controls (192.93,153.28) and (212.41,156.79) .. (251.83,158.35) .. controls (291.26,159.91) and (303.94,159.67) .. (337.35,160.06) .. controls (370.77,160.46) and (405.22,160.72) .. (429.03,160.94) .. controls (452.85,161.16) and (531.33,169.28) .. (525.33,137.28) .. controls (519.33,105.28) and (553.33,-0.22) .. (373.33,67.78) ;
\draw [shift={(388.14,83.75)}, rotate = 23.56] [fill={rgb, 255:red, 0; green, 0; blue, 0 }  ][line width=0.08]  [draw opacity=0] (7.14,-3.43) -- (0,0) -- (7.14,3.43) -- (4.74,0) -- cycle    ;
\draw [shift={(250.65,96.86)}, rotate = 317.05] [fill={rgb, 255:red, 0; green, 0; blue, 0 }  ][line width=0.08]  [draw opacity=0] (7.14,-3.43) -- (0,0) -- (7.14,3.43) -- (4.74,0) -- cycle    ;
\draw [shift={(300.59,148.75)}, rotate = 180.87] [fill={rgb, 255:red, 0; green, 0; blue, 0 }  ][line width=0.08]  [draw opacity=0] (7.14,-3.43) -- (0,0) -- (7.14,3.43) -- (4.74,0) -- cycle    ;
\draw [shift={(390.19,149.58)}, rotate = 180.41] [fill={rgb, 255:red, 0; green, 0; blue, 0 }  ][line width=0.08]  [draw opacity=0] (7.14,-3.43) -- (0,0) -- (7.14,3.43) -- (4.74,0) -- cycle    ;
\draw [shift={(472.8,145.83)}, rotate = 165.72] [fill={rgb, 255:red, 0; green, 0; blue, 0 }  ][line width=0.08]  [draw opacity=0] (7.14,-3.43) -- (0,0) -- (7.14,3.43) -- (4.74,0) -- cycle    ;
\draw [shift={(486.56,89.63)}, rotate = 51.11] [fill={rgb, 255:red, 0; green, 0; blue, 0 }  ][line width=0.08]  [draw opacity=0] (7.14,-3.43) -- (0,0) -- (7.14,3.43) -- (4.74,0) -- cycle    ;
\draw [shift={(397.97,76.95)}, rotate = 2.25] [fill={rgb, 255:red, 0; green, 0; blue, 0 }  ][line width=0.08]  [draw opacity=0] (7.14,-3.43) -- (0,0) -- (7.14,3.43) -- (4.74,0) -- cycle    ;
\draw [shift={(265.87,46.76)}, rotate = 353.15] [fill={rgb, 255:red, 0; green, 0; blue, 0 }  ][line width=0.08]  [draw opacity=0] (7.14,-3.43) -- (0,0) -- (7.14,3.43) -- (4.74,0) -- cycle    ;
\draw [shift={(181.94,108)}, rotate = 288.42] [fill={rgb, 255:red, 0; green, 0; blue, 0 }  ][line width=0.08]  [draw opacity=0] (7.14,-3.43) -- (0,0) -- (7.14,3.43) -- (4.74,0) -- cycle    ;
\draw [shift={(217.78,155.61)}, rotate = 189.47] [fill={rgb, 255:red, 0; green, 0; blue, 0 }  ][line width=0.08]  [draw opacity=0] (7.14,-3.43) -- (0,0) -- (7.14,3.43) -- (4.74,0) -- cycle    ;
\draw [shift={(297.21,159.65)}, rotate = 180.9] [fill={rgb, 255:red, 0; green, 0; blue, 0 }  ][line width=0.08]  [draw opacity=0] (7.14,-3.43) -- (0,0) -- (7.14,3.43) -- (4.74,0) -- cycle    ;
\draw [shift={(385.85,160.56)}, rotate = 180.52] [fill={rgb, 255:red, 0; green, 0; blue, 0 }  ][line width=0.08]  [draw opacity=0] (7.14,-3.43) -- (0,0) -- (7.14,3.43) -- (4.74,0) -- cycle    ;
\draw [shift={(486.39,161.36)}, rotate = 175.99] [fill={rgb, 255:red, 0; green, 0; blue, 0 }  ][line width=0.08]  [draw opacity=0] (7.14,-3.43) -- (0,0) -- (7.14,3.43) -- (4.74,0) -- cycle    ;
\draw [shift={(479.78,46.05)}, rotate = 7.84] [fill={rgb, 255:red, 0; green, 0; blue, 0 }  ][line width=0.08]  [draw opacity=0] (7.14,-3.43) -- (0,0) -- (7.14,3.43) -- (4.74,0) -- cycle    ;
\draw  [draw opacity=0][fill={rgb, 255:red, 155; green, 155; blue, 155 }  ,fill opacity=0.47 ] (251.39,133.25) .. controls (261.9,133.28) and (270.33,141.48) .. (270.23,151.55) .. controls (270.12,161.62) and (261.51,169.76) .. (251,169.72) .. controls (240.49,169.69) and (232.06,161.49) .. (232.17,151.42) .. controls (232.28,141.35) and (240.88,133.21) .. (251.39,133.25) -- cycle ;
\draw  [fill={rgb, 255:red, 0; green, 0; blue, 0 }  ,fill opacity=1 ] (428.92,140.09) .. controls (428.92,138.61) and (430.21,137.41) .. (431.81,137.41) .. controls (433.41,137.41) and (434.71,138.61) .. (434.71,140.09) .. controls (434.71,141.57) and (433.41,142.77) .. (431.81,142.77) .. controls (430.21,142.77) and (428.92,141.57) .. (428.92,140.09) -- cycle ;
\draw  [fill={rgb, 255:red, 0; green, 0; blue, 0 }  ,fill opacity=1 ] (363.03,148.98) .. controls (363.03,147.5) and (364.33,146.3) .. (365.93,146.3) .. controls (367.53,146.3) and (368.83,147.5) .. (368.83,148.98) .. controls (368.83,150.46) and (367.53,151.66) .. (365.93,151.66) .. controls (364.33,151.66) and (363.03,150.46) .. (363.03,148.98) -- cycle ;
\draw  [draw opacity=0][fill={rgb, 255:red, 155; green, 155; blue, 155 }  ,fill opacity=0.47 ] (431.9,131.57) .. controls (442.41,131.6) and (450.85,139.8) .. (450.74,149.87) .. controls (450.63,159.94) and (442.03,168.07) .. (431.52,168.04) .. controls (421.01,168) and (412.57,159.81) .. (412.68,149.74) .. controls (412.79,139.67) and (421.39,131.53) .. (431.9,131.57) -- cycle ;
\draw  [draw opacity=0][fill={rgb, 255:red, 155; green, 155; blue, 155 }  ,fill opacity=0.47 ] (368.89,57.75) .. controls (379.4,57.78) and (387.83,65.98) .. (387.73,76.05) .. controls (387.62,86.12) and (379.01,94.26) .. (368.5,94.22) .. controls (357.99,94.19) and (349.56,85.99) .. (349.67,75.92) .. controls (349.78,65.85) and (358.38,57.71) .. (368.89,57.75) -- cycle ;
\draw  [fill={rgb, 255:red, 0; green, 0; blue, 0 }  ,fill opacity=1 ] (310.68,159.98) .. controls (310.68,158.5) and (311.98,157.3) .. (313.58,157.3) .. controls (315.18,157.3) and (316.47,158.5) .. (316.47,159.98) .. controls (316.47,161.46) and (315.18,162.66) .. (313.58,162.66) .. controls (311.98,162.66) and (310.68,161.46) .. (310.68,159.98) -- cycle ;
\draw  [color={rgb, 255:red, 0; green, 0; blue, 0 }  ,draw opacity=1 ][dash pattern={on 4.5pt off 4.5pt}] (313.89,130.81) .. controls (330.39,130.86) and (343.63,143.97) .. (343.46,160.08) .. controls (343.29,176.19) and (329.77,189.2) .. (313.27,189.15) .. controls (296.76,189.09) and (283.52,175.99) .. (283.69,159.88) .. controls (283.87,143.77) and (297.38,130.75) .. (313.89,130.81) -- cycle ;
\draw  [fill={rgb, 255:red, 0; green, 0; blue, 0 }  ,fill opacity=1 ] (373.33,67.78) .. controls (373.33,66.3) and (374.63,65.1) .. (376.23,65.1) .. controls (377.83,65.1) and (379.13,66.3) .. (379.13,67.78) .. controls (379.13,69.27) and (377.83,70.47) .. (376.23,70.47) .. controls (374.63,70.47) and (373.33,69.27) .. (373.33,67.78) -- cycle ;
\draw   (251.83,154.53) -- (255.58,158.35) -- (251.83,162.18) -- (248.08,158.35) -- cycle ;
\draw   (432.78,157.11) -- (436.53,160.94) -- (432.78,164.76) -- (429.03,160.94) -- cycle ;
\draw   (365.98,156.71) -- (369.73,160.54) -- (365.98,164.36) -- (362.23,160.54) -- cycle ;

\draw (307.16,164.24) node [anchor=north west][inner sep=0.75pt]    {$\lambda _{2}$};
\draw (355.73,127.05) node [anchor=north west][inner sep=0.75pt]    {$\lambda _{1}$};
\draw (437.23,128.05) node [anchor=north west][inner sep=0.75pt]    {$\lambda _{0}$};
\draw (363.23,42.05) node [anchor=north west][inner sep=0.75pt]    {$\lambda _{p+1}$};
\draw (259.23,189.53) node [anchor=north west][inner sep=0.75pt]    {$\mathcal{H}( X_{2})_{a,\ \mu ( X)}( \lambda _{2})$};
\draw (240.23,162.55) node [anchor=north west][inner sep=0.75pt]    {$\theta _{j_{2} -1}$};
\draw (431.03,164.34) node [anchor=north west][inner sep=0.75pt]    {$\theta _{j_{2}}$};
\draw (359.73,162.65) node [anchor=north west][inner sep=0.75pt]    {$\lambda _{1}^{'}$};

\end{tikzpicture}

\caption{Proof of Theorem \ref{Teo: minimal pancake decomposition} in the case $\operatorname{tord}(S_{j_2},S_{j_1})>\beta$. Points inside shaded disks represent arcs with the tangency order higher than the snake's exponent. The dashed disk centered in $\lambda_2$ represents the horn neighborhood $\mathcal{H}(X_2)_{a,\mu(X)}(\lambda_2)$ for $a>0$ small enough.}\label{Fig: Condition (2) of Theorem}
\end{figure}

    Finally, let us prove that $\{X_i\}_{i=1}^{p+1}$ is minimal. By the construction of the minimal sequence of $X$, the Valette link of a single pancake cannot contain the pair $N_{j_{i-1}}, N_{j_{i}}$, $i=1,\ldots, p$, since it should be LNE. Hence, any pancake of $X$ must contain at most one $N_{j_{i}}$, and therefore any pancake decomposition of $X$ has at least $p+1$ pancakes.

\end{proof}

\begin{Def}\label{Def: Greedy decomposition in Snakes}
    Any decomposition $\{X_i\}_{i=1}^{p+1}$ of a snake $X$ satisfying the properties of Theorem \ref{Teo: minimal pancake decomposition} is defined as a \textbf{greedy (pancake) decomposition of $X$.}
\end{Def}



\section{Minimal pancake decomposition of circular snakes}

Circular snakes require some adaptations to the previous algorithm we presented to obtain minimal pancake decompositions for snakes. In this section we address those changes and present suitable algorithms for this case.

\subsection{Circular snakes without nodal zones}

Differently from what happens with snakes, there are circular snakes without nodal zones. Let us start adjusting the greedy algorithm for this case. Informally, we choose an arc $\gamma$ of a circular $\beta$-snake $X$ with multiplicity $m$ and consider the $m$ LNE H\"older triangles obtained by intersecting $X$ with a $\beta$-horn neighborhood of $\gamma$, $\mathcal{H}X_{a,\beta}(\gamma)$ for $a>0$ small enough. We face $\mathcal{H}X_{a,\beta}(\gamma)$ as if it was a node (an artificial node) and its LNE H\"older triangles as if they were nodal zones (artificial nodal zones). Then, we apply a similar greedy algorithm as in the previous section to obtain a minimal pancake decomposition.

\begin{Def}\label{Def: Minimal-partition-without-nodal-zones}
    Let $X$ be a circular $\beta$-snake without nodal zones such that $X$ has multiplicity $m>1$, and let $\gamma \in V(X)$ be an arc. For a fixed orientation on the link of $X$, we define a \textbf{minimal partition} of $X$ with base $\gamma$ as the set $\{T_1,\cdots,T_{m}\}$ defined as follows: let $\gamma_1,\dots,\gamma_m, \gamma_{m+1} \in V(X)$ such that
    \begin{itemize}
        \item $\gamma_1=\gamma_{m+1}=\gamma$ and $ \gamma_1,\dots,\gamma_m, \gamma_{m+1}$ are in this order following the orientation on the link of $X$;
        \item $\itord(\gamma_i,\gamma_{i+1})=\beta$ and $\tord(\gamma_i,\gamma_{i+1})>\beta$, for $i=0,1,\dots,m$.
    \end{itemize}
Then, $T_i=T(\gamma_i,\gamma_{i+1})$, for $i=1,\dots,m$, whose orientation of the link of each $T_i$ is induced by the orientation of the link of $X$.
\end{Def}

\begin{remark}
    Since $X$ is a circular $\beta$-snake without nodal zones and with multiplicity $m$, then for $a>0$ small enough, we have that
    $\mathcal{H}X_{a,\beta}(\gamma)$ consists of $m$ LNE H\"older triangles $\tilde T_1, \dots, \tilde T_m$, with $\gamma_1 \in \tilde T_1$. Since this holds for every $a>0$ small, $\tord(\tilde T_i, \tilde T_j)>\beta$, for every $1\le i<j\le m$, and thus we can take $\gamma_2 \in \tilde T_2, \dots,\gamma_m \in \tilde T_m $ satisfying Definition \ref{Def: Minimal-partition-without-nodal-zones}.
\end{remark}

\begin{Teo}\label{Teo: minimal-circular-no-nodes}
    Let $X$ be a circular $\beta$-snake without nodal zones and let $\{T_1,\dots,T_m\}$ be a minimal partition of $X$ with base $\gamma$ (see Definition \ref{Def: Minimal-partition-without-nodal-zones}). Let $\lambda_0=\lambda_{m+1}=\gamma$ and, for $1\le i\ \le m$, let $\lambda_i \in G(T_i)$ such that $\tord(\lambda_i,\lambda_{i-1})=\beta$ and $T(\lambda_{i-1},\lambda_i)$ is LNE. If, for $i=1,\dots,m+1$, $X_i:=T(\lambda_{i-1},\lambda_i)$ (the orientation of the link of each $T_i$ is induced by the orientation of the link of $X$), then $\{X_1, \dots,X_{m+1}\}$ is a minimal pancake decomposition of $X$.
\end{Teo}

\begin{proof}
    Let $\beta'>\beta$ and let $\tilde T$ be a $\beta'$-H\"older triangle such that $\gamma \in G(\tilde T)$. If $\tilde X=X \setminus(\tilde T \setminus \{0\})$, then $\tilde X$ is a spiral snake (see Definition 4.47 in \cite{GabrielovSouza}). By applying Theorem \ref{Teo: minimal pancake decomposition} to $\tilde X$, the arcs $\lambda_1, \dots, \lambda_m$ always exist and hence $\{X_i\}_{i=1}^{m+1}$ is a pancake decomposition of $X$. Therefore, it only remains to prove that $\{X_i\}_{i=1}^{m+1}$ is minimal.
    

    We will prove first that $\{X_i\}_{i=1}^{p+1}$ is a minimal pancake decomposition such that $\gamma$ is a boundary arc of two such pancakes. Since $\tord(\gamma_i, \gamma_j)>\beta$, for every $1\le i<j\le m$, the Valette link of a single pancake cannot contain two of the $m$ arcs $\gamma_1,\dots,\gamma_m$. Since $\gamma_1=\gamma$ is in exactly two pancakes ($X_1$ and $X_{m+1}$), the minimum number of pancakes is $2+(m-1)=m+1$. The value $m+1$ is also a global minimum, because $\gamma \in V(X)$ is arbitrary. Consequently, $\{X_i\}_{i=1}^{p+1}$ is a minimal pancake decomposition of $X$.

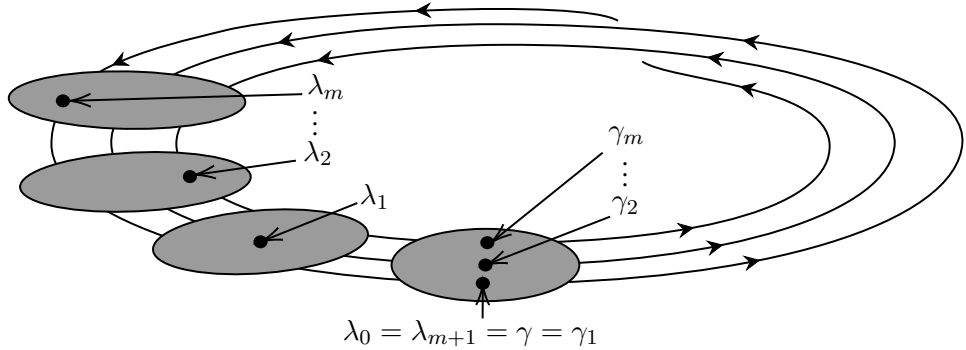
\begin{figure}[h!]
 \centering

\tikzset{every picture/.style={line width=0.75pt}} 

\begin{tikzpicture}[x=0.75pt,y=0.75pt,yscale=-1,xscale=1]

\draw    (414.66,14.73) .. controls (395.79,3.23) and (254.69,10.23) .. (216.95,19.48) .. controls (179.21,28.72) and (129.17,39.22) .. (128.94,76.84) .. controls (128.71,114.46) and (226.07,147.24) .. (356.02,147.24) .. controls (485.97,147.24) and (585.73,117.24) .. (588.36,76.44) .. controls (590.98,35.64) and (468.91,16.84) .. (354.71,16.44) .. controls (240.51,16.04) and (157.81,41.24) .. (159.13,76.84) .. controls (160.44,112.44) and (241.82,138.04) .. (356.02,138.04) .. controls (470.22,138.04) and (552.92,110.84) .. (554.23,76.84) .. controls (555.54,42.84) and (454.47,26.44) .. (357.33,26.84) .. controls (260.2,27.24) and (191.94,46.84) .. (191.94,76.84) .. controls (191.94,106.84) and (261.51,127.64) .. (359.96,126.84) .. controls (458.41,126.04) and (522.73,107.64) .. (521.41,77.24) .. controls (520.1,46.84) and (441.34,45.24) .. (426.9,35.24) ;
\draw [shift={(312.69,9.72)}, rotate = 357.55] [fill={rgb, 255:red, 0; green, 0; blue, 0 }  ][line width=0.08]  [draw opacity=0] (7.14,-3.43) -- (0,0) -- (7.14,3.43) -- (4.74,0) -- cycle    ;
\draw [shift={(159.92,37.02)}, rotate = 335.31] [fill={rgb, 255:red, 0; green, 0; blue, 0 }  ][line width=0.08]  [draw opacity=0] (7.14,-3.43) -- (0,0) -- (7.14,3.43) -- (4.74,0) -- cycle    ;
\draw [shift={(234.34,136.39)}, rotate = 191.33] [fill={rgb, 255:red, 0; green, 0; blue, 0 }  ][line width=0.08]  [draw opacity=0] (7.14,-3.43) -- (0,0) -- (7.14,3.43) -- (4.74,0) -- cycle    ;
\draw [shift={(485.95,135.6)}, rotate = 168.76] [fill={rgb, 255:red, 0; green, 0; blue, 0 }  ][line width=0.08]  [draw opacity=0] (7.14,-3.43) -- (0,0) -- (7.14,3.43) -- (4.74,0) -- cycle    ;
\draw [shift={(477.99,24.62)}, rotate = 8.25] [fill={rgb, 255:red, 0; green, 0; blue, 0 }  ][line width=0.08]  [draw opacity=0] (7.14,-3.43) -- (0,0) -- (7.14,3.43) -- (4.74,0) -- cycle    ;
\draw [shift={(243.55,25.74)}, rotate = 349.07] [fill={rgb, 255:red, 0; green, 0; blue, 0 }  ][line width=0.08]  [draw opacity=0] (7.14,-3.43) -- (0,0) -- (7.14,3.43) -- (4.74,0) -- cycle    ;
\draw [shift={(251.48,129.54)}, rotate = 190.71] [fill={rgb, 255:red, 0; green, 0; blue, 0 }  ][line width=0.08]  [draw opacity=0] (7.14,-3.43) -- (0,0) -- (7.14,3.43) -- (4.74,0) -- cycle    ;
\draw [shift={(467.32,127.9)}, rotate = 168.54] [fill={rgb, 255:red, 0; green, 0; blue, 0 }  ][line width=0.08]  [draw opacity=0] (7.14,-3.43) -- (0,0) -- (7.14,3.43) -- (4.74,0) -- cycle    ;
\draw [shift={(460.1,33.16)}, rotate = 8.16] [fill={rgb, 255:red, 0; green, 0; blue, 0 }  ][line width=0.08]  [draw opacity=0] (7.14,-3.43) -- (0,0) -- (7.14,3.43) -- (4.74,0) -- cycle    ;
\draw [shift={(263.37,34.72)}, rotate = 349.5] [fill={rgb, 255:red, 0; green, 0; blue, 0 }  ][line width=0.08]  [draw opacity=0] (7.14,-3.43) -- (0,0) -- (7.14,3.43) -- (4.74,0) -- cycle    ;
\draw [shift={(271.16,120.65)}, rotate = 189.82] [fill={rgb, 255:red, 0; green, 0; blue, 0 }  ][line width=0.08]  [draw opacity=0] (7.14,-3.43) -- (0,0) -- (7.14,3.43) -- (4.74,0) -- cycle    ;
\draw [shift={(452.59,119.2)}, rotate = 169.74] [fill={rgb, 255:red, 0; green, 0; blue, 0 }  ][line width=0.08]  [draw opacity=0] (7.14,-3.43) -- (0,0) -- (7.14,3.43) -- (4.74,0) -- cycle    ;
\draw [shift={(476.48,47.95)}, rotate = 13.84] [fill={rgb, 255:red, 0; green, 0; blue, 0 }  ][line width=0.08]  [draw opacity=0] (7.14,-3.43) -- (0,0) -- (7.14,3.43) -- (4.74,0) -- cycle    ;
\draw  [draw opacity=0][fill={rgb, 255:red, 155; green, 155; blue, 155 }  ,fill opacity=0.47 ] (300.48,137.98) .. controls (300.48,127.9) and (321.68,119.73) .. (347.83,119.73) .. controls (373.98,119.73) and (395.18,127.9) .. (395.18,137.98) .. controls (395.18,148.06) and (373.98,156.24) .. (347.83,156.24) .. controls (321.68,156.24) and (300.48,148.06) .. (300.48,137.98) -- cycle ;
\draw  [fill={rgb, 255:red, 0; green, 0; blue, 0 }  ,fill opacity=1 ] (343.45,147.24) .. controls (343.45,145.71) and (344.82,144.47) .. (346.51,144.47) .. controls (348.2,144.47) and (349.58,145.71) .. (349.58,147.24) .. controls (349.58,148.77) and (348.2,150.01) .. (346.51,150.01) .. controls (344.82,150.01) and (343.45,148.77) .. (343.45,147.24) -- cycle ;
\draw [color={rgb, 255:red, 0; green, 0; blue, 0 }  ,draw opacity=0.39 ]   (346.4,164.93) -- (346.5,149.24) ;
\draw [shift={(346.51,147.24)}, rotate = 90.37] [color={rgb, 255:red, 0; green, 0; blue, 0 }  ,draw opacity=0.39 ][line width=0.75]    (10.93,-3.29) .. controls (6.95,-1.4) and (3.31,-0.3) .. (0,0) .. controls (3.31,0.3) and (6.95,1.4) .. (10.93,3.29)   ;
\draw [color={rgb, 255:red, 0; green, 0; blue, 0 }  ,draw opacity=0.39 ]   (410.7,113.57) -- (349.69,137.26) ;
\draw [shift={(347.83,137.98)}, rotate = 338.78] [color={rgb, 255:red, 0; green, 0; blue, 0 }  ,draw opacity=0.39 ][line width=0.75]    (10.93,-3.29) .. controls (6.95,-1.4) and (3.31,-0.3) .. (0,0) .. controls (3.31,0.3) and (6.95,1.4) .. (10.93,3.29)   ;
\draw [color={rgb, 255:red, 0; green, 0; blue, 0 }  ,draw opacity=0.39 ]   (406.95,81.29) -- (352.01,125.58) ;
\draw [shift={(350.45,126.84)}, rotate = 321.12] [color={rgb, 255:red, 0; green, 0; blue, 0 }  ,draw opacity=0.39 ][line width=0.75]    (10.93,-3.29) .. controls (6.95,-1.4) and (3.31,-0.3) .. (0,0) .. controls (3.31,0.3) and (6.95,1.4) .. (10.93,3.29)   ;
\draw  [draw opacity=0][fill={rgb, 255:red, 155; green, 155; blue, 155 }  ,fill opacity=0.47 ] (201.51,115.88) .. controls (225.33,108.89) and (259.35,107.87) .. (277.5,113.6) .. controls (295.65,119.34) and (291.06,129.66) .. (267.24,136.66) .. controls (243.43,143.65) and (209.4,144.67) .. (191.25,138.94) .. controls (173.1,133.2) and (177.7,122.88) .. (201.51,115.88) -- cycle ;
\draw  [draw opacity=0][fill={rgb, 255:red, 155; green, 155; blue, 155 }  ,fill opacity=0.47 ] (152.81,82.76) .. controls (183.27,78.82) and (216.23,81.58) .. (226.44,88.92) .. controls (236.65,96.26) and (220.23,105.4) .. (189.78,109.33) .. controls (159.33,113.26) and (126.36,110.5) .. (116.16,103.17) .. controls (105.95,95.83) and (122.36,86.69) .. (152.81,82.76) -- cycle ;
\draw [color={rgb, 255:red, 0; green, 0; blue, 0 }  ,draw opacity=0.39 ]   (283.19,106.71) -- (236.24,125.53) ;
\draw [shift={(234.38,126.27)}, rotate = 338.17] [color={rgb, 255:red, 0; green, 0; blue, 0 }  ,draw opacity=0.39 ][line width=0.75]    (10.93,-3.29) .. controls (6.95,-1.4) and (3.31,-0.3) .. (0,0) .. controls (3.31,0.3) and (6.95,1.4) .. (10.93,3.29)   ;
\draw [color={rgb, 255:red, 0; green, 0; blue, 0 }  ,draw opacity=0.39 ]   (252.25,85.29) -- (202.61,92.28) ;
\draw [shift={(200.63,92.56)}, rotate = 351.98] [color={rgb, 255:red, 0; green, 0; blue, 0 }  ,draw opacity=0.39 ][line width=0.75]    (10.93,-3.29) .. controls (6.95,-1.4) and (3.31,-0.3) .. (0,0) .. controls (3.31,0.3) and (6.95,1.4) .. (10.93,3.29)   ;
\draw  [draw opacity=0][fill={rgb, 255:red, 155; green, 155; blue, 155 }  ,fill opacity=0.47 ] (174.01,40.49) .. controls (206.73,41.97) and (230.14,49.56) .. (226.31,57.44) .. controls (222.47,65.33) and (192.83,70.52) .. (160.11,69.04) .. controls (127.39,67.56) and (103.97,59.97) .. (107.81,52.08) .. controls (111.65,44.2) and (141.29,39.01) .. (174.01,40.49) -- cycle ;
\draw [color={rgb, 255:red, 0; green, 0; blue, 0 }  ,draw opacity=0.39 ]   (255.45,51.75) -- (139.81,55.07) ;
\draw [shift={(137.81,55.13)}, rotate = 358.36] [color={rgb, 255:red, 0; green, 0; blue, 0 }  ,draw opacity=0.39 ][line width=0.75]    (10.93,-3.29) .. controls (6.95,-1.4) and (3.31,-0.3) .. (0,0) .. controls (3.31,0.3) and (6.95,1.4) .. (10.93,3.29)   ;
\draw  [fill={rgb, 255:red, 0; green, 0; blue, 0 }  ,fill opacity=1 ] (344.77,137.98) .. controls (344.77,136.45) and (346.14,135.21) .. (347.83,135.21) .. controls (349.52,135.21) and (350.89,136.45) .. (350.89,137.98) .. controls (350.89,139.51) and (349.52,140.76) .. (347.83,140.76) .. controls (346.14,140.76) and (344.77,139.51) .. (344.77,137.98) -- cycle ;
\draw  [fill={rgb, 255:red, 0; green, 0; blue, 0 }  ,fill opacity=1 ] (345.72,126.71) .. controls (345.72,125.18) and (347.09,123.93) .. (348.78,123.93) .. controls (350.47,123.93) and (351.84,125.18) .. (351.84,126.71) .. controls (351.84,128.24) and (350.47,129.48) .. (348.78,129.48) .. controls (347.09,129.48) and (345.72,128.24) .. (345.72,126.71) -- cycle ;
\draw  [fill={rgb, 255:red, 0; green, 0; blue, 0 }  ,fill opacity=1 ] (231.32,126.27) .. controls (231.32,124.74) and (232.69,123.5) .. (234.38,123.5) .. controls (236.07,123.5) and (237.44,124.74) .. (237.44,126.27) .. controls (237.44,127.8) and (236.07,129.04) .. (234.38,129.04) .. controls (232.69,129.04) and (231.32,127.8) .. (231.32,126.27) -- cycle ;
\draw  [fill={rgb, 255:red, 0; green, 0; blue, 0 }  ,fill opacity=1 ] (195.72,93.37) .. controls (195.72,91.84) and (197.09,90.6) .. (198.78,90.6) .. controls (200.47,90.6) and (201.84,91.84) .. (201.84,93.37) .. controls (201.84,94.9) and (200.47,96.14) .. (198.78,96.14) .. controls (197.09,96.14) and (195.72,94.9) .. (195.72,93.37) -- cycle ;
\draw  [fill={rgb, 255:red, 0; green, 0; blue, 0 }  ,fill opacity=1 ] (131.68,55.13) .. controls (131.68,53.6) and (133.05,52.36) .. (134.74,52.36) .. controls (136.44,52.36) and (137.81,53.6) .. (137.81,55.13) .. controls (137.81,56.66) and (136.44,57.9) .. (134.74,57.9) .. controls (133.05,57.9) and (131.68,56.66) .. (131.68,55.13) -- cycle ;

\draw (274.36,164.07) node [anchor=north west][inner sep=0.75pt]    {$\lambda _{0} =\lambda _{m+1} =\gamma =\gamma _{1}$};
\draw (283,95.74) node [anchor=north west][inner sep=0.75pt]    {$\lambda _{1}$};
\draw (410.13,102.1) node [anchor=north west][inner sep=0.75pt]    {$\gamma _{2}$};
\draw (408.38,66.78) node [anchor=north west][inner sep=0.75pt]    {$\gamma _{m}$};
\draw (254.4,73.6) node [anchor=north west][inner sep=0.75pt]    {$\lambda _{2}$};
\draw (414.95,76.69) node [anchor=north west][inner sep=0.75pt]    {$\vdots $};
\draw (258.66,50.87) node [anchor=north west][inner sep=0.75pt]    {$\vdots $};
\draw (256.34,40.27) node [anchor=north west][inner sep=0.75pt]    {$\lambda _{m}$};

\end{tikzpicture}

\caption{Proof of Theorem \ref{Teo: minimal-circular-no-nodes}. Shaded disks represent the horn neighborhoods.}\label{6}
\end{figure}
\end{proof}

\begin{Def}\label{Def: Greedy decomposition in circ Snakes no nodes}
Any decomposition $\{X_i\}_{i=1}^{m+1}$ of a snake $X$ without nodal zones satisfying the properties of Theorem \ref{Teo: minimal-circular-no-nodes} is defined as a \textbf{greedy (pancake) decomposition of $X$ with base $\gamma$}.
\end{Def}

\subsection{Circular snakes with nodal zones}
When the circular snake has nodal zones, the greedy algorithm applied for snakes needs several changes. One of the main difficulties is that we do not necessarily have a periodic behavior when moving circularly along the link starting from a given nodal zone and accordingly to a given orientation. This impose the necessity of defining the notion of fundamental sequence beyond the notion of minimal sequence previously established (see Definition \ref{Def: minimal sequence-circular}). To make things worse, the periodic part of such a sequence may give ``more than one lap'' on the link of the circular snake (see Example \ref{exemplo2-circular}). To overcome this problem we proved Lemma \ref{Lem: fundamental-minimal} and in Theorem \ref{Teo:Minimal-pancake-circular-snakes-with-nodes} we construct a new circular snake $\tilde X$, which can be seen as a ``lifting'' of $X$ with $t$ ``fibers'', where $t$ is the number of laps that the fundamental sequence gives on the circular snake. Since $\tilde X$ has the property that the minimal and fundamental sequences coincide, the proof works similarly to the one in snakes' case.

\begin{Def}\label{Def: minimal sequence-circular}
    Let $X$ be a circular $\beta$-snake with given orientation $\varepsilon$ and nodal zone $N$. Let $\{N_i\}_{i=1}^{m}$ and $\{S_i\}_{i=1}^m$ be the decomposition of the Valette link of $X$ into nodal zones and segments, respectively, with $N_1=N$. Suppose that the nodal zones were enumerated according to the orientation $\varepsilon$ and $N_{i-1}$, $N_{i}$ are the nodal zones adjacent to $S_i$, for each $i \in \mathbb{Z}$, where the indices are taken modulo $m$ (see Theorem 3.23 in \cite{circular-snakes}). Consider, for each $i=1,\ldots,m$, arcs $\theta_i \in N_i$. We define the infinite sequence $\{j_1, j_2, \cdots\}$ recursively as follows: $j_1=1$ and, for each $i>1$, $j_i$ is the minimum integer greater than $j_{i-1}$ such that $T(\theta_{j_{i-1}},\theta_{j_i})$ is not LNE. Here, $T(\theta_{j_{i-1}},\theta_{j_i})$ have the orientation induced from $\varepsilon$ and $\theta_{k+m}=\theta_{k}$, for every $k\ge 1$.
    
    For each $i\ge 1$, let $\tilde j_i \in \{1,\dots,m\}$ be the only integer such that $\tilde j_i \equiv j_i \pmod{m}$. Since $j_{i}$ depends uniquely on $j_{i-1}$, the sequence $\{\tilde j_1, \tilde j_2, \cdots \}$ is eventually $p$-periodic, where $p>1$ is the fundamental period of such sequence. If $k$ is the minimum integer such that $\{\tilde j_{k+1},\cdots, \tilde j_{k+p}\}$ is the fundamental periodic block of the sequence $\{\tilde j_{1},\tilde j_{2},\cdots\}$, we define $\{\tilde j_{k+1},\cdots, \tilde j_{k+p}\}$ as the \textbf{fundamental sequence of $X$ with respect to $(N,\varepsilon)$}. If $q$ is the minimum integer such that $j_{k+q+1}\ge j_{k+1}+m$, the sequence $\{\tilde j_{k+1},\cdots, \tilde j_{k+q}\}$ is defined as the \textbf{minimal sequence of $X$ with respect to $(N,\varepsilon)$}.
\end{Def}

\begin{remark}
    Since circular snakes with at least one nodal zone have at least 2 nodes (see Corollary 3.35 in \cite{circular-snakes}), the minimal sequence $\{\tilde j_{k+1},\cdots, \tilde j_{k+q}\}$ always satisfies $q> 1$. It does not depend on the choice of the arcs $\theta_i$, but only on the nodal zones $N_i$ and thus only on the orientation of $X$ and the initial nodal zone $N_1$.
\end{remark}

We going to use the term circular snake name throughout this subsection, and for a more detailed treatment we refer the reader to \cite{circular-snakes}.
Remember that a word $W$ is \textbf{primitive} if it contains no repeated letters.

\begin{remark}\label{Rem: minimal-from-circular-name}
    As a direct consequence of Definition \ref{Def: minimal sequence-circular}, one can alternatively obtain the sequence $\{j_1, j_2, \cdots\}$ of a circular snake $X$ from its circular snake name $W=[x_1x_2\cdots x_mx_1]$ (with $x_1$ representing the node containing the nodal zone $N_1=N$, see Remark \ref{Rem:words}) as follows: consider the infinite word 
    $$\mathcal{W}=[x_1x_2\cdots x_mx_1x_2\cdots x_mx_1x_2\cdots x_m\cdots]=[y_1y_2\dots],$$
    define $j_1=1$ and, for each $i>1$, $j_i$ is the minimum integer $\ell$ greater than $j_{i-1}$ such that the subword $[y_{j_{i-1}}\cdots y_{\ell}]$ of $\mathcal{W}$ is not primitive. Recall that a word is \textbf{primitive} if it contains no repeated letters.
\end{remark} 

\begin{Exam}\label{exemplo1-circular}
    Consider the circular $\beta$-snake $X$ whose link is represented by Figures \ref{7}a, \ref{7}b and \ref{7}c. The only difference between those figures is the choice of the respective initial nodal zones $N_a, N_b, N_c$ (whose arc $\theta_1$ is indicated by $\star$) among the $m=15$ nodal zones and the respective orientations $\varepsilon_a,\varepsilon_b,\varepsilon_c$ on the link. Notice that $\varepsilon_a=\varepsilon_b=-\varepsilon_c$, $N_b=N_c$ and the corresponding arcs $\theta_{j_i}$ are indicated on each figure.

    In Figure \ref{7}a, the sequence $\{j_1,j_2,\cdots\}$ is $\{1,5,10, 13, 17, 20,25,28,32,35,\cdots\}$, and thus the fundamental sequence (and also the minimal sequence) of $X$ with respect to $(N_a,\varepsilon_a)$ is $\{5,10,13,2\}$. This example shows that $\{\tilde j_1,\tilde j_2,\cdots\}$ can be non periodic.

    In Figure \ref{7}b, the sequence $\{j_1,j_2,\cdots\}$ is $\{1,4,8, 12, 16, 19,23,27,31,\cdots\}$, and thus the fundamental sequence (and also the minimal sequence) of $X$ with respect to $(N_b,\varepsilon_b)$ is $\{1,4,8,12\}$. This example shows that $\{\tilde j_1,\tilde j_2,\cdots\}$ can be periodic. By comparing Figures \ref{7}a and \ref{7}b, we conclude that the fundamental and minimal sequences depend on the initial nodal zone $N$.

    In Figure \ref{7}c, the sequence $\{j_1,j_2,\cdots\}$ is $\{1,4,7, 12, 16, 19,22,27,31,\cdots\}$, and thus the fundamental sequence (and also the minimal sequence) of $X$ with respect to $(N_c,\varepsilon_c)$ is $\{1,4,7,12\}$. By comparing Figures \ref{7}b and \ref{7}c, we conclude that the fundamental and minimal sequences depend on the orientation $\varepsilon$, even when the nodal zone $N$ is the same.

\begin{figure}[h!]
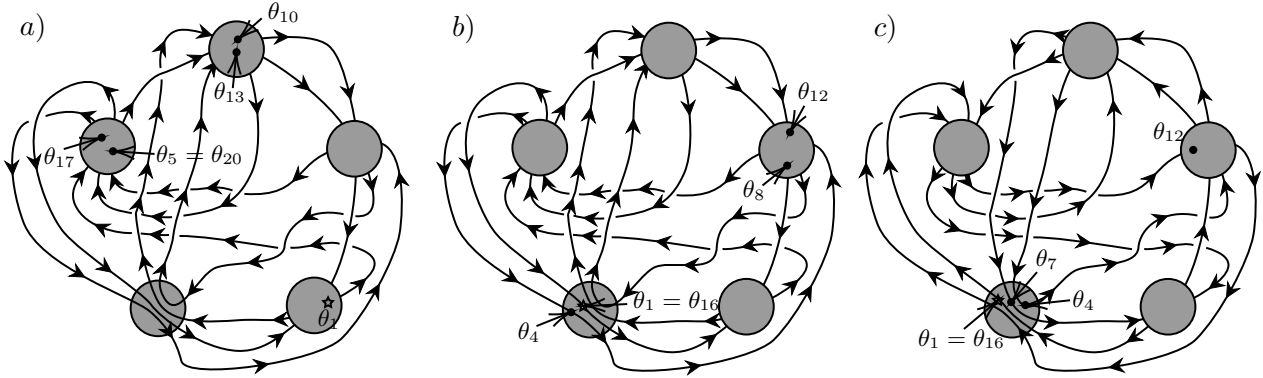

 \centering

\tikzset{every picture/.style={line width=0.75pt}} 


\caption{A circular $\beta$-snake with different initial nodal zones and orientations. Points inside shaded disks
represent arcs with the tangency order higher than the respective surface's exponent.}\label{7}
\end{figure}
\end{Exam}

\begin{Exam}\label{exemplo2-circular}
    In all three cases of Example \ref{exemplo1-circular}, the fundamental sequence is equal to the minimal sequence, but this is not true in general. For a counterexample, consider $X$ as the circular snake whose initial nodal zone $N \ni \theta_1$ and orientation $\varepsilon$ of its link are given as in Figure \ref{8}. The sequence $\{j_1,j_2,\cdots\}$ is $$\{1,4,8, 12, 16, 20,24,28,32,36,40,44,\cdots\},$$ and since $X$ has $m=10$ nodal zones, the fundamental sequence  of $X$ with respect to $(N,\varepsilon)$ is $\{4,8,2,6,10\}$. However, the minimal sequence of $X$ with respect to $(N,\varepsilon)$ is $\{4,8,2\}$.

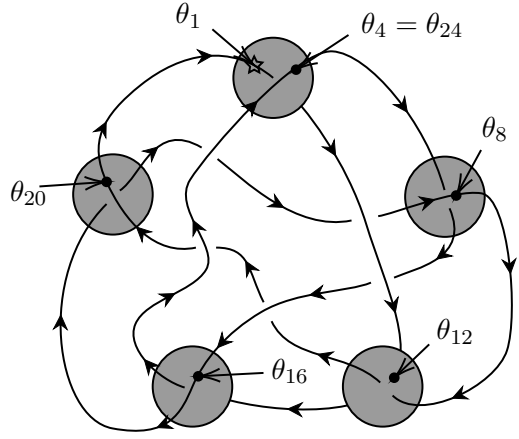
\begin{figure}[h!]
 \centering

\tikzset{every picture/.style={line width=0.75pt}} 

\begin{tikzpicture}[x=0.75pt,y=0.75pt,yscale=-1,xscale=1]

\draw    (360.03,78.27) .. controls (370.21,87.07) and (391.86,115.61) .. (401.57,147.14) .. controls (411.28,178.68) and (429.61,214.39) .. (410.95,230.46) .. controls (392.29,246.53) and (341.51,241.36) .. (318.25,231.42) ;
\draw [shift={(386.15,112.61)}, rotate = 239.48] [fill={rgb, 255:red, 0; green, 0; blue, 0 }  ][line width=0.08]  [draw opacity=0] (7.14,-3.43) -- (0,0) -- (7.14,3.43) -- (4.74,0) -- cycle    ;
\draw [shift={(416.24,191.9)}, rotate = 254.98] [fill={rgb, 255:red, 0; green, 0; blue, 0 }  ][line width=0.08]  [draw opacity=0] (7.14,-3.43) -- (0,0) -- (7.14,3.43) -- (4.74,0) -- cycle    ;
\draw [shift={(362.01,240.5)}, rotate = 2.69] [fill={rgb, 255:red, 0; green, 0; blue, 0 }  ][line width=0.08]  [draw opacity=0] (7.14,-3.43) -- (0,0) -- (7.14,3.43) -- (4.74,0) -- cycle    ;
\draw  [draw opacity=0][fill={rgb, 255:red, 155; green, 155; blue, 155 }  ,fill opacity=0.47 ] (334.6,73.37) .. controls (334.6,62.18) and (343.58,53.11) .. (354.66,53.11) .. controls (365.74,53.11) and (374.73,62.18) .. (374.73,73.37) .. controls (374.73,84.55) and (365.74,93.62) .. (354.66,93.62) .. controls (343.58,93.62) and (334.6,84.55) .. (334.6,73.37) -- cycle ;
\draw   (345.09,62.52) -- (346.37,65.19) -- (349.23,65.62) -- (347.16,67.7) -- (347.65,70.64) -- (345.09,69.25) -- (342.53,70.64) -- (343.02,67.7) -- (340.95,65.62) -- (343.81,65.19) -- cycle ;
\draw [color={rgb, 255:red, 0; green, 0; blue, 0 }  ,draw opacity=0.39 ]   (323.99,48.85) -- (343.57,65.7) ;
\draw [shift={(345.09,67.01)}, rotate = 220.72] [color={rgb, 255:red, 0; green, 0; blue, 0 }  ,draw opacity=0.39 ][line width=0.75]    (10.93,-3.29) .. controls (6.95,-1.4) and (3.31,-0.3) .. (0,0) .. controls (3.31,0.3) and (6.95,1.4) .. (10.93,3.29)   ;
\draw  [draw opacity=0][fill={rgb, 255:red, 155; green, 155; blue, 155 }  ,fill opacity=0.47 ] (253.79,132.22) .. controls (253.79,121.03) and (262.77,111.97) .. (273.85,111.97) .. controls (284.93,111.97) and (293.91,121.03) .. (293.91,132.22) .. controls (293.91,143.41) and (284.93,152.48) .. (273.85,152.48) .. controls (262.77,152.48) and (253.79,143.41) .. (253.79,132.22) -- cycle ;
\draw  [draw opacity=0][fill={rgb, 255:red, 155; green, 155; blue, 155 }  ,fill opacity=0.47 ] (421.22,133.37) .. controls (421.22,122.19) and (430.21,113.12) .. (441.29,113.12) .. controls (452.37,113.12) and (461.35,122.19) .. (461.35,133.37) .. controls (461.35,144.56) and (452.37,153.63) .. (441.29,153.63) .. controls (430.21,153.63) and (421.22,144.56) .. (421.22,133.37) -- cycle ;
\draw  [draw opacity=0][fill={rgb, 255:red, 155; green, 155; blue, 155 }  ,fill opacity=0.47 ] (293.9,229.16) .. controls (293.9,217.97) and (302.88,208.9) .. (313.96,208.9) .. controls (325.05,208.9) and (334.03,217.97) .. (334.03,229.16) .. controls (334.03,240.34) and (325.05,249.41) .. (313.96,249.41) .. controls (302.88,249.41) and (293.9,240.34) .. (293.9,229.16) -- cycle ;
\draw  [draw opacity=0][fill={rgb, 255:red, 155; green, 155; blue, 155 }  ,fill opacity=0.47 ] (389.83,229.16) .. controls (389.83,217.97) and (398.81,208.9) .. (409.89,208.9) .. controls (420.97,208.9) and (429.96,217.97) .. (429.96,229.16) .. controls (429.96,240.34) and (420.97,249.41) .. (409.89,249.41) .. controls (398.81,249.41) and (389.83,240.34) .. (389.83,229.16) -- cycle ;
\draw    (310.02,229.68) .. controls (297.09,225.83) and (279.89,209.96) .. (285.18,197.64) .. controls (290.46,185.31) and (313.88,181.77) .. (320.79,170.23) .. controls (327.69,158.69) and (303.55,139.46) .. (312.79,124.79) .. controls (322.04,110.13) and (366.31,58.94) .. (384.54,59.99) .. controls (402.78,61.04) and (431.58,95.4) .. (441.1,130.01) ;
\draw [shift={(289.7,215.57)}, rotate = 48.92] [fill={rgb, 255:red, 0; green, 0; blue, 0 }  ][line width=0.08]  [draw opacity=0] (7.14,-3.43) -- (0,0) -- (7.14,3.43) -- (4.74,0) -- cycle    ;
\draw [shift={(305.03,182.25)}, rotate = 152.36] [fill={rgb, 255:red, 0; green, 0; blue, 0 }  ][line width=0.08]  [draw opacity=0] (7.14,-3.43) -- (0,0) -- (7.14,3.43) -- (4.74,0) -- cycle    ;
\draw [shift={(314.26,145.27)}, rotate = 64.39] [fill={rgb, 255:red, 0; green, 0; blue, 0 }  ][line width=0.08]  [draw opacity=0] (7.14,-3.43) -- (0,0) -- (7.14,3.43) -- (4.74,0) -- cycle    ;
\draw [shift={(346.61,85.2)}, rotate = 135.07] [fill={rgb, 255:red, 0; green, 0; blue, 0 }  ][line width=0.08]  [draw opacity=0] (7.14,-3.43) -- (0,0) -- (7.14,3.43) -- (4.74,0) -- cycle    ;
\draw [shift={(422.33,90.71)}, rotate = 234.63] [fill={rgb, 255:red, 0; green, 0; blue, 0 }  ][line width=0.08]  [draw opacity=0] (7.14,-3.43) -- (0,0) -- (7.14,3.43) -- (4.74,0) -- cycle    ;
\draw    (443.61,139.22) .. controls (454.51,164.46) and (426.53,169.51) .. (415.63,174.2) ;
\draw [shift={(437.16,164.86)}, rotate = 322.86] [fill={rgb, 255:red, 0; green, 0; blue, 0 }  ][line width=0.08]  [draw opacity=0] (7.14,-3.43) -- (0,0) -- (7.14,3.43) -- (4.74,0) -- cycle    ;
\draw    (404,176.72) .. controls (391.64,181.77) and (367.03,180.78) .. (348.77,191.51) .. controls (330.51,202.24) and (319.49,216.32) .. (313.96,229.16) .. controls (308.43,241.99) and (302.26,255.7) .. (273.55,250.29) .. controls (244.85,244.88) and (234.31,173.12) .. (270.64,137.05) ;
\draw [shift={(373.15,183)}, rotate = 349.12] [fill={rgb, 255:red, 0; green, 0; blue, 0 }  ][line width=0.08]  [draw opacity=0] (7.14,-3.43) -- (0,0) -- (7.14,3.43) -- (4.74,0) -- cycle    ;
\draw [shift={(326.59,209.54)}, rotate = 312.53] [fill={rgb, 255:red, 0; green, 0; blue, 0 }  ][line width=0.08]  [draw opacity=0] (7.14,-3.43) -- (0,0) -- (7.14,3.43) -- (4.74,0) -- cycle    ;
\draw [shift={(296.12,249.88)}, rotate = 337.41] [fill={rgb, 255:red, 0; green, 0; blue, 0 }  ][line width=0.08]  [draw opacity=0] (7.14,-3.43) -- (0,0) -- (7.14,3.43) -- (4.74,0) -- cycle    ;
\draw [shift={(247.57,193.1)}, rotate = 92.01] [fill={rgb, 255:red, 0; green, 0; blue, 0 }  ][line width=0.08]  [draw opacity=0] (7.14,-3.43) -- (0,0) -- (7.14,3.43) -- (4.74,0) -- cycle    ;
\draw    (277.33,130.25) .. controls (288.41,122.85) and (294.63,94.5) .. (318.61,108.56) ;
\draw [shift={(296.51,108.29)}, rotate = 139.25] [fill={rgb, 255:red, 0; green, 0; blue, 0 }  ][line width=0.08]  [draw opacity=0] (7.14,-3.43) -- (0,0) -- (7.14,3.43) -- (4.74,0) -- cycle    ;
\draw    (326.6,114.33) .. controls (339.9,124.74) and (366.21,151.12) .. (394.19,145.35) ;
\draw [shift={(359.92,138.79)}, rotate = 208.89] [fill={rgb, 255:red, 0; green, 0; blue, 0 }  ][line width=0.08]  [draw opacity=0] (7.14,-3.43) -- (0,0) -- (7.14,3.43) -- (4.74,0) -- cycle    ;
\draw    (407.63,143.18) .. controls (415.08,142.28) and (445.42,129.48) .. (459.59,131.64) .. controls (473.77,133.81) and (478.42,192.88) .. (472.6,210.91) .. controls (466.79,228.94) and (429,247.69) .. (414.9,233.7) ;
\draw [shift={(436.08,134.85)}, rotate = 164.36] [fill={rgb, 255:red, 0; green, 0; blue, 0 }  ][line width=0.08]  [draw opacity=0] (7.14,-3.43) -- (0,0) -- (7.14,3.43) -- (4.74,0) -- cycle    ;
\draw [shift={(474.33,172.14)}, rotate = 264.45] [fill={rgb, 255:red, 0; green, 0; blue, 0 }  ][line width=0.08]  [draw opacity=0] (7.14,-3.43) -- (0,0) -- (7.14,3.43) -- (4.74,0) -- cycle    ;
\draw [shift={(445.06,235.06)}, rotate = 336.85] [fill={rgb, 255:red, 0; green, 0; blue, 0 }  ][line width=0.08]  [draw opacity=0] (7.14,-3.43) -- (0,0) -- (7.14,3.43) -- (4.74,0) -- cycle    ;
\draw    (408.65,227.14) .. controls (391.57,211.27) and (363.3,216.39) .. (353.86,196.19) ;
\draw [shift={(377.15,212.9)}, rotate = 17.25] [fill={rgb, 255:red, 0; green, 0; blue, 0 }  ][line width=0.08]  [draw opacity=0] (7.14,-3.43) -- (0,0) -- (7.14,3.43) -- (4.74,0) -- cycle    ;
\draw    (349.86,186.64) .. controls (341.86,170.77) and (343.32,159.05) .. (325.88,159.05) ;
\draw [shift={(340.8,165.89)}, rotate = 60.88] [fill={rgb, 255:red, 0; green, 0; blue, 0 }  ][line width=0.08]  [draw opacity=0] (7.14,-3.43) -- (0,0) -- (7.14,3.43) -- (4.74,0) -- cycle    ;
\draw    (315.7,159.05) .. controls (294.99,163.38) and (279.73,145.71) .. (270.64,125.87) .. controls (261.56,106.04) and (268.1,92.7) .. (288.09,76.83) .. controls (308.07,60.96) and (333.04,55.34) .. (354.66,73.37) ;
\draw [shift={(285.73,148.68)}, rotate = 43.38] [fill={rgb, 255:red, 0; green, 0; blue, 0 }  ][line width=0.08]  [draw opacity=0] (7.14,-3.43) -- (0,0) -- (7.14,3.43) -- (4.74,0) -- cycle    ;
\draw [shift={(270.43,95.69)}, rotate = 116.67] [fill={rgb, 255:red, 0; green, 0; blue, 0 }  ][line width=0.08]  [draw opacity=0] (7.14,-3.43) -- (0,0) -- (7.14,3.43) -- (4.74,0) -- cycle    ;
\draw [shift={(323.85,62.18)}, rotate = 176.25] [fill={rgb, 255:red, 0; green, 0; blue, 0 }  ][line width=0.08]  [draw opacity=0] (7.14,-3.43) -- (0,0) -- (7.14,3.43) -- (4.74,0) -- cycle    ;
\draw  [fill={rgb, 255:red, 0; green, 0; blue, 0 }  ,fill opacity=1 ] (364.04,69.23) .. controls (364.04,68.03) and (365.02,67.07) .. (366.22,67.07) .. controls (367.43,67.07) and (368.4,68.03) .. (368.4,69.23) .. controls (368.4,70.42) and (367.43,71.39) .. (366.22,71.39) .. controls (365.02,71.39) and (364.04,70.42) .. (364.04,69.23) -- cycle ;
\draw  [fill={rgb, 255:red, 0; green, 0; blue, 0 }  ,fill opacity=1 ] (444.78,132.77) .. controls (444.78,131.57) and (445.76,130.6) .. (446.96,130.6) .. controls (448.17,130.6) and (449.15,131.57) .. (449.15,132.77) .. controls (449.15,133.96) and (448.17,134.93) .. (446.96,134.93) .. controls (445.76,134.93) and (444.78,133.96) .. (444.78,132.77) -- cycle ;
\draw  [fill={rgb, 255:red, 0; green, 0; blue, 0 }  ,fill opacity=1 ] (413.54,224.73) .. controls (413.54,223.53) and (414.51,222.56) .. (415.72,222.56) .. controls (416.92,222.56) and (417.9,223.53) .. (417.9,224.73) .. controls (417.9,225.92) and (416.92,226.89) .. (415.72,226.89) .. controls (414.51,226.89) and (413.54,225.92) .. (413.54,224.73) -- cycle ;
\draw  [fill={rgb, 255:red, 0; green, 0; blue, 0 }  ,fill opacity=1 ] (314.7,224.01) .. controls (314.7,222.81) and (315.68,221.84) .. (316.88,221.84) .. controls (318.08,221.84) and (319.06,222.81) .. (319.06,224.01) .. controls (319.06,225.2) and (318.08,226.17) .. (316.88,226.17) .. controls (315.68,226.17) and (314.7,225.2) .. (314.7,224.01) -- cycle ;
\draw  [fill={rgb, 255:red, 0; green, 0; blue, 0 }  ,fill opacity=1 ] (268.46,125.87) .. controls (268.46,124.68) and (269.44,123.71) .. (270.64,123.71) .. controls (271.85,123.71) and (272.82,124.68) .. (272.82,125.87) .. controls (272.82,127.07) and (271.85,128.04) .. (270.64,128.04) .. controls (269.44,128.04) and (268.46,127.07) .. (268.46,125.87) -- cycle ;
\draw [color={rgb, 255:red, 0; green, 0; blue, 0 }  ,draw opacity=0.39 ]   (392.66,51.73) -- (367.89,68.13) ;
\draw [shift={(366.22,69.23)}, rotate = 326.5] [color={rgb, 255:red, 0; green, 0; blue, 0 }  ,draw opacity=0.39 ][line width=0.75]    (10.93,-3.29) .. controls (6.95,-1.4) and (3.31,-0.3) .. (0,0) .. controls (3.31,0.3) and (6.95,1.4) .. (10.93,3.29)   ;
\draw [color={rgb, 255:red, 0; green, 0; blue, 0 }  ,draw opacity=0.39 ]   (459.52,109.79) -- (447.92,131.01) ;
\draw [shift={(446.96,132.77)}, rotate = 298.65] [color={rgb, 255:red, 0; green, 0; blue, 0 }  ,draw opacity=0.39 ][line width=0.75]    (10.93,-3.29) .. controls (6.95,-1.4) and (3.31,-0.3) .. (0,0) .. controls (3.31,0.3) and (6.95,1.4) .. (10.93,3.29)   ;
\draw [color={rgb, 255:red, 0; green, 0; blue, 0 }  ,draw opacity=0.39 ]   (433.36,207.88) -- (417.16,223.35) ;
\draw [shift={(415.72,224.73)}, rotate = 316.32] [color={rgb, 255:red, 0; green, 0; blue, 0 }  ,draw opacity=0.39 ][line width=0.75]    (10.93,-3.29) .. controls (6.95,-1.4) and (3.31,-0.3) .. (0,0) .. controls (3.31,0.3) and (6.95,1.4) .. (10.93,3.29)   ;
\draw [color={rgb, 255:red, 0; green, 0; blue, 0 }  ,draw opacity=0.39 ]   (349.79,222.3) -- (318.88,223.9) ;
\draw [shift={(316.88,224.01)}, rotate = 357.04] [color={rgb, 255:red, 0; green, 0; blue, 0 }  ,draw opacity=0.39 ][line width=0.75]    (10.93,-3.29) .. controls (6.95,-1.4) and (3.31,-0.3) .. (0,0) .. controls (3.31,0.3) and (6.95,1.4) .. (10.93,3.29)   ;
\draw [color={rgb, 255:red, 0; green, 0; blue, 0 }  ,draw opacity=0.39 ]   (236.78,128.18) -- (268.65,126.01) ;
\draw [shift={(270.64,125.87)}, rotate = 176.1] [color={rgb, 255:red, 0; green, 0; blue, 0 }  ,draw opacity=0.39 ][line width=0.75]    (10.93,-3.29) .. controls (6.95,-1.4) and (3.31,-0.3) .. (0,0) .. controls (3.31,0.3) and (6.95,1.4) .. (10.93,3.29)   ;

\draw (303.98,34.08) node [anchor=north west][inner sep=0.75pt]  [font=\normalsize]  {$\theta _{1}$};
\draw (397.61,39.09) node [anchor=north west][inner sep=0.75pt]  [font=\normalsize]  {$\theta _{4} =\theta _{24}$};
\draw (458.41,91.59) node [anchor=north west][inner sep=0.75pt]  [font=\normalsize]  {$\theta _{8}$};
\draw (435.61,193.6) node [anchor=north west][inner sep=0.75pt]  [font=\normalsize]  {$\theta _{12}$};
\draw (352.3,213.67) node [anchor=north west][inner sep=0.75pt]  [font=\normalsize]  {$\theta _{16}$};
\draw (220.49,123.89) node [anchor=north west][inner sep=0.75pt]  [font=\normalsize]  {$\theta _{20}$};

\end{tikzpicture}

\caption{A circular $\beta$-snake whose fundamental sequence differs from its minimal sequence. Points inside shaded disks
represent arcs with the tangency order higher than the surface's exponent.}\label{8}
\end{figure}
\end{Exam}

In the rest of this subsection, let $X$ be a $\beta$-snake with an orientation $\varepsilon$, $N=N_1$ be a initial nodal zone. Let $p$ and $q$ be the respective lengths of the fundamental and minimal sequences of $X$, and let $\{N_i\}_{i=1}^{m}$, $\{S_i\}_{i=1}^m$ be the decomposition of $V(X)$ into nodal zones and segments, respectively, as in Definition \ref{Def: minimal sequence-circular}.

\begin{Lem}\label{Lem: fundamental-minimal}
    Let $t\in \mathbb{Z}_{>0}$ such that $j_{k+p+1}=j_{k+1}+tm$. Then, $p\ge t(q-1)+1$.
\end{Lem}

\begin{proof}
    For each $1\le s\le t-1$, let $x_s$ be the minimum integer such that $j_{k+q}+(s-1)m+1\le j_{k+x_s+1}$ and let $y_s$ be the maximum integer such that $j_{k+y_s}\le j_{k+q}+sm$. Informally, we are looking at the indices $j_i$ that appear on the $s^{\rm th}$ lap on the link of $X$ (which contains $m$ of such indices, one for each nodal zone), starting from $j_{k+q}$ and following the orientation $\varepsilon$. Such indices are precisely $j_{k+x_s+1},\dots,j_{k+y_s}$. 
    
    The key observation is that, for every $1\le r \le q-1$, there is at least one index $l\in[x_s+1,y_s]$ such that $j_{k+r}+sm < j_{k+l}<j_{k+r+1}+sm$. If that was not the case, then by discrete continuity there is $\tilde l$ such that $j_{k+\tilde l}<j_{k+r}+sm <j_{k+r+1}+sm<j_{k+\tilde l+1}$ (we cannot have equality on each inequality, because otherwise the fundamental sequence of $X$ would have length lesser than $p$). However, this contradicts the minimality of $j_{k+\tilde l+1}$ with respect to $j_{k+\tilde l}$, since $T(\theta_{j_{k+\tilde l}},\theta_{j_{k+r+1}+sm})$ is not LNE, because $$T(\theta_{j_{k+r}},\theta_{j_{k+r+1}})=T(\theta_{j_{k+r}+sm},\theta_{j_{k+r+1}+sm})\subset T(\theta_{j_{k+\tilde l}},\theta_{j_{k+r+1}+sm})$$ and $T(\theta_{j_{k+r}},\theta_{j_{k+r+1}})$ is not LNE by construction. Therefore, the fundamental sequence $\{j_{k+1},\cdots,j_{k+p}\}$ has at least $q-1$ terms $j_{k+l}$ with $l\in[x_s+1,y_s]$. Since all intervals $[x_1+1,y_1], \dots, [x_{t-1}+1,y_{t-1}]$ are disjoint, we have at least $(t-1)(q-1)$ indices $j_{k+l}$ with $l>q$ in $\{j_{k+1},\cdots,j_{k+p}\}$. Counting with $j_{k+1},\dots,j_{k+q}$, we have $p\ge (t-1)(q-1)+q=t(q-1)+1$ and the lemma follows.
\end{proof}

\begin{Teo}\label{Teo:Minimal-pancake-circular-snakes-with-nodes}
Let $X$ be a circular $\beta$-snake with nodal zones, and let $\{\tilde j_{k+1}, \dots, \tilde j_{k+q}\}$ be its minimal sequence with respect to $(N,\varepsilon)$, for some nodal zone $N$ and orientation $\varepsilon$. Consider a cyclic sequence of arcs $\{\lambda_i\}_{i=1}^{q} \subset V(X)$ such that:
\begin{enumerate}
    \item For each $1 \leq i \leq q$, we have $\lambda_i \in S_{j_{k+i}}$;
    \item The arcs $\lambda_i$ are ordered so that the H\"older triangles $X_i = T(\lambda_i, \lambda_{i+1}) \subset X$, with $\lambda_{q+1} := \lambda_0$ have the orientation induced from $X$;
    \item For $1\le i\le q$, $\mathcal{H}(X_i)_{a,\mu(X)}(\lambda_{i+1})\setminus\{0\}$ consists in only one connected component, for every $a>0$ small enough.
\end{enumerate}
Then, the collection $\{X_i = T(\lambda_i, \lambda_{i+1})\}_{i=1}^{q}$ is a minimal pancake decomposition of $X$.
\end{Teo}

\begin{proof}
Note initially that, by the same argument in the proof of Theorem \ref{Teo: minimal pancake decomposition}, the arcs $\lambda_1,\dots,\lambda_q$ satisfying all three conditions indeed exist. The proof that each $X_i$ is LNE follows analogously from the respective proof in Theorem \ref{Teo: minimal pancake decomposition}, so it remains to prove that $\{X_i\}_{i=1}^{q}$ is a minimal pancake decomposition of $X$.


We look initially at the case where $p=q$ i.e., the fundamental and minimal sequences coincide. Note that by the construction of the minimal sequence of $X$, the Valette link of a single pancake cannot intersect both nodal zones $N_{j_{k+i}}$ and $N_{j_{k+i+1}}$, for any $i = 1, \ldots, p$, since this would contradict normal embedding. Hence, each pancake can intersect at most one of the nodal zones $N_{k+j_i}$, and any pancake decomposition of $X$ must have at least $p$ pancakes. This shows that $\{X_i\}_{i=1}^{q}$ is a minimal pancake decomposition of the circular snake $X$ in this case, as $p=q$.

Consider now the case $q<p$ and let $t\in \mathbb{Z}_{>0}$ as in Lemma \ref{Lem: fundamental-minimal}. Let also $[[x_1x_2\dots x_mx_1]]$ be the circular snake name of $X$, with $x_1$ corresponding to the node containing $N=N_1$, and define $w=x_1x_2\dots x_m$. Consider $\tilde{X}$ as a circular $\beta$-snake whose circular snake name is $[[w\cdots wx_1]]$, where we concatenated $t$ copies of $w$ (such circular snake always exists, by Theorem 7.10 in \cite{circular-snakes}). Let $\{\tilde N_i\}_{i=1}^{tm}$ and $\{\tilde S_i\}_{i=1}^{tm}$ be the decomposition of the Valette link of $\tilde X$ into nodal zones and segments, respectively. Suppose that $\tilde N_{i-1}$, $\tilde N_{i}$ are the nodal zones adjacent to $\tilde S_i$, for each $i \in \mathbb{Z}$, where the indices are taken modulo $tm$. By the way we constructed its circular snake name (see Remark \ref{Rem: minimal-from-circular-name}), the fundamental and minimal sequences of $\tilde X$ coincide, and the fundamental sequence of $\tilde X$ is the same of $X$. Therefore, by the previous case, any minimal pancake decomposition of $\tilde X$ has $p$ elements. 

Suppose now that $X$ has a pancake decomposition with less than $q$ elements. Let $\{X_i'\}_{i=1}^{q-1}$ be such decomposition (one can subdivide any pancake to obtain exactly $q-1$ of them, if necessary), and suppose that $X_i'=T(\lambda_i',\lambda_{i+1}')$ have the orientation induced from $X$ (consider $\lambda_{q}'=\lambda_1'$). Suppose also that, for $i=1,\dots,q-1$, we have $\lambda_i' \in W_{s(i)}$, where $W_{s(i)}$ is either a segment $S_{s(i)}$ or a nodal zone $N_{s(i)}$ of $X$ (note that $1\le s(i)\le m)$. Now, for $i=1,\dots,t(q-1)$ with $i=r(i)+\ell(q-1)$ ($r(i)$ is the remainder of $i$ modulo $q-1$), define $\tilde{W}_i$ as either the nodal zone $\tilde{N}_{s(r(i))+\ell m}$ of $\tilde X$, if $W_{s(r(i))}$ is a nodal zone of $X$, or the segment $\tilde{S}_{s(r(i))+\ell m}$ of $\tilde X$, if $W_{r(i)}$ is a segment of $X$. Consider in $\tilde X$ the H\"older triangles $\tilde{X}_i'=T(\tilde{\lambda}_i',\tilde{\lambda}_{i+1}')$, where $\tilde{\lambda}_i' \in \tilde{W}_i$, for $1\le i\le t(q-1)$, and $\tilde{\lambda}_{t(q-1)+1}'=\tilde{\lambda}_1'$ (the orientation of $\tilde{X}_i'$ is induced from the orientation of $\tilde X$). Since $\{X'_i\}_{i=1}^{q-1}$ is a pancake decomposition of $X$, its induced decomposition $\{\tilde{X}_i'\}_{i=1}^{t(q-1)}$ on $\tilde X$ is also a pancake decomposition. This implies $p\le t(q-1)$, a contradiction with Lemma \ref{Lem: fundamental-minimal}. The result then follows.
\end{proof}

\begin{Def}\label{Def: Greedy decomposition in circ Snakes with nodes}
Any decomposition $\{X_i\}_{i=1}^{q}$ of a circular snake $X$ with nodal zones satisfying the properties of Theorem \ref{Teo:Minimal-pancake-circular-snakes-with-nodes} is defined as a \textbf{greedy (pancake) decomposition of $X$ with respect to $(N,\varepsilon)$}.
\end{Def}


\section{Weakly Outer Equivalence of Minimal Decompositions}\label{Sec: Weakly canonicity}

This section is dedicated to the canonicity of the greedy pancake decompositions seen in the previous sections. We prove that for a fixed orientation (and also a fixed nodal zone, in the case of circular snakes) any two greedy pancake decompositions of two weakly outer bi-Lipschitz snakes (or circular snakes) are also weakly bi-Lipschitz equivalent in the sense of Definition \ref{Def: equiv of pancake decomp}. Furthermore, we prove that the hypothesis on the orientation (and on the nodal zone for circular snakes)  are necessary, as shown in  Examples \ref{Exam1} and \ref{Exam2}.

\begin{Prop}\label{Prop: canonical-snakes}
    Let $X=T(\gamma_1,\gamma_2)$ and $X'=T(\gamma_1',\gamma_2')$ be $\beta$-snakes and let $h: X\to X'$ be a weakly outer bi-Lipschitz map such that $h(\gamma_1)=\gamma_1'$ and $h(\gamma_2)=\gamma_2'$. Suppose that $\{X_i\}_{i=1}^{p+1}$, $\{X_j'\}_{j=1}^{p'+1}$ are two greedy decompositions of $X$ and $X'$, respectively, such that $X_i=T(\lambda_{i-1},\lambda_i)$, for all $i=1,\dots,p+1$, and $X_j'=T(\lambda'_{j-1},\lambda_j')$, for all $j=1,\dots,p'+1$, with $\lambda_0=\gamma_1$ and $\lambda_0'=\gamma_1'$ (see Definition \ref{Def: Greedy decomposition in Snakes}). Then, $p=p'$ and $\{X_i\}_{i=1}^{p+1}$ and $\{X_i'\}_{i=1}^{p+1}$ are weakly outer bi-Lipschitz equivalent.
\end{Prop}

\begin{proof}
     It follows from Theorem 6.28 in \cite{GabrielovSouza} (See Theorem \ref{Teo:weak equivalence}) that, with the orientations from $\gamma_1$ to $\gamma_2$ for $X$, and from $\gamma'_1$ to $\gamma'_2$ for $X'$, we have a one-to-one correspondence between the nodes, nodal zones, segments and the clusters of cluster partitions of $X$ and $X'$, respectively. Therefore, by the construction in the definition of the minimal sequences of $X$ and $X'$, we necessarily obtain the boundary arcs $\lambda_i$ and $\lambda'_i$ in corresponding segments of $X$ and $X'$, respectively, with respect to $h$. Hence, $p=p'$. 

    Now, we can slightly adapt the proof of Gabrielov and Souza for Theorem 6.28 in \cite{GabrielovSouza} to construct a weakly outer bi-Lipschitz homeomorphism $\tilde h \colon X \to X'$ such that $\tilde h(X_i) = X'_i$, for $i=1,\ldots, p+1$. Let us assume first that $X$ and $X'$ are not bubbles or spiral snakes. Consider the boundary arcs of $X$ and $X'$ and choose one arc in each interior nodal zone of $X$ and $X'$, respectively, and consider pancake decompositions for $X$ and $X'$ such that those arcs are the boundary arcs of the pancakes. Enumerate the chosen arcs accordingly to the orientations which provide the correspondence between the nodes, nodal zones, segments and the clusters in the cluster partitions of $X$ and $X'$. If $\theta_0,\theta_1,\dots,\theta_n$ are the arcs in $X$ and $\theta_0',\theta_1',\dots,\theta_n'$ are the arcs in $X'$ following such enumeration (with $\theta_0=\gamma_1$, $\theta_n=\gamma_2$; $\theta_0'=\gamma_1'$, $\theta_n'=\gamma_2$ and $\theta_i'$ is on the same nodal zone corresponding to the nodal zone containing $\theta_i$), let us denote by $\{P_i=T(\theta_{i-1},\theta_i)\}_{i=1}^{n}$ and $\{P_i'=T(\theta_{i-1}',\theta_i')\}_{i=1}^{n}$ the mentioned pancake decompositions for $X$ and $X'$, respectively. It follows from Proposition 4.56 in \cite{GabrielovSouza} that each segment $X$ (resp., $X'$) correspond to one of the pancakes of $P_i$ (resp., $P'_i$). Finally, they provide a weakly outer bi-Lipschitz homeomorphism for each pair of correspondent pancakes $P_i$ and $P'_i$. The desired weakly outer equivalence between $X$ and $X'$ comes from the gluing of those homeomorphisms for each pair of corresponding pancakes.
    
    Since we already know that the boundary arcs of $\{X_i\}_{i=1}^{p+1}$ and $\{X_i'\}_{i=1}^{p+1}$ are in corresponding segments, the adaptation we need is the following. We keep their construction for $\tilde h$ in the cases where $P_i$ and $P'_i$ do not contain boundary arcs of $\{X_i\}_{i=1}^{p+1}$ and $\{X_i'\}_{i=1}^{p+1}$, respectively, as interior arcs. We shall address the cases where some $\lambda_j \in I(P_{i_j})$ and $\lambda'_j \in I(P'_{i_j})$. Assuming that $P_{i_j} = T(\theta_{i_j - 1}, \theta_{i_j})$ and $P'_{i_j} = T(\theta'_{i_j - 1}, \theta'_{i_j})$, we define weakly outer bi-Lipschitz homeomorphisms from $T(\theta_{i_j - 1}, \lambda_i)$ to $T(\theta'_{i_j - 1}, \lambda'_i)$ and from $T(\lambda_i, \theta_{i_j})$ to $T(\lambda'_i, \theta'_{i_j})$, mapping $\lambda_i$ to $\lambda'_i$. 
    Finally, we define $\tilde h \colon P_i \to P'_i$ as the natural gluing of those two homeomorphisms. This now provides the the desired weakly outer homeomorphism $\tilde h\colon X \to X'$ such that $\tilde h(X_i) = X'_i$, for $i=1,\ldots, p+1$. Thus, $\{X_i\}_{i=1}^{p+1}$ and $\{X_i'\}_{i=1}^{p+1}$ are weakly outer bi-Lipschitz equivalent. 
    
    The case where $X$ and $X'$ are bubbles or spiral snakes admits the same adaptation to the pancakes considered for this case. The only difference is that one must take the arcs $\theta_i$ in the connected components of $(X\cap\mathcal{H}_{a,\beta}(\gamma_1))\setminus\{0\}$ ($a>0$ small enough) instead in nodal zones.
\end{proof}

\begin{Prop}\label{Prop: canonical-circular-snakes-with-nodes}
    Let $X$ and $X'$ be circular $\beta$-snakes with respective nodal zones $N$, $N'$ and orientations $\varepsilon$, $\varepsilon'$. Let $h: X\to X'$ be a weakly outer bi-Lipschitz map such that $h(N)=N'$ and the orientation of $h(X')$ is $\varepsilon'$. Suppose that $\{X_i\}_{i=1}^{q}$, $\{X_i'\}_{i=1}^{q'}$ are the greedy decompositions of $X$ and $X'$ with respect to $(N,\varepsilon)$ and $(N',\varepsilon')$, respectively (see Definition \ref{Def: Greedy decomposition in circ Snakes with nodes}). Then, $q=q'$ and $\{X_i\}_{i=1}^{q}$, $\{X_i'\}_{i=1}^{q}$ are weakly outer bi-Lipschitz equivalent.
\end{Prop}

\begin{proof}
    It follows from Theorem 8.3 in \cite{circular-snakes} (See Theorem \ref{Teo:weak equivalence}) that, starting from $N$ and $N'$ and following the orientations in $X$ and $X'$, respectively, we have a one-to-one correspondence between the nodes, nodal zones, segments and the clusters of cluster partitions of $X$ and $X'$. Therefore, by the construction in the definition of the fundamental and minimal sequences of $X$ and $X'$, we necessarily obtain the boundary arcs $\lambda_i$ and $\lambda'_i$ in corresponding segments of $X$ and $X'$, respectively, with respect to $h$. Hence, their fundamental and minimal sequences are the same. In particular, $q=q'$. 

    In the same fashion as the proof of Proposition \ref{Prop: canonical-snakes}, choose one arc in each interior nodal zone of $X$ and $X'$, respectively, and consider pancake decompositions for $X$ and $X'$ such that those arcs are the boundary arcs of the pancakes. Enumerate the chosen arcs accordingly to the orientations and the initial nodal zones which provide the correspondence between the nodes, nodal zones, segments, and the clusters in the cluster partitions of $X$ and $X'$. If $\theta_1,\theta_1,\dots$ are the arcs in $X$ and $\theta_0',\theta_1',\dots,$ are the arcs in $X'$ following such enumeration (with $\theta_1 \in N$, $\theta_1' \in N'$, $\theta_i'$ is on the same nodal zone corresponding to the nodal zone containing $\theta_i$ and $\theta_{i+m}=\theta_{i}$, $\theta_{i+m}'=\theta_{i}'$, for all $i$, where $m$ is the number of nodal zones of $X$ and $X'$), let us denote by $\{P_i=T(\theta_{i-1},\theta_i)\}_{i=1}^{m}$ and $\{P_i'=T(\theta_{i-1}',\theta_i')\}_{i=1}^{m}$ the mentioned pancake decompositions for $X$ and $X'$, respectively. It follows from Proposition 4.56 in \cite{GabrielovSouza} that each segment $X$ (resp., $X'$) correspond to one of the pancakes of $P_i$ (resp., $P'_i$). Finally, they provide a weakly outer bi-Lipschitz homeomorphism for each pair of correspondent pancakes $P_i$ and $P'_i$. The desired weakly outer bi-Lipschitz map between $X$ and $X'$ comes from the gluing of those homeomorphisms for each pair of corresponding pancakes. Since we already know that the boundary arcs of $\{X_i\}_{i=1}^{q}$ and $\{X_i'\}_{i=1}^{q}$ are in corresponding segments, the rest of the proof is the same as the proof of Proposition \ref{Prop: canonical-snakes}. 
\end{proof}

\begin{Prop}\label{Prop: canonicity for circular snakes with nodal zones}
    Let $X$ and $X'$ be circular $\beta$-snakes without nodal zones with the same multiplicity $m$. If $\{X_i\}_{i=1}^{m+1}$ and $\{X_i'\}_{i=1}^{m+1}$ are greedy decompositions of $X$ and $X'$, respectively (see Definition \ref{Def: Greedy decomposition in circ Snakes no nodes}), then $\{X_i\}_{i=1}^{m+1}$ and $\{X_i'\}_{i=1}^{m+1}$ are weakly outer bi-Lipschitz equivalent.
\end{Prop}

\begin{proof}
    Suppose that $\{X_i\}_{i=1}^{m+1}$ and $\{X_i'\}_{i=1}^{m+1}$ are greedy decompositions of $X$ and $X'$ with bases $\gamma \in V(X)$ and $\gamma'\in V(X')$, respectively. Since $X$ and $X'$ are circular $\beta$-snakes without nodal zones and have the same multiplicity $m$, by Theorem 8.4 in \cite{circular-snakes}, there is a weakly outer bi-Lipschitz map $h:X \to X'$ such that $h(\gamma)=\gamma'$ and $h$ induces in $X'$ the same orientation of $X_1=T(\gamma_1,\gamma_2)$. By taking arcs $\theta_i$ in the respective connected components $T_i$ of $(X\cap\mathcal{H}_{a,\beta}(\gamma))\setminus\{0\}$ ($a>0$ small enough), such that the link of $T_1,\dots,T_m$ are in this order on the link of $X$ (following the given orientation in $X$), $\theta_1=\gamma$ and $\theta_{m+1}=\theta_1$, the proof now follows analogously as the proof of Proposition \ref{Prop: canonical-snakes} for spiral snakes.
\end{proof}

\begin{Exam}\label{Exam1}
 Let $X=T(\gamma_1,\gamma_2)$ be a $\beta$-snake oriented from $\gamma_1$ to $\gamma_2$, whose snake name is $W=[abacdbcd]$. Let $\{X_i\}_{i=1}^3$, with $X_i=T(\lambda_{i-1},\lambda_i)$, be the greedy decomposition of $X$ (see Figure \ref{9}a). Notice that $X_1$, $X_2$, and $X_3$ contain 2, 4, and 2 nodal zones, respectively. On the other hand, let $X'=T(\gamma_2,\gamma_1)$ be the same $\beta$-snake, but now oriented from $\gamma_2$ to $\gamma_1$. Let $\{X_i'\}_{i=1}^3$, with $X_i'=T(\lambda_{i-1}',\lambda_i')$, be the greedy decomposition of $X'$ (see Figure \ref{9}b). Notice that $X_1'$, $X_2'$, and $X_3'$ contain 3, 4, and 1 nodal zones, respectively. 

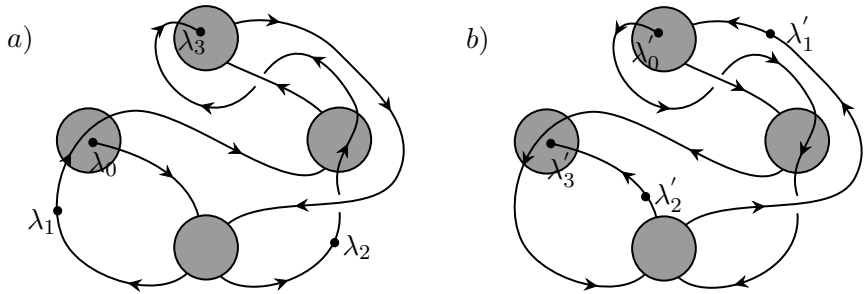
\begin{figure}[h!]
 \centering

\tikzset{every picture/.style={line width=0.75pt}} 

\begin{tikzpicture}[x=0.75pt,y=0.75pt,yscale=-0.8,xscale=0.8]

\draw  [draw opacity=0][fill={rgb, 255:red, 155; green, 155; blue, 155 }  ,fill opacity=0.47 ] (113.6,102.87) .. controls (113.6,91.68) and (122.58,82.61) .. (133.66,82.61) .. controls (144.74,82.61) and (153.73,91.68) .. (153.73,102.87) .. controls (153.73,114.05) and (144.74,123.12) .. (133.66,123.12) .. controls (122.58,123.12) and (113.6,114.05) .. (113.6,102.87) -- cycle ;
\draw    (136.3,104.15) .. controls (171.3,111.65) and (209.8,130.65) .. (203.8,170.65) .. controls (197.8,210.65) and (131.3,194.15) .. (117.8,165.15) .. controls (104.3,136.15) and (123.8,86.15) .. (161.8,84.15) .. controls (199.8,82.15) and (273.8,154.15) .. (288.3,111.15) .. controls (302.8,68.15) and (191.8,57.15) .. (210.3,38.15) .. controls (228.8,19.15) and (270.3,26.15) .. (289.8,46.15) .. controls (309.3,66.15) and (344.8,94.65) .. (327.8,126.15) .. controls (310.8,157.65) and (221.8,127.65) .. (212.3,170.15) .. controls (202.8,212.65) and (295.8,198.15) .. (291.8,148.15) ;
\draw [shift={(186.6,126.19)}, rotate = 218.89] [fill={rgb, 255:red, 0; green, 0; blue, 0 }  ][line width=0.08]  [draw opacity=0] (7.14,-3.43) -- (0,0) -- (7.14,3.43) -- (4.74,0) -- cycle    ;
\draw [shift={(159.03,192.95)}, rotate = 10.9] [fill={rgb, 255:red, 0; green, 0; blue, 0 }  ][line width=0.08]  [draw opacity=0] (7.14,-3.43) -- (0,0) -- (7.14,3.43) -- (4.74,0) -- cycle    ;
\draw [shift={(122.25,111.14)}, rotate = 117.71] [fill={rgb, 255:red, 0; green, 0; blue, 0 }  ][line width=0.08]  [draw opacity=0] (7.14,-3.43) -- (0,0) -- (7.14,3.43) -- (4.74,0) -- cycle    ;
\draw [shift={(230.08,111.13)}, rotate = 207.04] [fill={rgb, 255:red, 0; green, 0; blue, 0 }  ][line width=0.08]  [draw opacity=0] (7.14,-3.43) -- (0,0) -- (7.14,3.43) -- (4.74,0) -- cycle    ;
\draw [shift={(252.05,67.57)}, rotate = 25.11] [fill={rgb, 255:red, 0; green, 0; blue, 0 }  ][line width=0.08]  [draw opacity=0] (7.14,-3.43) -- (0,0) -- (7.14,3.43) -- (4.74,0) -- cycle    ;
\draw [shift={(253.76,28.15)}, rotate = 189.22] [fill={rgb, 255:red, 0; green, 0; blue, 0 }  ][line width=0.08]  [draw opacity=0] (7.14,-3.43) -- (0,0) -- (7.14,3.43) -- (4.74,0) -- cycle    ;
\draw [shift={(323.83,83.88)}, rotate = 236.38] [fill={rgb, 255:red, 0; green, 0; blue, 0 }  ][line width=0.08]  [draw opacity=0] (7.14,-3.43) -- (0,0) -- (7.14,3.43) -- (4.74,0) -- cycle    ;
\draw [shift={(264.18,144.24)}, rotate = 355.93] [fill={rgb, 255:red, 0; green, 0; blue, 0 }  ][line width=0.08]  [draw opacity=0] (7.14,-3.43) -- (0,0) -- (7.14,3.43) -- (4.74,0) -- cycle    ;
\draw [shift={(259.02,191.25)}, rotate = 161.14] [fill={rgb, 255:red, 0; green, 0; blue, 0 }  ][line width=0.08]  [draw opacity=0] (7.14,-3.43) -- (0,0) -- (7.14,3.43) -- (4.74,0) -- cycle    ;
\draw  [draw opacity=0][fill={rgb, 255:red, 155; green, 155; blue, 155 }  ,fill opacity=0.47 ] (187.6,170.37) .. controls (187.6,159.18) and (196.58,150.11) .. (207.66,150.11) .. controls (218.74,150.11) and (227.73,159.18) .. (227.73,170.37) .. controls (227.73,181.55) and (218.74,190.62) .. (207.66,190.62) .. controls (196.58,190.62) and (187.6,181.55) .. (187.6,170.37) -- cycle ;
\draw  [draw opacity=0][fill={rgb, 255:red, 155; green, 155; blue, 155 }  ,fill opacity=0.47 ] (187.6,38.37) .. controls (187.6,27.18) and (196.58,18.11) .. (207.66,18.11) .. controls (218.74,18.11) and (227.73,27.18) .. (227.73,38.37) .. controls (227.73,49.55) and (218.74,58.62) .. (207.66,58.62) .. controls (196.58,58.62) and (187.6,49.55) .. (187.6,38.37) -- cycle ;
\draw  [draw opacity=0][fill={rgb, 255:red, 155; green, 155; blue, 155 }  ,fill opacity=0.47 ] (271.6,102.37) .. controls (271.6,91.18) and (280.58,82.11) .. (291.66,82.11) .. controls (302.74,82.11) and (311.73,91.18) .. (311.73,102.37) .. controls (311.73,113.55) and (302.74,122.62) .. (291.66,122.62) .. controls (280.58,122.62) and (271.6,113.55) .. (271.6,102.37) -- cycle ;
\draw    (291.8,137.15) .. controls (283.6,111.3) and (310.3,97.65) .. (301.3,79.65) .. controls (292.3,61.65) and (267.3,30.65) .. (244.3,60.65) ;
\draw [shift={(296.85,105.46)}, rotate = 116.6] [fill={rgb, 255:red, 0; green, 0; blue, 0 }  ][line width=0.08]  [draw opacity=0] (7.14,-3.43) -- (0,0) -- (7.14,3.43) -- (4.74,0) -- cycle    ;
\draw [shift={(274.61,50.32)}, rotate = 27.08] [fill={rgb, 255:red, 0; green, 0; blue, 0 }  ][line width=0.08]  [draw opacity=0] (7.14,-3.43) -- (0,0) -- (7.14,3.43) -- (4.74,0) -- cycle    ;
\draw    (238.8,68.65) .. controls (216.8,95.15) and (188.8,77.65) .. (179.8,59.65) .. controls (170.8,41.65) and (175.8,15.65) .. (205.8,35.65) ;
\draw [shift={(204.43,80.84)}, rotate = 14.25] [fill={rgb, 255:red, 0; green, 0; blue, 0 }  ][line width=0.08]  [draw opacity=0] (7.14,-3.43) -- (0,0) -- (7.14,3.43) -- (4.74,0) -- cycle    ;
\draw [shift={(181.27,30.15)}, rotate = 137.81] [fill={rgb, 255:red, 0; green, 0; blue, 0 }  ][line width=0.08]  [draw opacity=0] (7.14,-3.43) -- (0,0) -- (7.14,3.43) -- (4.74,0) -- cycle    ;
\draw  [fill={rgb, 255:red, 0; green, 0; blue, 0 }  ,fill opacity=1 ] (134.06,104.15) .. controls (134.06,102.86) and (135.06,101.81) .. (136.3,101.81) .. controls (137.54,101.81) and (138.54,102.86) .. (138.54,104.15) .. controls (138.54,105.44) and (137.54,106.49) .. (136.3,106.49) .. controls (135.06,106.49) and (134.06,105.44) .. (134.06,104.15) -- cycle ;
\draw  [fill={rgb, 255:red, 0; green, 0; blue, 0 }  ,fill opacity=1 ] (111.84,147.26) .. controls (111.84,145.97) and (112.84,144.92) .. (114.08,144.92) .. controls (115.32,144.92) and (116.32,145.97) .. (116.32,147.26) .. controls (116.32,148.55) and (115.32,149.6) .. (114.08,149.6) .. controls (112.84,149.6) and (111.84,148.55) .. (111.84,147.26) -- cycle ;
\draw  [fill={rgb, 255:red, 0; green, 0; blue, 0 }  ,fill opacity=1 ] (286.46,167.35) .. controls (286.46,166.06) and (287.46,165.01) .. (288.7,165.01) .. controls (289.94,165.01) and (290.94,166.06) .. (290.94,167.35) .. controls (290.94,168.64) and (289.94,169.69) .. (288.7,169.69) .. controls (287.46,169.69) and (286.46,168.64) .. (286.46,167.35) -- cycle ;
\draw  [fill={rgb, 255:red, 0; green, 0; blue, 0 }  ,fill opacity=1 ] (201.96,34.35) .. controls (201.96,33.06) and (202.96,32.01) .. (204.2,32.01) .. controls (205.44,32.01) and (206.44,33.06) .. (206.44,34.35) .. controls (206.44,35.64) and (205.44,36.69) .. (204.2,36.69) .. controls (202.96,36.69) and (201.96,35.64) .. (201.96,34.35) -- cycle ;
\draw  [draw opacity=0][fill={rgb, 255:red, 155; green, 155; blue, 155 }  ,fill opacity=0.47 ] (402.26,103.53) .. controls (402.26,92.35) and (411.25,83.28) .. (422.33,83.28) .. controls (433.41,83.28) and (442.39,92.35) .. (442.39,103.53) .. controls (442.39,114.72) and (433.41,123.79) .. (422.33,123.79) .. controls (411.25,123.79) and (402.26,114.72) .. (402.26,103.53) -- cycle ;
\draw    (424.97,104.82) .. controls (459.97,112.32) and (498.47,131.32) .. (492.47,171.32) .. controls (486.47,211.32) and (419.97,194.82) .. (406.47,165.82) .. controls (392.97,136.82) and (412.47,86.82) .. (450.47,84.82) .. controls (488.47,82.82) and (562.47,154.82) .. (576.97,111.82) .. controls (591.47,68.82) and (480.47,57.82) .. (498.97,38.82) .. controls (517.47,19.82) and (558.97,26.82) .. (578.47,46.82) .. controls (597.97,66.82) and (633.47,95.32) .. (616.47,126.82) .. controls (599.47,158.32) and (510.47,128.32) .. (500.97,170.82) .. controls (491.47,213.32) and (584.47,198.82) .. (580.47,148.82) ;
\draw [shift={(469.97,122.79)}, rotate = 36] [fill={rgb, 255:red, 0; green, 0; blue, 0 }  ][line width=0.08]  [draw opacity=0] (7.14,-3.43) -- (0,0) -- (7.14,3.43) -- (4.74,0) -- cycle    ;
\draw [shift={(453.99,194.67)}, rotate = 187.96] [fill={rgb, 255:red, 0; green, 0; blue, 0 }  ][line width=0.08]  [draw opacity=0] (7.14,-3.43) -- (0,0) -- (7.14,3.43) -- (4.74,0) -- cycle    ;
\draw [shift={(407.8,118.17)}, rotate = 294.5] [fill={rgb, 255:red, 0; green, 0; blue, 0 }  ][line width=0.08]  [draw opacity=0] (7.14,-3.43) -- (0,0) -- (7.14,3.43) -- (4.74,0) -- cycle    ;
\draw [shift={(512.63,108.66)}, rotate = 27.37] [fill={rgb, 255:red, 0; green, 0; blue, 0 }  ][line width=0.08]  [draw opacity=0] (7.14,-3.43) -- (0,0) -- (7.14,3.43) -- (4.74,0) -- cycle    ;
\draw [shift={(547.02,71.26)}, rotate = 206.09] [fill={rgb, 255:red, 0; green, 0; blue, 0 }  ][line width=0.08]  [draw opacity=0] (7.14,-3.43) -- (0,0) -- (7.14,3.43) -- (4.74,0) -- cycle    ;
\draw [shift={(536.07,27.95)}, rotate = 6.34] [fill={rgb, 255:red, 0; green, 0; blue, 0 }  ][line width=0.08]  [draw opacity=0] (7.14,-3.43) -- (0,0) -- (7.14,3.43) -- (4.74,0) -- cycle    ;
\draw [shift={(608.57,78.88)}, rotate = 54.07] [fill={rgb, 255:red, 0; green, 0; blue, 0 }  ][line width=0.08]  [draw opacity=0] (7.14,-3.43) -- (0,0) -- (7.14,3.43) -- (4.74,0) -- cycle    ;
\draw [shift={(559.74,144.44)}, rotate = 176.23] [fill={rgb, 255:red, 0; green, 0; blue, 0 }  ][line width=0.08]  [draw opacity=0] (7.14,-3.43) -- (0,0) -- (7.14,3.43) -- (4.74,0) -- cycle    ;
\draw [shift={(541.24,193.9)}, rotate = 344.3] [fill={rgb, 255:red, 0; green, 0; blue, 0 }  ][line width=0.08]  [draw opacity=0] (7.14,-3.43) -- (0,0) -- (7.14,3.43) -- (4.74,0) -- cycle    ;
\draw  [draw opacity=0][fill={rgb, 255:red, 155; green, 155; blue, 155 }  ,fill opacity=0.47 ] (476.26,171.03) .. controls (476.26,159.85) and (485.25,150.78) .. (496.33,150.78) .. controls (507.41,150.78) and (516.39,159.85) .. (516.39,171.03) .. controls (516.39,182.22) and (507.41,191.29) .. (496.33,191.29) .. controls (485.25,191.29) and (476.26,182.22) .. (476.26,171.03) -- cycle ;
\draw  [draw opacity=0][fill={rgb, 255:red, 155; green, 155; blue, 155 }  ,fill opacity=0.47 ] (476.26,39.03) .. controls (476.26,27.85) and (485.25,18.78) .. (496.33,18.78) .. controls (507.41,18.78) and (516.39,27.85) .. (516.39,39.03) .. controls (516.39,50.22) and (507.41,59.29) .. (496.33,59.29) .. controls (485.25,59.29) and (476.26,50.22) .. (476.26,39.03) -- cycle ;
\draw  [draw opacity=0][fill={rgb, 255:red, 155; green, 155; blue, 155 }  ,fill opacity=0.47 ] (560.26,103.03) .. controls (560.26,91.85) and (569.25,82.78) .. (580.33,82.78) .. controls (591.41,82.78) and (600.39,91.85) .. (600.39,103.03) .. controls (600.39,114.22) and (591.41,123.29) .. (580.33,123.29) .. controls (569.25,123.29) and (560.26,114.22) .. (560.26,103.03) -- cycle ;
\draw    (580.47,137.82) .. controls (572.27,111.97) and (598.97,98.32) .. (589.97,80.32) .. controls (580.97,62.32) and (555.97,31.32) .. (532.97,61.32) ;
\draw [shift={(582.61,112.16)}, rotate = 294.73] [fill={rgb, 255:red, 0; green, 0; blue, 0 }  ][line width=0.08]  [draw opacity=0] (7.14,-3.43) -- (0,0) -- (7.14,3.43) -- (4.74,0) -- cycle    ;
\draw [shift={(568.84,54.21)}, rotate = 213.04] [fill={rgb, 255:red, 0; green, 0; blue, 0 }  ][line width=0.08]  [draw opacity=0] (7.14,-3.43) -- (0,0) -- (7.14,3.43) -- (4.74,0) -- cycle    ;
\draw    (527.47,69.32) .. controls (505.47,95.82) and (477.47,78.32) .. (468.47,60.32) .. controls (459.47,42.32) and (464.47,16.32) .. (494.47,36.32) ;
\draw [shift={(499.79,82.78)}, rotate = 187.33] [fill={rgb, 255:red, 0; green, 0; blue, 0 }  ][line width=0.08]  [draw opacity=0] (7.14,-3.43) -- (0,0) -- (7.14,3.43) -- (4.74,0) -- cycle    ;
\draw [shift={(465.91,35.86)}, rotate = 299.31] [fill={rgb, 255:red, 0; green, 0; blue, 0 }  ][line width=0.08]  [draw opacity=0] (7.14,-3.43) -- (0,0) -- (7.14,3.43) -- (4.74,0) -- cycle    ;
\draw  [fill={rgb, 255:red, 0; green, 0; blue, 0 }  ,fill opacity=1 ] (422.73,104.82) .. controls (422.73,103.53) and (423.73,102.48) .. (424.97,102.48) .. controls (426.2,102.48) and (427.21,103.53) .. (427.21,104.82) .. controls (427.21,106.11) and (426.2,107.15) .. (424.97,107.15) .. controls (423.73,107.15) and (422.73,106.11) .. (422.73,104.82) -- cycle ;
\draw  [fill={rgb, 255:red, 0; green, 0; blue, 0 }  ,fill opacity=1 ] (482.5,138.33) .. controls (482.5,137.04) and (483.51,135.99) .. (484.74,135.99) .. controls (485.98,135.99) and (486.99,137.04) .. (486.99,138.33) .. controls (486.99,139.62) and (485.98,140.67) .. (484.74,140.67) .. controls (483.51,140.67) and (482.5,139.62) .. (482.5,138.33) -- cycle ;
\draw  [fill={rgb, 255:red, 0; green, 0; blue, 0 }  ,fill opacity=1 ] (561.18,35.62) .. controls (561.18,34.33) and (562.19,33.28) .. (563.43,33.28) .. controls (564.66,33.28) and (565.67,34.33) .. (565.67,35.62) .. controls (565.67,36.91) and (564.66,37.95) .. (563.43,37.95) .. controls (562.19,37.95) and (561.18,36.91) .. (561.18,35.62) -- cycle ;
\draw  [fill={rgb, 255:red, 0; green, 0; blue, 0 }  ,fill opacity=1 ] (490.63,35.02) .. controls (490.63,33.73) and (491.63,32.68) .. (492.87,32.68) .. controls (494.1,32.68) and (495.11,33.73) .. (495.11,35.02) .. controls (495.11,36.31) and (494.1,37.35) .. (492.87,37.35) .. controls (491.63,37.35) and (490.63,36.31) .. (490.63,35.02) -- cycle ;

\draw (131.72,108) node [anchor=north west][inner sep=0.75pt]    {$\lambda _{0}$};
\draw (92.92,144.4) node [anchor=north west][inner sep=0.75pt]    {$\lambda _{1}$};
\draw (291.72,160.4) node [anchor=north west][inner sep=0.75pt]    {$\lambda _{2}$};
\draw (187.72,31.2) node [anchor=north west][inner sep=0.75pt]    {$\lambda _{3}$};
\draw (80.33,28.07) node [anchor=north west][inner sep=0.75pt]    {$a)$};

\draw (420.39,108.67) node [anchor=north west][inner sep=0.75pt]    {$\lambda _{3}^{'}$};
\draw (487.99,124.27) node [anchor=north west][inner sep=0.75pt]    {$\lambda _{2}^{'}$};
\draw (571.05,22.8) node [anchor=north west][inner sep=0.75pt]    {$\lambda _{1}^{'}$};
\draw (473.72,34.53) node [anchor=north west][inner sep=0.75pt]    {$\lambda _{0}^{'}$};
\draw (369.67,28.73) node [anchor=north west][inner sep=0.75pt]    {$b)$};

\end{tikzpicture}

\caption{Greedy decompositions that are not weakly outer bi-Lipschitz equivalent. Points inside shaded disks
represent arcs with the tangency order higher than the respective surface exponent.}\label{9}
\end{figure}
 
 Since weakly outer bi-Lipschitz maps preserve the number of nodal zones, we conclude that $\{X_i\}_{i=1}^3$ and $\{X_i'\}_{i=1}^3$ are not weakly outer bi-Lipschitz equivalent, although $X$ and $X'$ are weakly outer bi-Lipschitz equivalent (in fact, $X$ and $X'$ are the same surface).
\end{Exam}

\begin{Exam}\label{Exam2}
 Consider the circular $\beta$-snake $X$ of Example \ref{exemplo1-circular}. Considering the minimal sequence of $X$ with respect to $(N_a,\varepsilon_a)$ we obtain the greedy decomposition $\{X_i\}_{i=1}^4$, where $X_1$, 
 $X_2$, $X_3$, and $X_4$ contain 3, 5, 3, and 4 nodal zones, respectively (see Figure \ref{7}a). Considering the minimal sequence of $X$ with respect to $(N_b,\varepsilon_b)$ we obtain the greedy decomposition $\{X'_i\}_{i=1}^4$, where $X'_1$, 
 $X'_2$, $X'_3$, and $X'_4$ contain 3, 4, 4, and 4 nodal zones, respectively (see Figure \ref{7}b). Finally, considering the minimal sequence of $X$ with respect to $(N_c,\varepsilon_c)$, we obtain the greedy decomposition $\{X''_i\}_{i=1}^4$, where $X''_1$, 
 $X''_2$, $X''_3$, and $X''_4$ contain 3, 3, 5, and 4 nodal zones, respectively (see Figure \ref{7}c). 
 
 Since weakly outer bi-Lipschitz maps preserve the number of nodal zones, we conclude that $\{X_i\}_{i=1}^4$, $\{X_i'\}_{i=1}^4$ are not weakly outer bi-Lipschitz equivalent, and that $\{X'_i\}_{i=1}^4$, $\{X_i''\}_{i=1}^4$ are not weakly outer bi-Lipschitz equivalent. Moreover, $\{X_i\}_{i=1}^4$ and $\{X_i''\}_{i=1}^4$ are pancake decompositions whose pancakes have the same quantity of nodal zones, such quantities are not cyclically equal, since in $\{X_i\}_{i=1}^4$, the two pancakes with 3 nodal zones are not adjacent, and in $\{X_i''\}_{i=1}^4$ they are adjacent. So, $\{X_i\}_{i=1}^4$ and $\{X_i''\}_{i=1}^4$ are not weakly outer bi-Lipschitz equivalent, because weakly outer bi-Lipschitz maps preserve the number of nodal zones and their cyclic order on circular snakes.
 \end{Exam}

\section{Some final remarks on the greedy decomposition}

    The greedy pancake decomposition both for snakes and circular snakes are minimal and canonical up to weakly bi-Lipschitz equivalence. This is due to their regular behavior outside nodal zones, that is, arcs on segments always have the same multiplicity. However, the greedy algorithm can fail to give a minimal pancake decomposition for H\"older triangles in general. 

    \begin{Exam}
     
    Consider the non-snake bubble (see Definition 4.45 and Example 4.51 in \cite{GabrielovSouza}) $X=T(\lambda_0,\lambda_2)$ whose link and orientation are given in Figure \ref{10}a. If we generalize Definitions \ref{Def of multiplicity} and \ref{constant zone} for H\"older triangle and try to apply the greedy algorithm for snakes in this case, we obtain $X_1=T(\lambda_0,\lambda_1)$ and $X_2=T(\lambda_1,\lambda_2)$ as the elements of the decomposition of $X$. However, $X_1$ clearly is not a pancake, since it is not LNE. The natural way to obtain a minimal pancake decomposition for bubles in general is to consider $\lambda_1$ as an abnormal arc.
    \end{Exam}

    \begin{figure}[h!]
 \centering

\tikzset{every picture/.style={line width=0.75pt}} 

\begin{tikzpicture}[x=0.75pt,y=0.75pt,yscale=-0.9,xscale=1]

\draw    (130.95,37.58) .. controls (131.07,72.25) and (163.74,70.8) .. (202.4,70.67) .. controls (241.07,70.53) and (257.05,64.48) .. (257.75,73.81) .. controls (258.45,83.14) and (242.43,79.86) .. (202.43,80) .. controls (162.43,80.14) and (131.1,78.92) .. (131.2,108.25) ;
\draw [shift={(160.14,68.99)}, rotate = 190.68] [fill={rgb, 255:red, 0; green, 0; blue, 0 }  ][line width=0.08]  [draw opacity=0] (7.14,-3.43) -- (0,0) -- (7.14,3.43) -- (4.74,0) -- cycle    ;
\draw [shift={(234.27,69.32)}, rotate = 176] [fill={rgb, 255:red, 0; green, 0; blue, 0 }  ][line width=0.08]  [draw opacity=0] (7.14,-3.43) -- (0,0) -- (7.14,3.43) -- (4.74,0) -- cycle    ;
\draw [shift={(228.9,80.3)}, rotate = 1.09] [fill={rgb, 255:red, 0; green, 0; blue, 0 }  ][line width=0.08]  [draw opacity=0] (7.14,-3.43) -- (0,0) -- (7.14,3.43) -- (4.74,0) -- cycle    ;
\draw [shift={(156.61,82.11)}, rotate = 350.69] [fill={rgb, 255:red, 0; green, 0; blue, 0 }  ][line width=0.08]  [draw opacity=0] (7.14,-3.43) -- (0,0) -- (7.14,3.43) -- (4.74,0) -- cycle    ;
\draw  [draw opacity=0][fill={rgb, 255:red, 155; green, 155; blue, 155 }  ,fill opacity=0.47 ] (257.68,53.74) .. controls (268.86,53.71) and (277.96,62.66) .. (278,73.74) .. controls (278.04,84.82) and (269,93.83) .. (257.82,93.87) .. controls (246.63,93.91) and (237.53,84.96) .. (237.49,73.88) .. controls (237.45,62.8) and (246.49,53.78) .. (257.68,53.74) -- cycle ;
\draw  [draw opacity=0][fill={rgb, 255:red, 155; green, 155; blue, 155 }  ,fill opacity=0.47 ] (158.45,57.35) .. controls (169.64,57.31) and (178.74,66.26) .. (178.78,77.35) .. controls (178.82,88.43) and (169.78,97.44) .. (158.59,97.48) .. controls (147.41,97.52) and (138.31,88.57) .. (138.27,77.49) .. controls (138.23,66.41) and (147.27,57.39) .. (158.45,57.35) -- cycle ;
\draw  [fill={rgb, 255:red, 0; green, 0; blue, 0 }  ,fill opacity=1 ] (130.16,37.02) .. controls (131.45,37.02) and (132.5,38.02) .. (132.5,39.26) .. controls (132.51,40.49) and (131.46,41.5) .. (130.17,41.51) .. controls (128.88,41.51) and (127.83,40.51) .. (127.83,39.27) .. controls (127.82,38.04) and (128.87,37.03) .. (130.16,37.02) -- cycle ;
\draw  [fill={rgb, 255:red, 0; green, 0; blue, 0 }  ,fill opacity=1 ] (130.81,105.48) .. controls (132.1,105.47) and (133.15,106.47) .. (133.15,107.71) .. controls (133.16,108.95) and (132.11,109.96) .. (130.82,109.96) .. controls (129.53,109.96) and (128.48,108.96) .. (128.48,107.73) .. controls (128.47,106.49) and (129.52,105.48) .. (130.81,105.48) -- cycle ;
\draw  [fill={rgb, 255:red, 0; green, 0; blue, 0 }  ,fill opacity=1 ] (221.55,79.16) .. controls (222.84,79.16) and (223.89,80.16) .. (223.89,81.39) .. controls (223.9,82.63) and (222.86,83.64) .. (221.56,83.64) .. controls (220.27,83.65) and (219.22,82.65) .. (219.22,81.41) .. controls (219.21,80.17) and (220.26,79.17) .. (221.55,79.16) -- cycle ;

\draw    (400.93,127.23) .. controls (418.27,127.18) and (434,128.37) .. (434,124.37) .. controls (434,120.37) and (433.52,105.8) .. (433.47,86.47) .. controls (433.41,67.13) and (433.27,47.89) .. (438.8,47.8) .. controls (444.33,47.71) and (443.83,54.46) .. (444.13,86.47) .. controls (444.43,118.47) and (444.58,121.17) .. (444.24,125.17) ;
\draw [shift={(420.68,127.35)}, rotate = 179.83] [fill={rgb, 255:red, 0; green, 0; blue, 0 }  ][line width=0.08]  [draw opacity=0] (7.14,-3.43) -- (0,0) -- (7.14,3.43) -- (4.74,0) -- cycle    ;
\draw [shift={(433.61,102.66)}, rotate = 89.12] [fill={rgb, 255:red, 0; green, 0; blue, 0 }  ][line width=0.08]  [draw opacity=0] (7.14,-3.43) -- (0,0) -- (7.14,3.43) -- (4.74,0) -- cycle    ;
\draw [shift={(433.78,63.57)}, rotate = 92.64] [fill={rgb, 255:red, 0; green, 0; blue, 0 }  ][line width=0.08]  [draw opacity=0] (7.14,-3.43) -- (0,0) -- (7.14,3.43) -- (4.74,0) -- cycle    ;
\draw [shift={(443.98,68.81)}, rotate = 269.12] [fill={rgb, 255:red, 0; green, 0; blue, 0 }  ][line width=0.08]  [draw opacity=0] (7.14,-3.43) -- (0,0) -- (7.14,3.43) -- (4.74,0) -- cycle    ;
\draw [shift={(444.35,108.44)}, rotate = 269.41] [fill={rgb, 255:red, 0; green, 0; blue, 0 }  ][line width=0.08]  [draw opacity=0] (7.14,-3.43) -- (0,0) -- (7.14,3.43) -- (4.74,0) -- cycle    ;
\draw  [draw opacity=0][fill={rgb, 255:red, 155; green, 155; blue, 155 }  ,fill opacity=0.47 ] (355.05,122.08) .. controls (355.02,110.89) and (363.98,101.8) .. (375.06,101.77) .. controls (386.14,101.74) and (395.15,110.78) .. (395.18,121.97) .. controls (395.21,133.15) and (386.26,142.25) .. (375.18,142.28) .. controls (364.09,142.31) and (355.09,133.27) .. (355.05,122.08) -- cycle ;
\draw  [fill={rgb, 255:red, 0; green, 0; blue, 0 }  ,fill opacity=1 ] (362.77,125.14) .. controls (362.77,123.85) and (363.77,122.8) .. (365.01,122.8) .. controls (366.24,122.79) and (367.25,123.84) .. (367.25,125.13) .. controls (367.26,126.42) and (366.26,127.47) .. (365.02,127.47) .. controls (363.78,127.47) and (362.78,126.43) .. (362.77,125.14) -- cycle ;
\draw  [fill={rgb, 255:red, 0; green, 0; blue, 0 }  ,fill opacity=1 ] (442.09,124.19) .. controls (442.09,122.9) and (443.09,121.85) .. (444.33,121.85) .. controls (445.56,121.84) and (446.57,122.89) .. (446.57,124.18) .. controls (446.58,125.47) and (445.58,126.52) .. (444.34,126.52) .. controls (443.1,126.53) and (442.1,125.48) .. (442.09,124.19) -- cycle ;
\draw    (365.01,125.13) .. controls (364.01,116.47) and (370.33,114.43) .. (367,104.1) .. controls (363.67,93.77) and (318.99,67.31) .. (318,51.1) .. controls (317.01,34.89) and (340.33,21.43) .. (372,21.43) .. controls (403.67,21.43) and (439.33,29.43) .. (428.67,54.43) .. controls (418,79.43) and (397.67,82.03) .. (382,94.03) .. controls (366.33,106.03) and (383.33,126.37) .. (400.93,127.23) ;
\draw [shift={(367.76,112.16)}, rotate = 105.46] [fill={rgb, 255:red, 0; green, 0; blue, 0 }  ][line width=0.08]  [draw opacity=0] (7.14,-3.43) -- (0,0) -- (7.14,3.43) -- (4.74,0) -- cycle    ;
\draw [shift={(338.47,77.11)}, rotate = 40.67] [fill={rgb, 255:red, 0; green, 0; blue, 0 }  ][line width=0.08]  [draw opacity=0] (7.14,-3.43) -- (0,0) -- (7.14,3.43) -- (4.74,0) -- cycle    ;
\draw [shift={(341.44,26.29)}, rotate = 159.01] [fill={rgb, 255:red, 0; green, 0; blue, 0 }  ][line width=0.08]  [draw opacity=0] (7.14,-3.43) -- (0,0) -- (7.14,3.43) -- (4.74,0) -- cycle    ;
\draw [shift={(412.97,27.25)}, rotate = 198.87] [fill={rgb, 255:red, 0; green, 0; blue, 0 }  ][line width=0.08]  [draw opacity=0] (7.14,-3.43) -- (0,0) -- (7.14,3.43) -- (4.74,0) -- cycle    ;
\draw [shift={(406.45,79.58)}, rotate = 327.37] [fill={rgb, 255:red, 0; green, 0; blue, 0 }  ][line width=0.08]  [draw opacity=0] (7.14,-3.43) -- (0,0) -- (7.14,3.43) -- (4.74,0) -- cycle    ;
\draw [shift={(381.45,117.9)}, rotate = 230.37] [fill={rgb, 255:red, 0; green, 0; blue, 0 }  ][line width=0.08]  [draw opacity=0] (7.14,-3.43) -- (0,0) -- (7.14,3.43) -- (4.74,0) -- cycle    ;
\draw  [fill={rgb, 255:red, 0; green, 0; blue, 0 }  ,fill opacity=1 ] (394.03,86.09) .. controls (394.02,84.8) and (395.02,83.75) .. (396.26,83.75) .. controls (397.5,83.75) and (398.51,84.79) .. (398.51,86.08) .. controls (398.51,87.37) and (397.51,88.42) .. (396.27,88.43) .. controls (395.04,88.43) and (394.03,87.39) .. (394.03,86.09) -- cycle ;
\draw  [draw opacity=0][fill={rgb, 255:red, 155; green, 155; blue, 155 }  ,fill opacity=0.47 ] (413.72,49.19) .. controls (413.69,38) and (422.65,28.91) .. (433.73,28.88) .. controls (444.81,28.85) and (453.82,37.89) .. (453.85,49.08) .. controls (453.88,60.27) and (444.92,69.36) .. (433.84,69.39) .. controls (422.76,69.42) and (413.75,60.38) .. (413.72,49.19) -- cycle ;
\draw  [fill={rgb, 255:red, 0; green, 0; blue, 0 }  ,fill opacity=1 ] (386.09,123.67) .. controls (386.09,122.38) and (387.09,121.33) .. (388.33,121.33) .. controls (389.56,121.33) and (390.57,122.37) .. (390.57,123.66) .. controls (390.58,124.95) and (389.58,126) .. (388.34,126.01) .. controls (387.1,126.01) and (386.1,124.97) .. (386.09,123.67) -- cycle ;
\draw  [fill={rgb, 255:red, 0; green, 0; blue, 0 }  ,fill opacity=1 ] (431.76,124.37) .. controls (431.76,123.08) and (432.76,122.03) .. (433.99,122.03) .. controls (435.23,122.03) and (436.24,123.07) .. (436.24,124.36) .. controls (436.24,125.65) and (435.24,126.7) .. (434.01,126.7) .. controls (432.77,126.71) and (431.76,125.66) .. (431.76,124.37) -- cycle ;
\draw  [fill={rgb, 255:red, 0; green, 0; blue, 0 }  ,fill opacity=1 ] (436.56,47.81) .. controls (436.56,46.51) and (437.56,45.47) .. (438.79,45.46) .. controls (440.03,45.46) and (441.04,46.5) .. (441.04,47.79) .. controls (441.04,49.09) and (440.04,50.13) .. (438.81,50.14) .. controls (437.57,50.14) and (436.56,49.1) .. (436.56,47.81) -- cycle ;
\draw    (583.6,126.57) .. controls (600.93,126.52) and (616.67,127.7) .. (616.67,123.7) .. controls (616.67,119.7) and (616.19,105.13) .. (616.13,85.8) .. controls (616.08,66.47) and (615.93,47.22) .. (621.47,47.13) .. controls (627,47.04) and (626.5,53.79) .. (626.8,85.8) .. controls (627.1,117.81) and (627.24,120.5) .. (626.91,124.5) ;
\draw [shift={(596.78,126.67)}, rotate = 0.43] [fill={rgb, 255:red, 0; green, 0; blue, 0 }  ][line width=0.08]  [draw opacity=0] (7.14,-3.43) -- (0,0) -- (7.14,3.43) -- (4.74,0) -- cycle    ;
\draw [shift={(616.39,108.68)}, rotate = 268.99] [fill={rgb, 255:red, 0; green, 0; blue, 0 }  ][line width=0.08]  [draw opacity=0] (7.14,-3.43) -- (0,0) -- (7.14,3.43) -- (4.74,0) -- cycle    ;
\draw [shift={(616.2,69.47)}, rotate = 271.55] [fill={rgb, 255:red, 0; green, 0; blue, 0 }  ][line width=0.08]  [draw opacity=0] (7.14,-3.43) -- (0,0) -- (7.14,3.43) -- (4.74,0) -- cycle    ;
\draw [shift={(626.5,61.19)}, rotate = 88.54] [fill={rgb, 255:red, 0; green, 0; blue, 0 }  ][line width=0.08]  [draw opacity=0] (7.14,-3.43) -- (0,0) -- (7.14,3.43) -- (4.74,0) -- cycle    ;
\draw [shift={(626.95,101.08)}, rotate = 89.4] [fill={rgb, 255:red, 0; green, 0; blue, 0 }  ][line width=0.08]  [draw opacity=0] (7.14,-3.43) -- (0,0) -- (7.14,3.43) -- (4.74,0) -- cycle    ;
\draw  [draw opacity=0][fill={rgb, 255:red, 155; green, 155; blue, 155 }  ,fill opacity=0.47 ] (537.72,121.41) .. controls (537.69,110.23) and (546.65,101.13) .. (557.73,101.1) .. controls (568.81,101.07) and (577.82,110.11) .. (577.85,121.3) .. controls (577.88,132.49) and (568.92,141.58) .. (557.84,141.61) .. controls (546.76,141.64) and (537.75,132.6) .. (537.72,121.41) -- cycle ;
\draw  [fill={rgb, 255:red, 0; green, 0; blue, 0 }  ,fill opacity=1 ] (545.44,124.47) .. controls (545.44,123.18) and (546.44,122.13) .. (547.67,122.13) .. controls (548.91,122.13) and (549.92,123.17) .. (549.92,124.46) .. controls (549.92,125.75) and (548.92,126.8) .. (547.69,126.8) .. controls (546.45,126.81) and (545.44,125.76) .. (545.44,124.47) -- cycle ;
\draw  [fill={rgb, 255:red, 0; green, 0; blue, 0 }  ,fill opacity=1 ] (624.76,123.52) .. controls (624.76,122.23) and (625.76,121.18) .. (626.99,121.18) .. controls (628.23,121.18) and (629.24,122.22) .. (629.24,123.51) .. controls (629.24,124.8) and (628.24,125.85) .. (627.01,125.86) .. controls (625.77,125.86) and (624.76,124.82) .. (624.76,123.52) -- cycle ;
\draw    (547.68,124.47) .. controls (546.68,115.8) and (553,113.77) .. (549.67,103.43) .. controls (546.33,93.1) and (501.65,66.64) .. (500.67,50.43) .. controls (499.68,34.22) and (523,20.77) .. (554.67,20.77) .. controls (586.33,20.77) and (622,28.77) .. (611.33,53.77) .. controls (600.67,78.77) and (580.33,81.37) .. (564.67,93.37) .. controls (549,105.37) and (566,125.7) .. (583.6,126.57) ;
\draw [shift={(548.38,117.78)}, rotate = 289.93] [fill={rgb, 255:red, 0; green, 0; blue, 0 }  ][line width=0.08]  [draw opacity=0] (7.14,-3.43) -- (0,0) -- (7.14,3.43) -- (4.74,0) -- cycle    ;
\draw [shift={(526.4,80.92)}, rotate = 220.18] [fill={rgb, 255:red, 0; green, 0; blue, 0 }  ][line width=0.08]  [draw opacity=0] (7.14,-3.43) -- (0,0) -- (7.14,3.43) -- (4.74,0) -- cycle    ;
\draw [shift={(517.82,28.27)}, rotate = 335.47] [fill={rgb, 255:red, 0; green, 0; blue, 0 }  ][line width=0.08]  [draw opacity=0] (7.14,-3.43) -- (0,0) -- (7.14,3.43) -- (4.74,0) -- cycle    ;
\draw [shift={(589.32,24.62)}, rotate = 15.83] [fill={rgb, 255:red, 0; green, 0; blue, 0 }  ][line width=0.08]  [draw opacity=0] (7.14,-3.43) -- (0,0) -- (7.14,3.43) -- (4.74,0) -- cycle    ;
\draw [shift={(594.81,75.05)}, rotate = 144.29] [fill={rgb, 255:red, 0; green, 0; blue, 0 }  ][line width=0.08]  [draw opacity=0] (7.14,-3.43) -- (0,0) -- (7.14,3.43) -- (4.74,0) -- cycle    ;
\draw [shift={(560.33,111.89)}, rotate = 59.22] [fill={rgb, 255:red, 0; green, 0; blue, 0 }  ][line width=0.08]  [draw opacity=0] (7.14,-3.43) -- (0,0) -- (7.14,3.43) -- (4.74,0) -- cycle    ;
\draw  [draw opacity=0][fill={rgb, 255:red, 155; green, 155; blue, 155 }  ,fill opacity=0.47 ] (596.39,48.52) .. controls (596.36,37.34) and (605.31,28.24) .. (616.4,28.21) .. controls (627.48,28.18) and (636.49,37.23) .. (636.52,48.41) .. controls (636.55,59.6) and (627.59,68.69) .. (616.51,68.72) .. controls (605.43,68.75) and (596.42,59.71) .. (596.39,48.52) -- cycle ;
\draw  [fill={rgb, 255:red, 0; green, 0; blue, 0 }  ,fill opacity=1 ] (568.76,123.01) .. controls (568.76,121.72) and (569.76,120.67) .. (570.99,120.66) .. controls (572.23,120.66) and (573.24,121.7) .. (573.24,123) .. controls (573.24,124.29) and (572.24,125.34) .. (571.01,125.34) .. controls (569.77,125.34) and (568.76,124.3) .. (568.76,123.01) -- cycle ;
\draw  [fill={rgb, 255:red, 0; green, 0; blue, 0 }  ,fill opacity=1 ] (614.43,123.71) .. controls (614.42,122.41) and (615.42,121.37) .. (616.66,121.36) .. controls (617.9,121.36) and (618.9,122.4) .. (618.91,123.69) .. controls (618.91,124.99) and (617.91,126.03) .. (616.67,126.04) .. controls (615.44,126.04) and (614.43,125) .. (614.43,123.71) -- cycle ;
\draw  [fill={rgb, 255:red, 0; green, 0; blue, 0 }  ,fill opacity=1 ] (619.23,47.14) .. controls (619.22,45.85) and (620.22,44.8) .. (621.46,44.8) .. controls (622.7,44.79) and (623.7,45.84) .. (623.71,47.13) .. controls (623.71,48.42) and (622.71,49.47) .. (621.47,49.47) .. controls (620.24,49.47) and (619.23,48.43) .. (619.23,47.14) -- cycle ;
\draw  [fill={rgb, 255:red, 0; green, 0; blue, 0 }  ,fill opacity=1 ] (536.03,91.69) .. controls (536.02,90.4) and (537.02,89.35) .. (538.26,89.35) .. controls (539.5,89.35) and (540.51,90.39) .. (540.51,91.68) .. controls (540.51,92.97) and (539.51,94.02) .. (538.27,94.03) .. controls (537.04,94.03) and (536.03,92.99) .. (536.03,91.69) -- cycle ;
\draw  [draw opacity=0][fill={rgb, 255:red, 155; green, 155; blue, 155 }  ,fill opacity=0.47 ] (418.65,124.46) .. controls (418.62,113.27) and (427.58,104.18) .. (438.66,104.15) .. controls (449.74,104.12) and (458.75,113.16) .. (458.78,124.35) .. controls (458.81,135.53) and (449.86,144.63) .. (438.78,144.66) .. controls (427.69,144.69) and (418.69,135.64) .. (418.65,124.46) -- cycle ;
\draw  [draw opacity=0][fill={rgb, 255:red, 155; green, 155; blue, 155 }  ,fill opacity=0.47 ] (601.45,123.66) .. controls (601.42,112.47) and (610.38,103.38) .. (621.46,103.35) .. controls (632.54,103.32) and (641.55,112.36) .. (641.58,123.55) .. controls (641.61,134.73) and (632.66,143.83) .. (621.58,143.86) .. controls (610.49,143.89) and (601.49,134.84) .. (601.45,123.66) -- cycle ;

\draw (74.33,21.07) node [anchor=north west][inner sep=0.75pt]    {$a)$};
\draw (104.92,29.73) node [anchor=north west][inner sep=0.75pt]    {$\lambda _{0}$};
\draw (104.22,93.23) node [anchor=north west][inner sep=0.75pt]    {$\lambda _{2}$};
\draw (210.89,88.4) node [anchor=north west][inner sep=0.75pt]    {$\lambda _{1}$};
\draw (292.67,20.4) node [anchor=north west][inner sep=0.75pt]    {$b)$};
\draw (346.05,123.2) node [anchor=north west][inner sep=0.75pt]    {$\lambda _{0}$};
\draw (382.02,126.22) node [anchor=north west][inner sep=0.75pt]    {$\lambda _{2}$};
\draw (423.52,127.73) node [anchor=north west][inner sep=0.75pt]    {$\lambda _{3}$};
\draw (444.09,127.59) node [anchor=north west][inner sep=0.75pt]    {$\lambda _{5}$};
\draw (441.77,29.48) node [anchor=north west][inner sep=0.75pt]    {$\lambda _{4}$};
\draw (475.33,19.73) node [anchor=north west][inner sep=0.75pt]    {$c)$};
\draw (528.72,122.53) node [anchor=north west][inner sep=0.75pt]    {$\lambda _{5}$};
\draw (564.69,125.55) node [anchor=north west][inner sep=0.75pt]    {$\lambda _{3}$};
\draw (606.19,127.07) node [anchor=north west][inner sep=0.75pt]    {$\lambda _{2}$};
\draw (626.76,126.92) node [anchor=north west][inner sep=0.75pt]    {$\lambda _{0}$};
\draw (624.44,28.82) node [anchor=north west][inner sep=0.75pt]    {$\lambda _{1}$};
\draw (375.77,67.88) node [anchor=north west][inner sep=0.75pt]    {$\lambda _{1}$};
\draw (534.17,70.68) node [anchor=north west][inner sep=0.75pt]    {$\lambda _{4}$};

\end{tikzpicture}

\caption{Surfaces where the greedy algorithm fails to obtain a minimal pancake decomposition. Points inside shaded disks
represent arcs with the tangency order higher than the respective surface's exponent.}\label{10}
\end{figure}
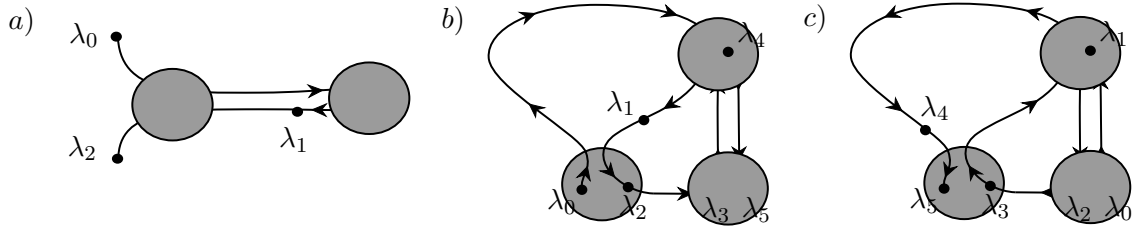

    A conjecture that naturally arise in this context of determining a minimal pancake decomposition for H\"older triangles in general is the following. Given a H\"older triangle $T$, dividing it into snakes and non-snake bubbles, considering the minimal pancake decomposition in each one of those parts and then joining adjacent pancakes which union are LNE into a new pancake will provide a minimal decomposition for $T$. Unfortunately, this is not true, as shown in Example \ref{Exam: greedy algorithm failing for HT}. This explicit that the ones intending to obtain a minimal or canonical pancake decomposition for surfaces in general will need to develop more sophisticated tools, since the greedy algorithm proved itself to be insufficient.

    \begin{Exam}\label{Exam: greedy algorithm failing for HT}
    Let $X=T(\lambda_0,\lambda_5)$ be a H\"older triangle with link as Figures \ref{10}b and \ref{10}c, where two distinct orientations are considered. In Figure \ref{10}b, it consists of a bubble snake $T(\lambda_0,\lambda_2)$, a LNE H\"older triangle $T(\lambda_2,\lambda_3)$ and a non-snake bubble $T(\lambda_3,\lambda_5)$ (the same holds for Figure \ref{10}c but with a reversed order). The arcs $\lambda_1$ and $\lambda_4$ were chosen so that, in Figure \ref{10}b, $\{X_1,X_2\}$ is a minimal pancake decomposition of the snake bubble, $\{X_3\}$ is the minimal decomposition of the LNE H\"older triangle and $\{X_4,X_5\}$ is the minimal decomposition of the non-snake bubble. A similar decomposition was considered in Figure \ref{10}c (here, $X_i=T(\lambda_{i-1},\lambda_i)$, for $i=1,\dots,5$). 
    
    In Figure \ref{10}b, $\{X_1, X_2\cup X_3\cup X_4,X_5\}$ is a minimal pancake decomposition of $X$, showing that, in this case, by joining adjacent pancakes which union were LNE into a new pancake we obtained a minimal pancake decomposition for $X$. However, in Figure \ref{10}c, the pancakes $X_1, X_2, X_3, X_4,X_5$ will not produce a minimal pancake decomposition for $X$, since  $X_2\cup X_3\cup X_4$ is not LNE.
    
    Notice that in both cases we applied the greedy algorithm for snakes to obtain the arcs $\lambda_i$, but depending on the orientation considered the conjecture was false, as seen in Figure \ref{10}c. Indeed, if we consider a H\"older triangle which link is the gluing of the link of Figure \ref{10}b at $\lambda_5$ with the link of Figure \ref{10}c at $\lambda_0$ through a LNE H\"older triangle connecting those arcs, then the conjecture will be false independently of the chosen orientation.   
    
    \end{Exam}

\begin{Exam}\label{Exam: not outer equiv pancake decomp}
    It is possible to consider the notion of outer equivalence of pancake decompositions exchanging the weakly outer bi-Lipschitz homeomorphism in Definition \ref{Def: equiv of pancake decomp} by an outer bi-Lipschitz one. Unfortunately, both the two natural candidates of pancake decompositions fail to be canonical with respect to this outer equivalence, namely, the pancake decomposition presented in Proposition 4.56 of \cite{GabrielovSouza} for snakes (as in Corollary 3.37 of \cite{circular-snakes} for circular snakes with nodal zones) and the greedy pancake decomposition in this paper. For example, let us consider two surface germs $X$ and $Y$ with links as in Figure \ref{11}a and \ref{11}b, respectively. 

    \begin{figure}[h!]
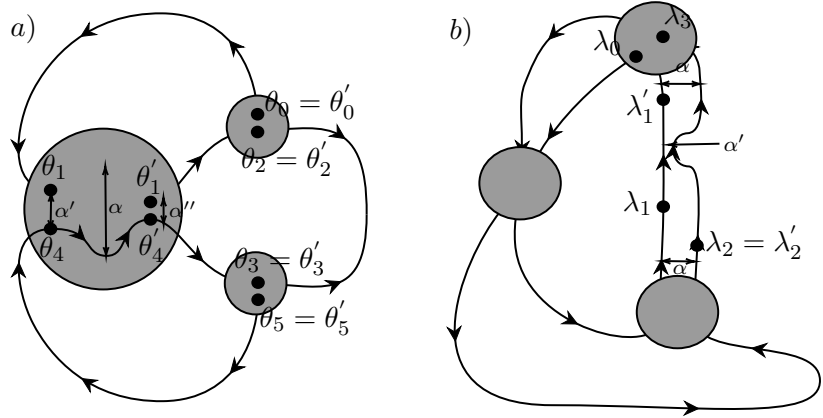

\centering

\tikzset{every picture/.style={line width=0.75pt}} 



    \caption{Surfaces where the greedy algorithm fails to obtain a minimal pancake decomposition. Points inside shaded disks
represent arcs with the tangency order higher than the respective surface's exponent.}\label{11}
\end{figure}
    
    Let $\{X_i=T(\theta_{i-1},\theta_i)\}_{i=1}^5$ and $\{X'_i=T(\theta'_{i-1},\theta'_i)\}_{i=1}^5$ be pancake decompositions of $X$ with boundary arcs represented by the dots on the link of $X$. If $\alpha''>\alpha'>\alpha>\mu(X)$ are such that $$\alpha''=\operatorname{tord}(X_1',X_4') = \operatorname{tord}(\theta'_1,\theta'_4)>\operatorname{tord}(\theta_1,\theta_4)=\operatorname{tord}(X_1,X_4)=\alpha',$$ 
    then those pancake decompositions are not outer equivalent, since outer bi-Lipschitz homeomorphisms preserve tangency orders. Analogously, let $\{Y_i=T(\lambda_{i-1},\lambda_i)\}_{i=1}^3$ and $\{Y'_i=T(\lambda'_{i-1},\lambda'_i)\}_{i=1}^3$ be greedy pancake decompositions of $Y$ with boundary arcs represented by the dots on the link of $Y$. If $\alpha'>\alpha>\mu(X)$ are such that $$\alpha'=\operatorname{tord}(Y_1',Y_3') = \operatorname{tord}(\lambda_1',Y_3')>\operatorname{tord}(\lambda_1,Y_3)=\operatorname{tord}(Y_1,Y_3)=\alpha,$$ then those pancake decompositions are not outer equivalent.
    \end{Exam}

\end{document}